\documentclass[12pt,a4paper]{article}

\usepackage[utf8]{inputenc}
\usepackage{amsfonts}
\usepackage{amsmath,amsthm}
\usepackage{amssymb}
\usepackage{bm}
\usepackage{upgreek}
\usepackage{xcolor}  
\usepackage[aligntableaux=center]{ytableau}
\usepackage[left=2.5cm,right=2.55cm]{geometry}
\usepackage{hyperref}

\usepackage{biblatex}
\hypersetup{unicode}
\hypersetup{breaklinks=true}
\hypersetup{colorlinks=true,linkcolor=blue,citecolor=red,urlcolor=violet,
bookmarksdepth=subsection,final}

\def\dd{\mathrm{d}}
\def\pd{\partial}
\def\eps{\varepsilon}
\def\fii{\varphi}
\renewcommand\pounds{\mathsterling}
\def\Lie{\pounds}
\def\ii{\mathrm{i}}
\def\man{\mathcal{M}}
\def\Nat{\mathbb{N}}
\def\TT{\mathcal{T}}
\def\RR{\mathcal{R}}

\def\youngboxsize{13pt}
\def\youngfontsize{9pt}
\def\youngtinyboxsize{6pt}
\def\youngtinyfontsize{6pt}
\def\youngbig(#1){{\ytableausetup{boxsize=\youngboxsize}\fontsize{\youngfontsize}{\youngfontsize}\selectfont\ytableaushort{#1}}}
\def\young(#1){{\ytableausetup{boxsize=\youngtinyboxsize}\fontsize{\youngtinyfontsize}{\youngtinyfontsize}\selectfont\ytableaushort{#1}}}
\def\yngbig(#1){{\ytableausetup{boxsize=\youngboxsize}\ydiagram{#1}}}
\def\yng(#1){{\ytableausetup{boxsize=\youngtinyboxsize}\ydiagram{#1}}}

\newtheorem{theorem}{Theorem}
\newtheorem{definition}{Definition}
\newtheorem{lemma}{Lemma}
\newtheorem{proposition}{Proposition}
\theoremstyle{remark}
\newtheorem{remark}{Remark}
\theoremstyle{plain}

\makeatletter
\renewcommand{\qed}{ $\blacksquare$ }
\renewenvironment{proof}[1][\proofname]{\par
  \pushQED{\qed}%
  \normalfont \topsep6\p@\@plus6\p@\relax
  \footnotesize%
  \trivlist
  \item[\hskip\labelsep
        \normalfont\bfseries
    #1\@addpunct{.}]\ignorespaces
}{%
  \hfill\popQED\endtrivlist\@endpefalse
}
\makeatother

\begin{document}

\title{Structure of the Riemann tensor in higher-dimensional Kerr-NUT-(A)dS spaces}

\author{David Matejov$^{*\dagger}$
       \quad and \quad
        Igor Khavkine$^{**\dagger\dagger}$ 	\\[4mm]
	{\small $^*$Institute of Theoretical Physics, Faculty of Mathematics and Physics,} \\
	{\small Charles University, V~Hole\v{s}ovi\v{c}k\'ach~2, 180~00 Prague 8, Czech Republic
	} \\[3mm]
	{\small $^{**}$Institute of Mathematics, Czech Academy of Sciences,} \\
	{\small \v{Z}itn{\'a} 25, 115 67 Praha 1, Czech Republic} \\[3mm]
    {\small  E-mail: $^\dagger$\texttt{d.matejov@gmail.com, $^{\dagger\dagger}$khavkine@math.cas.cz}}\\}

\maketitle

\begin{abstract}
We study the algebraic and differential structure of the Riemann
curvature tensor of the higher-dimensional Kerr-NUT-(A)dS spaces,
motivated by the eventual goal of finding an IDEAL (Intrinsic,
Deductive, Explicit and ALgorithmic) characterization of this family of
metrics. The special geometry of these spaces is governed by the
principal tensor $h_{ab}$, a non-degenerate closed conformal
Killing-Yano $2$-form, whose existence singles out the Kerr-NUT-(A)dS
family; we therefore focus on the algebraic relationship between the
Riemann curvature $R_{abcd}$ and $h_{ab}$. In our investigations we
encountered an obstruction, which results in a no-go theorem: no
non-trivial $2$-form, including $h_{ab}$ itself, can be constructed
covariantly from the undifferentiated Riemann tensor alone. Despite
that, we do characterize the family of Riemann-symmetric tensors that
could be the curvature of a Kerr-NUT-(A)dS metric by an algebraic and a
differential condition, stemming from the integrability of the defining
equation of $h_{ab}$ and the second Bianchi identity. Our calculations
apply in all higher dimensions, which becomes feasible by a judicious
application of representation theoretic techniques related to a
semi-direct product group $U(1)^n\rtimes S_n$ that stabilizes $h_{ab}$.
\end{abstract}

\vfil\noindent
PACS class:  04.20.Jb, 04.70.Bw, 04.50.Gh, 02.40.-k



\section{Introduction}

The higher-dimensional geometries known as \emph{Kerr-NUT-(A)dS}
spaces are, as the name suggests, a generalization of the
corresponding well-known $4$-dimensional family of exact solutions of
the (cosmological) vacuum Einstein equations~\cite{ChenLuPope2006}. It is most natural to
introduce this family \emph{off-shell}, as a highly-symmetric metric
ansatz~\eqref{eq:HigherDimensionalKNAdSmetric} containing several free
functions. Imposing the vacuum Einstein equations
(with $\Lambda$) then fixes these functions to a specific polynomial
form and puts the metric \emph{on-shell}. The most symmetric form of the
metric is in Euclidean signature, but could be analytically continued
to any signature, including the physical Lorentzian signature. Both on-
and off-shell, the family exhibits remarkably rich geometric structure:
a whole tower of Killing and Killing-Yano tensors, and the separability
and complete integrability of the geodesic equation and of the linear
test field equations (scalar, spinor, electromagnetic). It also contains
all the well-known higher-dimensional black hole solutions (the spherically
symmetric Schwarzschild-Tangherlini and the rotating Myers-Perry
metrics~\cite{MyersPerry1986}, with or without a cosmological
constant~\cite{GibbonsLuPagePope2005}) as special cases,
which is one reason these geometries also arise in string-theory and
supergravity contexts. All of this special structure is tied to the
existence of a single object, the \emph{principal tensor} $h_{ab}$, a
non-degenerate closed conformal Killing-Yano $2$-form. The
Kerr-NUT-(A)dS family and its properties are surveyed in detail in the
review~\cite{FKK}.

In this work we study the algebraic and differential structures of the
Riemann curvature tensor $R_{abcd}$ of Kerr-NUT-(A)dS spaces, motivated
by the possibility that they would be useful to find a so-called IDEAL characterization of this family of metrics.

\paragraph{IDEAL characterization: background and motivation.}
It is already known that the curvature tensor of Kerr-NUT-(A)dS metrics is of \emph{algebraic
type~D} with respect to an aligned null frame~\cite{MTCh,
OrtaggioPravdaPravdova}, in the sense of the higher-dimensional version
of the Petrov classification~\cite{OrtaggioPravdaPravdova}. We approach
the subject, however, from a different angle. Our long-term motivation
is to take some steps towards an \emph{IDEAL%
	\footnote{IDEAL stands for Intrinsic, Deductive, Explicit and ALgorithmic,
	see~\cite{Ferrando2010} (footnote, p.2) for the source of the
	acronym.} %
characterization} of this family. Such a characterization consists of a
list $\{T_\lambda[g]\}_\lambda$ of tensors constructed covariantly and
concomitantly from the metric $g_{ab}$ and the curvature $R_{abcd}$
(and its covariant derivatives), with the defining property that all of
them vanish, $T_\lambda[g] = 0$, on a neighborhood of a point if and
only if that neighborhood is locally isometric to a reference geometry
(a fixed metric or a whole family thereof). Crucially, the conditions
$T_\lambda[g]=0$ do not refer to any auxiliary structure such as a
preferred frame. In particular, the higher-dimensional Petrov type,
which is defined relative to an externally specified aligned null frame,
cannot by itself serve as an IDEAL condition.

There are several approaches to intrinsic characterization of metric
geometries. The dominant one relies on the Cartan-Karlhede method,
which builds a characterization by progressively adapting a null frame
to the geometry~\cite[Ch.9]{Stephani}; its main strength is that it is
systematic. The IDEAL approach, by contrast, does not yet rest on a
general theory and successful characterizations have so far relied on
ad-hoc constructions. Its advantages nevertheless make it worth
pursuing. The most obvious is the conceptual simplicity of \emph{testing}
whether a given metric---perhaps written in unfamiliar
coordinates---is locally isometric to a reference geometry: one simply
evaluates the tensors $T_\lambda[g]$ and checks that they all vanish.
In numerical relativity these equalities can be satisfied only
approximately, $T_\lambda[g]\approx 0$, or in a limit, $T_\lambda[g]\to
0$, so the tensors can be used to quantify the closeness or convergence
of a dynamical spacetime to a reference geometry, an idea explored
in~\cite{okounkova2018, okounkova2020}. A second advantage arises in
linearized gravity: by the Stewart-Walker lemma~\cite[Lem.2.2]{StewartWalker},
the vanishing $T_\lambda[g]=0$ on a background $g_{ab}$ implies that the
linearization $\dot{T}_\lambda[h]$ is invariant under linearized
diffeomorphisms, so the $\dot{T}_\lambda[h]$ furnish local gauge
invariant observables and their joint kernel is a candidate for a local
characterization of the pure gauge modes $h_{ab}=\Lie_v g_{ab}$; this
has been explored for various cosmological and black hole
backgrounds~\cite{Canepa, FrobHackHiguchi, Frob2018, Khavkine2019Compat,
AksteinerBackdahl}.

\paragraph{Approach: from 4D to higher dimensions.}
In four dimensions, the Schwarz\-schild and Kerr black holes already have
known IDEAL characterizations~\cite{FerrandoSaez2009}
(cf.~\cite[Sec.III]{GarciaParrado2015} and
\cite[Sec.1.5.1]{MatejovThesis} for simplified presentations). Their key
step is the classical type~D condition, expressed frame-independently
through the vanishing of a characteristic polynomial of the self-dual
Weyl tensor acting as an endomorphism on $2$-forms, together with the
(cosmological) vacuum condition. These already force the metric,
locally, into the finite-parameter Pleba\'{n}ski-Demia\'{n}ski family, and
the characterization is then completed by the use of specific curvature scalars that
fix or eliminate the remaining parameters. Carrying this strategy to
higher dimensions runs into two obstructions. First, there is no known
test for type~D that avoids explicit reference to the aligned null frame
(the best available are the Bel-Debever type
conditions~\cite{Ortaggio2009}). Second, there is no known
higher-dimensional analog of the Pleba\'{n}ski-Demia\'{n}ski family
exhausting the type~D (cosmological) vacua. An alternative starting
point is therefore needed.

A natural and sufficiently strong condition is the very existence
of the principal tensor $h_{ab}$. It is known~\cites{Houri2007}[Sec.5.3]{FKK}
that a principal tensor exists precisely on the Kerr-NUT-(A)dS family, so
one expects the structure of the curvature $R_{abcd}$ to be tightly
linked to that of $h_{ab}$. If this link could be inverted---that is,
if $h_{ab}$ could be recovered covariantly from $R_{abcd}$ and its
derivatives---then substituting the resulting formula into the
equations defining a principal tensor would already yield an IDEAL
characterization of the family. The central question of this work is
thus to understand the algebraic relationship between the Riemann tensor
$R_{abcd}$ and the principal tensor $h_{ab}$.

\paragraph{Summary of results.}
Our results, which may be seen as intermediate steps towards an IDEAL
characterization, can be summarized as follows. To simplify the
calculations, we always work in Euclidean signature, but due to the
origin of these geometries in General Relativity, we will use
\emph{space} and \emph{spacetime} interchangeably. As a maximal-rank
$2$-form on a $D=2n+\varepsilon$ dimensional space ($\eps=0,1$), the
principal tensor decomposes pointwise into $n$ mutually orthogonal
rank-$2$ blocks $h_{\mu\,ab}$ ($\mu=1,\ldots,n$). Each block is
invariant under an independent $U(1)$ rotation, and the blocks may be
permuted among themselves; together these generate the semi-direct
product stabilizer group $U(1)^n\rtimes S_n$ that governs the whole
relationship between $h_{ab}$ and the curvature (in Lorentzian signature
the stabilizer group would have to be slightly modified).

Using the explicit curvature of Kerr-NUT-(A)dS computed
in~\cite{Hamamoto2007}, we first show (Theorem~\ref{the:RiemannInBasis})
that $R_{abcd}$ lies in a distinguished ``totally zero weight'' subspace
singled out by this stabilizer group, where it takes a compact form
built from symmetrized products of the blocks $h_{\mu\,ab}$ and their
symmetric squares $Q_{\mu\,ab}=h_{\mu\,a}{}^{c}h_{\mu\,bc}$. On the
negative side, we prove a no-go result
(Theorem~\ref{th:NoTwoFormFromRiemann}): no non-trivial $2$-form, in
particular including $h_{ab}$ itself, can be built covariantly from the
undifferentiated Riemann tensor alone. This unfortunately closes the
simplest route to an IDEAL characterization.

We then bring in the integrability condition of the conformal
Killing-Yano equation, which the curvature must satisfy by virtue of the
existence of $h_{ab}$. Read as a linear constraint
$\mathcal{I}_{abcd}{}^{ijkl} r_{ijkl}=0$ (with coefficients built from
$h_{ab}$) on the potential curvature $r_{ijkl}$, we show
(Theorems~\ref{the:GeneralizedIntegrability}
and~\ref{th:IntegrabilityCondition}) that it forces the curvature
pointwise into the same totally zero weight subspace and reduces it to a
family $r(d)_{abcd}$ with free parameters $d_1,\ldots, d_n$. Finally
(Theorems~\ref{the:IIBianchiIdentity} and~\ref{th:SolutionOfdEquation}),
show that imposing the (differential) II.~Bianchi identity $\nabla_{[a}
r(d)_{bc]de} = 0$ fixes the $d_\mu$ parameters enough for $r(d)_{abcd}$
to coincide with the curvature of a Kerr-NUT-(A)dS metric (though not
necessarily the same one as the underlying background geometry). This
last step still presupposes that the covariant derivative is the
Levi-Civita connection of a fixed Kerr-NUT-(A)dS background, so the
conclusion is not yet an IDEAL characterization but an intermediate step
towards one. We defer ideas about how these results might be assembled
into a complete characterization to the Discussion.

Where necessary, we have made liberal use of the \texttt{xAct}~\cite{xAct}
tensor algebra package for the \textsc{Mathematica} computer algebra system.

\paragraph{Summary of contents.}
The article is organized as follows. Section~\ref{sec:prelim} collects
the necessary background on higher-dimensional Kerr-NUT-(A)dS
spaces: the off-shell metric, the principal tensor, the Darboux
frame, and the explicit connection and curvature of~\cite{Hamamoto2007}.
Section~\ref{sec:GroupTheory} reviews the group-theoretic tools we
use---basic representation theory, Young diagrams, decomposition of
tensors into invariant parts. Section~\ref{sec:InvarProperties}
introduces the stabilizer group $U(1)^n\rtimes S_n$ of the principal
tensor, the associated notion of weight, and the corresponding
decomposition of Riemann-symmetric tensors, culminating in the compact
expression for the Kerr-NUT-(A)dS curvature
(Theorem~\ref{the:RiemannInBasis}). Section~\ref{sec:TwoFormFromRiemann}
proves the no-go theorem for constructing a $2$-form from the
undifferentiated curvature. Section~\ref{sec:GeneralizedIC} analyzes the
integrability condition of the conformal Killing-Yano equation as a
constraint on the curvature. Section~\ref{sec:IIBianchiIdentity} adds
the second Bianchi identity and integrates the resulting differential
constraint. We conclude in Section~\ref{sec:discuss} with a discussion
and an outlook to future work.

\section{Preliminaries} \label{sec:prelim}

In this section we summarize the main results about higher-dimensional Kerr-NUT-(A)dS spacetimes that are relevant for our study. More details can be conveniently found in the review \cite{FKK}. Throughout the article we work in Euclidean signature, in dimension $D = 2n + \eps$, where $\eps =0$ for even dimensions and $\eps=1$ for odd dimensions.

The \emph{off-shell} Kerr-NUT-(A)dS metric in all dimensions reads
\begin{align} \label{eq:HigherDimensionalKNAdSmetric}
ds^2 = \sum_{\mu=1}^n \left[ \frac{U_\mu}{X_\mu}\dd x_\mu^2 + \frac{X_\mu}{U_\mu}\left( \sum_{k=0}^{n-1} A^{(k)}_\mu \dd\psi_k \right)^2 \right]
+ \eps \frac{c}{A^{(n)}} \left( \sum_{k=0}^{n} A^{(k)} \dd\psi_k \right)^2,
\end{align}
where the metric functions are
\begin{equation} \label{eq:MetricFunctions}
A^{(k)} = \sum_{\nu_1< \dots <\nu_k} x^2_{\nu_1} \dots x^2_{\nu_k}, \quad
A^{(k)}_\mu = \sum_{\substack{\nu_1< \dots <\nu_k \\ \nu_i \neq \mu}} x^2_{\nu_1} \dots x^2_{\nu_k}, \quad
U_\mu = \prod_{\substack{\nu\\ \nu \neq \mu}} (x_\nu^2 - x_\mu^2).
\end{equation}
The constant $c$ is a free parameter. We recognize two types of coordinates - \emph{radial and longitudinal} coordinates $x_\mu$ and \emph{Killing} coordinates $\psi_k$. The metric functions $X_\mu$ are arbitrary functions of the corresponding sing\-le coordinate $X_\mu = X_\mu(x_\mu)$. Their precise form follows from the Einstein field equations and can be chosen as 
\begin{align*}
	X_\mu = \begin{cases}
		\displaystyle \sum_{k=0}^n c_k x_\mu^{2k} - 2 b_\mu x_\mu & \text{for D even},\\
		\displaystyle \sum_{k=1}^n c_k x_\mu^{2k} - 2 b_\mu - \frac{c}{x_\mu^2} & \text{for D odd,}
	\end{cases}
\end{align*}
in which case we say that the metric is \emph{on-shell}.

Different coordinate ranges and values of the $X_\mu$ functions can change the signature of the metric. The highly symmetric form \eqref{eq:HigherDimensionalKNAdSmetric} with real $X_\mu$'s, is capable of representing \emph{higher dimensional (pseudo)-spheres} and \emph{Euclidean instantons} (see \cite[Sec.4.3]{FKK} for specific signs and ranges of coordinates). We will always assume that we are working in coordinate ranges where the Euclidean signature holds.
To obtain black-hole metrics with a Lorentzian signature, one has to employ a \emph{Wick rotation} and assume that some coordinates and constants are imaginary at the beginning, which for us will be unnecessary. Details can be found in \cite[Sec.4.4]{FKK}.

The special significance of Kerr-NUT-(A)dS spacetime lies in its rich (hidden) symmetry structure encoded in the \emph{non-degenerate closed conformal Killing-Yano form} or the \emph{principal tensor} $h_{ab}$, which obeys
\begin{align} \label{eq:PrincipalTensorEquation}
\nabla_{a} h_{bc} = g_{ab} \xi_c - g_{ac} \xi_b, \qquad \xi_{a} = \frac{1}{D-1}\nabla^b h_{ba}.
\end{align}
The \emph{primary vector} $\xi_a$ is a Killing vector, $\nabla_{(a}\xi_{b)} = 0$.
The symmetric tensor $Q_{ab} \equiv h_{ac} h_{b}{}^{c}$ is a \emph{conformal Killing tensor},
\begin{align} \label{eq:KillingTensorQ}
\nabla_{(a} Q_{bc)} = 2g_{(ab} \eta_{c)} \equiv 2g_{(ab} h_{c)d} \xi^d.
\end{align}
Higher order Killing tensors and conformal Killing tensors can be generated from $h_{ab}$ as well using the Hodge dual and wedge product, therefore, $h_{ab}$ is the most important object from this group \cite[Sec.2.8]{FKK}.

The symmetries of the spacetime are better revealed and easier studied in a so-called \emph{Darboux basis}, which is a non-holonomic orthonormal frame of one-forms, in which the metric is diagonal 
\begin{align} \label{eq:MetricInDarbouxFrame}
g = \sum_{\mu=1}^n (e^\mu \otimes e^\mu + \hat{e}^\mu \otimes \hat{e}^\mu) + \hat{e}^0 \otimes \hat{e}^0,
\end{align}
and the principal tensor acquires the simple block diagonal form
\begin{align} \label{eq:PrincipalTensor}
h = \sum_{\mu=1}^n \, x_\mu e^\mu \wedge \hat{e}^\mu, \quad \Leftrightarrow \quad 
(h) = \left( \begin{matrix}
B_1 && 0 && \dots \\
0 && B_2 && \dots \\
\vdots && \vdots && \ddots
\end{matrix} \right), \quad
B_i = \left( \begin{matrix}
0 && x_i \\
-x_i && 0
\end{matrix} \right).
\end{align}
Such frame exists for every two-form. 
Here and further in the text, we employ the usual notation $e^{\mu + n} \equiv \hat{e}^n, e^{2n+1} \equiv \hat{e}^0$. 
For the metric \eqref{eq:HigherDimensionalKNAdSmetric} the Darboux basis explicitly reads
\begin{align} \label{eq:DarbouxFrameDef}
e^\mu = \sqrt{\frac{U_\mu}{X_\mu}} \dd x_\mu, \quad
\hat{e}^\mu = \sqrt{\frac{X_\mu}{U_\mu}} \sum_{k=0}^{n-1} A^{(k)}_\mu \dd\psi_k, \quad
\hat{e}^0 = \varepsilon \sqrt{\frac{c}{A^{(n)}}} \sum_{k=0}^n A^{(k)} \dd\psi_k.
\end{align}
The conformal Killing tensor \eqref{eq:KillingTensorQ} takes form
\begin{align} \label{eq:ConformalKillingTensorQ}
Q = \sum_{\mu=1}^n x_{\mu}^2 (e^\mu \otimes e^\mu + \hat{e}^\mu \otimes \hat{e}^\mu).
\end{align}

\begin{remark} \label{rem:GreekIndexConvention}
We use the Greek indices to label vectors of the Darboux frame and associated quantities. Repeated Greek indices are \emph{not} summed over, sums are always explicitly written. The range of the Greek indices is $\{1\dots n\}$, where the dimension of the spacetime is $D= 2n+\varepsilon$. 
\end{remark}

Existence of the principal tensor $h_{ab}$ is strongly restrictive for the spacetime. It even turns out that the spacetime is necessarily Kerr-NUT-(A)dS as the following theorem shows \cite[Thm.1]{Houri2007}, \cite[Thm.1]{Houri2009}:
\begin{proposition} \label{th:Uniqueness}
Let $h_{ab}$ be the principal tensor \eqref{eq:PrincipalTensorEquation} in $D$-dimensional spacetime $(\man,g)$, which satisfies
\begin{align}
\dd h = 0, \quad \Lie_\xi g = 0, \qquad \Lie_\xi h = 0.
\end{align}
Then locally (on any open, connected, simply-connected domain)%
	\footnote{The \emph{local} qualification is not given in~\cite{Houri2007, Houri2009}, but the main step in their integration uses the Frobenius theorem, which without further hypotheses only gives local results.} %
$(\man,g)$ is only the higher-dimensional Kerr-NUT-(A)dS spacetime. 
\end{proposition}

The covariant derivative of an arbitrary spacetime function $f$, which is constant along Killing vector orbits, can be expressed in the following way
\begin{align*}
\nabla f = \sum_{\mu=1}^n \partial_\mu f \, \dd x_\mu = \sum_\mu \partial_\mu f \sqrt{\frac{X_\mu}{U_\mu}} e^\mu \equiv \sum_\mu M_\mu \partial_\mu f e^\mu, \quad M_\mu \equiv \sqrt{\frac{X_\mu}{U_\mu}},
\end{align*}
where we used that $\psi_k$ are the Killing coordinates and the definition of the Darboux frame \eqref{eq:DarbouxFrameDef}.
The covariant derivative of a vector of an orthonormal frame is captured by the rotation one-forms $\upomega^b{}_a$ \cite{MTW}:
\begin{align} \label{eq:RotationOneFormDef}
\nabla e_a \equiv e_b \, \upomega^b{}_{a}.
\end{align}
They are governed by the first structure equation $\dd e^a + \upomega^a_{~b} \wedge e^b = 0$. With one index lowered, they are antisymmetric $\upomega_{ab} = - \upomega_{ba}$. Using these two properties, $\upomega_{ab}$ can be found explicitly. For the Kerr-NUT-(A)dS spacetime, they were calculated in \cite[Secs.2--3]{Hamamoto2007}.

\begin{proposition} \label{lem:RotOneForms}
The rotation one-forms of the Darboux frame defined by \eqref{eq:RotationOneFormDef} are 
\begin{align}
\upomega_{\mu \nu} &= V_{\nu \mu} e^\mu - V_{\mu \nu} e^\nu \quad \qquad \qquad
\upomega_{n+\mu, n+\nu} = -U_{\mu \nu} e^\mu + U_{\nu \mu} e^\nu \nonumber \\
\upomega_{\mu, n+\nu} &= \delta_{\mu \nu} \big(-\pd_\mu M_\mu \hat{e}^\mu + \sum_{\rho} U_{\mu \rho} \hat{e}^\rho + s_\mu \hat{e}^0 \big) - V_{\mu \nu} \hat{e}^\nu + U_{\mu \nu} \hat{e}^\mu, \label{eq:RotOneForms} \\
\upomega_{0 \mu} &= - s_\mu \hat{e}^\mu + \varepsilon \frac{M_\mu}{x_\mu} \hat{e}^0, \qquad \qquad
\upomega_{0, n+\mu} = s_\mu e^\mu, \nonumber
\end{align}
where we introduced auxiliary functions\footnote{Notation for a special function $U_{\mu\nu}$ is similar to notation of metric functions $U_\mu$. The reader can distinguish them easily by number of indices. Functions $U_\mu$ will almost always remain hidden inside $M_\mu$.}
\begin{align} \label{eq:FunctionsUV}
V_{\mu \nu} \equiv (1-\delta_{\mu \nu}) \frac{x_\mu}{x^2_\mu - x^2_\nu}M_\mu, \quad
U_{\mu \nu} \equiv (1-\delta_{\mu \nu}) \frac{x_\mu}{x^2_\mu - x^2_\nu}M_\nu, \quad
s_\mu \equiv \varepsilon \frac{1}{x_\mu} \sqrt{\frac{c}{A^{(n)}}}.
\end{align}
\end{proposition}

The curvature of the $D$-dimensional Kerr-NUT-(A)dS spacetime was also explicitly calculated in \cite{Hamamoto2007}, where it was presented as curvature two-forms $\mathcal{R}_{\mu \nu}$ \cite{MTW, Wald}, which are related to the Riemann tensor as 
\begin{align} \label{eq:RiemmanFromRforms}
\mathcal{R}_{\mu \nu} = \sum_{\substack{\alpha, \beta =1 \\ \alpha<\beta}}^D R_{\mu \nu \alpha \beta} \, e^{\alpha} \wedge e^{\beta} \quad \iff \quad R = \sum_{\substack{\mu, \nu=1 \\ \mu<\nu}}^D \mathcal{R}_{\mu \nu} \, e^{\mu} \wedge e^{\nu}.
\end{align}
We summarize the result, simplifying of the original formulas for our purposes, in the following
\begin{proposition}
The curvature two-forms of the $D$-dimensional Kerr-NUT-(A)dS spacetime in the \emph{Euclidean} Darboux frame read
\begin{equation} \label{eq:HamamotoFormulas}
\begin{aligned}
\mathcal{R}_{\mu \nu} &= S_{\mu \nu} e^{\mu} \wedge e^{\nu} + F_{\mu \nu} \hat{e}^{\mu} \wedge \hat{e}^{\nu}, &
\mathcal{R}_{\mu, n + \mu} &= D_\mu e^{\mu} \wedge \hat{e}^{\mu} + 2\sum_{\rho\neq\mu} F_{\mu \rho} e^{\rho} \wedge \hat{e}^{\rho}, \\
\mathcal{R}_{\mu, n + \nu} &= S_{\mu \nu} e^{\mu} \wedge \hat{e}^{\nu} + F_{\mu \nu} e^{\nu} \wedge \hat{e}^{\mu}, &
\mathcal{R}_{n+\mu, n+\nu} &= F_{\mu \nu} e^{\mu} \wedge e^{\nu} + S_{\mu \nu} \hat{e}^{\mu} \wedge \hat{e}^{\nu}, \\
\mathcal{R}_{\mu, 2n + 1} &= E_\mu e^{\mu} \wedge \hat{e}^{0}, &
\mathcal{R}_{n + \mu, 2n + 1} &= E_\mu \hat{e}^{\mu} \wedge \hat{e}^{0},
\end{aligned}
\end{equation}
where we used the definitions
\begin{equation} \label{eq:CurvatureFunctions}
\begin{aligned}
S_{\mu \nu} &\equiv - \frac{1}{2(x^2_\mu - x^2_\nu)} \left( x_\mu \frac{\pd \mathsf{M}}{\pd x_\mu} - x_\nu \frac{\pd \mathsf{M}}{\pd x_\nu}\right), \quad &
D_{\mu} &\equiv - \frac{1}{2} \frac{\pd^2 \mathsf{M}}{\pd x_\mu^2}, \\
F_{\mu \nu} &\equiv \frac{1}{2(x^2_\mu - x^2_\nu)} \left( x_\mu \frac{\pd \mathsf{M}}{\pd x_\nu} - x_\nu \frac{\pd \mathsf{M}}{\pd x_\mu}\right), &
E_{\mu} &\equiv - \frac{1}{2 x_\mu} \frac{\pd \mathsf{M}}{\pd x_\mu},
\end{aligned}
\end{equation}
with
\begin{equation} \label{eq:MainMetricFunction}
\mathsf{M} \equiv \sum^n_{\mu=1} M_\mu^2 + \varepsilon \frac{c}{A^{(n)}} = \sum^n_{\mu=1} \frac{X_\mu}{U_\mu} + \varepsilon \frac{c}{A^{(n)}}.
\end{equation}
The functions $S_{\mu\nu}, F_{\mu\nu}$ are by definition symmetric $S_{\mu\nu} = S_{\nu\mu}, \; F_{\mu\nu} = F_{\nu\mu}$. 
\end{proposition}

\section{Relevant facts from the group theory} \label{sec:GroupTheory}

For the convenience of the reader, in this section we briefly summarize
concepts from the group theory that are important for our work
(see~\cite{Simon} for example for the relevant background). It also
serves to fix our notation and definitions. 

A \emph{representation} or \emph{(linear) action} of a group $G$ on a vector space $V$ is a homomorphism $\pi\colon G \to \operatorname{End}(V)$. \emph{Irreducible}
representations (\emph{irreps}) are those that contain no subspace invariant under the
group action. Given two representations $\pi_i\colon G_i \to
\operatorname{End}(V_i)$ ($i=1,2$), an \emph{intertwiner} is a linear
map $I\colon V_1 \to V_2$ that respects the group action, $\pi_2(g) I =
I \pi_1(g)$ for all $g\in G$; two representations are \emph{equivalent}
if there exists an invertible intertwiner. \emph{Schur's lemma}~\cite[\S
II.4]{Simon} is the fundamental result that an intertwiner between
inequivalent irreducible complex representations must be identically
zero, while between equivalent representations it is unique up to
rescaling.

As usual, direct sums $V_1 \oplus V_2$ and tensor products $V_1 \otimes
V_2$ are given a representation structure by simultaneous action of
group elements on both summands and factors, respectively, naturally
defining representations of the \emph{direct product} group $G_1 \times
G_2$. When $G_1 = G_2 = G$, representations of $G$ are obtained by
restricting to the diagonal subgroup $G \subset G \times G$, consisting
of pairs $(g,g)$, $g\in G$. It is a useful fact that the irreps of $G_1
\times G_2$ coincide with the tensor products $\pi \cong \pi_1 \otimes
\pi_2$ of irreps $\pi_1$ of $G_1$ and $\pi_2$ of $G_2$. This fact has a
more complicated generalization to \emph{semi-direct products} $H \rtimes
G$, which stands for the set $H \times G$ with the group law $(h_1, g_1)
\cdot (h_2, g_2) = (h_1 (g_1 \triangleright h_1), g_1 g_2)$ with respect
to a particular action ${\triangleright}\colon G \times H \to H$ of
$G$ on $H$ by automorphisms~\cite[Sec.I.3]{Simon}. The details of the
generalization are not relevant to us, except that in parts
section~\ref{sec:InvarProperties} we will take insipration from the
simplified version where $H$ is abelian (cf.~\cite[Ch.V]{Simon}). The
inspiration comes from the observation that, given an irrep of $G$ on
$V^G$, irreps of $H \rtimes G$ are constructed on $\bigoplus_i V_i^H
\otimes V^G$ where the summation is over an orbit of the induced action of $G$
on irreps of $H$.

We will rely on representation theory and Schur's lemma in particular as
an aid in solving some large systems of linear equations, whose
corresponding linear maps have the intertwining property with respect to
a group action. The decomposing the domain and codomain of the linear
map into direct sums of irreducible representations, Schur's lemma
asserts that the linear map must have a block diagonal form, when the
representations are grouped by equivalence class. This turns a
potentially dense linear system into a sparse one and allows practical
solutions of even very high dimensional systems.

In several places, we give explicit decompositions of some
representations into irreducibles (for instance when taking tensor
products, or restricting a representation to a smaller group). Such
decompositions follow from a general theory that we do not discuss here,
except to mention that they can be computed using a computer program
such as \emph{Lie}~\cite{Lie}. In fact, the general theory is not needed
for practical calculations. In all cases, we describe witnesses to the
decomposition, of say $V$ into $\bigoplus_k V_k$, as explicit
intertwining projections $P_k\colon V \to V_k$ and embeddings $E_k
\colon V_k \to V$ that satisfy a completeness relation, namely $P_k E_l
= 0$ if $k\ne l$, $P_k E_k$ invertible on $V_k$ and $\sum_k E_k P_k$
invertible on $V$. The abstract decomposition into irreducibles served
us mostly as a guide to which and how many representations to include
among the $V_k$.

In practice, we have found that such an explicit decomposition gives an
acceptable sparse simplification even if the individual $V_k$ are not
themselves irreducible, which can significantly reduce the number of
needed projections and embeddings (cf.\
section~\ref{sec:DecompositionIntoDiagonals}). It suffices that an
eventualy full decomposition of $V_k$ and $V_l$ into irreducibles gives
only a small overlap in equivalent sub-representations when $k\ne l$.

\subsection{Young diagrams and tableaux} \label{sec:YoungDiagrams}

A specific way how to write a natural number $n$ as a sum of natural numbers $n = m_1 + \dots + m_k$ is called a \emph{partition}. 
Every partition of $n$ we can associate with a diagram made of $n$ boxes $\square$ in such a way that the first row is made of $m_1$ boxes, second row of $m_2$ boxes etc.,\ up to $k$-th row with $m_k$ boxes. All rows are aligned from the left. Such an object is called a \emph{Young diagram} $\lambda$ or a Young frame. Its size $|\lambda|$ is the number of boxes $n$. From the construction it is obvious that Young diagrams and partitions are in one-to-one correspondence and carry the same information. 

Young tableaux are useful, because they label irreducible representations of the $n$-object \emph{permutation group} $S_n$ and the \emph{general linear group} in $n$-dimensions $GL(n)$.

If we fill the Young diagram $\lambda$ with numbers from 1 to $n$ we obtain a \emph{Young tableau} $\lambda(T)$. It is called \emph{standard}, when the numbers increase in every row and every column. 
For a box $(i,j)$ of a Young diagram we define a \emph{hook length} $\mathfrak{H}(i,j)$ as the number of all boxes below it and next to it to the right including itself. 
The number of all standard tableaux is then given by \emph{hook length formula}
\begin{align} \label{eq:HookLengthFormula}
f(\lambda) = \frac{n!}{\prod_{i,j} \mathfrak{H}(i,j)}.
\end{align}
Its the dimension of the corresponding irrep of $S_n$ group. 

Analogously we define shifted hook lengths. Fix a number $d \in \Nat$ (dimension of a vector space $V$). \emph{Shifted hook lengths} of a diagram are obtained by placing the number $d$ in the upper most left corner and increasing the number by one for every box to the right and decreasing by one for every box below it.
The \emph{shifted hook length formula} determines dimension of the corresponding irrep of $GL(n)$ and reads
\begin{align} \label{eq:ShiftedHookLengthFormula}
F(\lambda) = \frac{\prod_{i,j} (n+j-i)}{\prod_{i,j} \mathfrak{H}(i,j)} .
\end{align}

Given an $n$-dimensional vector space $V$, the representations of $S_k$ and $GL(n)$ are linked via \emph{Schur-Weyl duality} \cite{SchurWeylWiki}, namely the decomposition
\begin{equation} \label{eq:SchurWeyl}
	V^{\otimes k} \cong \sum_{\lambda} (\lambda)_{S_k} \otimes (\lambda)_{GL(n)} ,
\end{equation}
where $GL(n)$ acts on $V^{\otimes k}$ via the tensor product of the fundamental representation and $S_k$ acts by permuting the tensor factors, while $\lambda$ ranges over all Young diagrams with $k$ boxes and at most $n$ rows.
Ignoring the $S_k$ action, the decomposition~\eqref{eq:SchurWeyl} and the dimension formula~\eqref{eq:HookLengthFormula} give us the multiplicity with which each irreducible representation of $GL(n)$ occurs in the decomposition of $V^{\otimes k}$. If we assign, as is typical, an abstract index to each tensor factor, filling $\lambda$ produces a Young tableau $\lambda(T)$. Each such tableau defines an idempotent projector on $V^{\otimes k}$, whose image is one of the irreducible invariant subspaces of type $\lambda$:
\begin{definition}
The \emph{Young symmetrizer} corresponding to a Young tableau $\lambda(T)$ is
\begin{align} \label{eq:YoungSymmetrizer}
P_{\lambda(T)} = \frac{f(\lambda)}{n!} \sum_{\varsigma \in C(T)} \mathrm{sgn}(\varsigma) \varsigma \sum_{\varrho \in R(T)} \varrho,
\end{align}
where $C(T)$ ($R(T)$) is a set of all permutations of the tensor factors of $V^{\otimes k}$ that leave columns (rows) of $\lambda(T)$ invariant. 
\end{definition}

\subsection{$S_n$ invariant tensors}

Representations of $GL(n)$ and $S_n$ can be related in another way, which is distinct from Schur-Weyl duality. Namely, given an $n$-dimensional vector space $V$ and a distinguished basis $\{e_a^\mu\}_{\mu=1}^{n}$, where $a$ is an abstract index corresponding to the vector space $V$ and $\mu$ labels the basis vectors, there is a corresponding subgroup $S_n \subset GL(n)$ that permutes the basis vectors and acts on $V$ by linear extension. It will be important for us later to decompose tensor representations of $GL(n)$ into subspaces invariant with respect to the $S_n$ subgroup.

Our approach is structurally similar to the better known case of the orthogonal subgroup $O(n) \subset GL(n)$, where it is well-known that map between tensor representations of $GL(n)$ that intertwines the $O(n)$ action is generated by index permutations and contractions with or multiplications by the Kronecker $\delta_{ab}$ left invariant by the $O(n)$ subgroup~\cite[\S\S10-5--7]{hamermesh}, \cite{Littlewood1958}.
The difference for $S_n$ is that the role of $\delta_{ab}$ is replaced by the \emph{diagonal} invariant tensor%
	\footnote{In this section we assume that indices run over all values $\{1,\dots,n\}$ without writing the range explicitly.}
\begin{equation} \label{eq:DefDelta}
\Delta_{abc} \equiv \sum_\mu e^\mu_a e^\mu_b e^\mu_c.
\end{equation}
Following the main observation of~\cite[Thm.7, Thm.8]{Littlewood1958} and~\cite{Jones1994}, \emph{all} maps between tensor representations of $GL(n)$ that intertwine the $S_n$ action are generated by index permutations and contractions with or multiplications by $\Delta_{abc}$, and any variant with raised indices. Note that we can freely raise and lower the tensor indices because we have a distinguished basis, a corresponding distinguished dual basis $\{e_\mu^a\}_{\mu=1}^n$ and distinguished metric $\delta_{ab} = \sum_\mu e^\mu_a e^\mu_b$.
In particular, it is straightforward to observe that the following are also invariant tensors:
\begin{align}
	v_a &\equiv \Delta_{ab}{}^b
		= \sum_\mu e^\mu_a , \nonumber \\
	\delta_{ab} &\equiv \Delta_a{}^{cd} \Delta_{cdb}
		= \sum_\mu e^\mu e^\mu, \label{eq:DefPowersOfDelta} \\
	(\Delta^p)_{a\dots bcd} &\equiv (\Delta^{p-1})_{a\dots b}{}^{i}\Delta_{icd}
		= \sum_\mu e^\mu_a \cdots e^\mu_b e^\mu_c e^\mu_d
	\quad (p\ge 2, ~ \Delta^1_{abc} \equiv \Delta_{abc}), \nonumber
\end{align}
which are also called \emph{isotropic} (or \emph{diagonal}) tensors.%
	\footnote{For uniformity of notation, $v_a$ and $\delta_{ab}$ could be written as $\Delta_a$ and $\Delta_{ab}$. We keep the current notation to be consistent with the notation and \textsc{Mathematica} code used in \cite{MatejovThesis}.}
For future reference we record here their contraction properties:
\begin{lemma} \label{lem:ContractionsOfVDelta}
Contractions of the $v$ and $\Delta$ maps:
\begin{align*}
v^a v_a &= n, & 
\Delta_a{}^{ij} \Delta_{b ij} &= \delta_{ab}, \\
\Delta_{ab}{}^{i} v_i &= \delta_{ab}, &
\Delta^{abi} \Delta_{icd} &\equiv (\Delta^2)^{ab}{}_{cd}, \\
\Delta_{a}{}^{ij} v_i v_j &= v_a, &
(\Delta^2)^{ab}{}_{ij} \Delta^{ijc} &= \Delta^{abc}, \\
(\Delta^2)^{abci} v_i &= \Delta^{abc}. & &
\end{align*}
\end{lemma}

\subsection{General decomposition of a tensor into diagonals} \label{sec:DecompositionIntoDiagonals}

When analyzing tensor equations following from a problem whose symmetry restricts to $S_n \subset GL(n)$, it is advantageous to decompose the tensors into $S_n$ invariant subspaces. Recalling the analogy with $O(n)$, where invariant subspaces can be constructed by subtracting \emph{metric traces}, $S_n$ invariant subspaces can be constructed by subtracting \emph{diagonals}, that is all parts proportional to the diagonal tensors \eqref{eq:DefPowersOfDelta}. For instance
\begin{equation}
	T_{ab} = T^{\perp}_{ab} + T^{(0)}_i \Delta^{i}{}_{ab} ,
	\quad \text{with} \quad
	T^{(0)}_i = \Delta_i{}^{ab} T_{ab} , \quad
	T^{\perp}_{ab} = T_{ab} - T^{(0)}_i \Delta^{i}{}_{ab} .
\end{equation}
Notice that neither the diagonal $T^{(0)}_i$ nor the diagonal-less $T^\perp_{ab}$ parts are $S_n$-irreducible. To achieve irreducibility, we would need to also subtract their contractions with $v_a$ (cf.~\cite[Sec.1.6.6]{MatejovThesis} for a more detailed discussion). For our purposes, it will not be necessary to subtract the $v_a$ contributions, which will significantly decrease the number of invariant subspaces in each decomposition. In this section, we will give explicit formulas for those diagonal decompositions that we will actually need.

For a general three index tensor $T_{abc}$ with no additional symmetries, we define its (partial) diagonals\footnote{The number which labels the diagonals denotes position of the unrepeated index.}
\begin{align*}
T^{(0)}_{a} &\equiv T_{aaa} = T_{ijk} (\Delta^2)_{a}^{~ijk}, \qquad
T^{(1)}_{ab} \equiv T_{abb} = T_{aij} \Delta_{b}^{~ij}, \\
T^{(2)}_{ab} &\equiv T_{bab} = T_{iaj} \Delta_{b}^{~ij}, \qquad \quad \,
T^{(3)}_{ab} \equiv T_{bba} = T_{ija} \Delta_{b}^{~ij}.
\end{align*}
Further, we subtract the full diagonal $T^{(0)}$ from all the partial diagonals $T^{(1)},T^{(2)},T^{(3)}$ to get trace-free quantities
\begin{align} \label{eq:DiagonalsT}
T^{(I) \perp}_{ab} = T^{(I)}_{ab} - T^{(0)}_i \Delta^{i}_{~ab}, \qquad I \in \{1,2,3\}. 
\end{align}
Analogously $T^{\perp}_{abc}$ is defined as diagonal-less. Because of the suitable choice of definition \eqref{eq:DiagonalsT}, we can write
\begin{align} \label{eq:DecT}
T_{abc} = T^{\perp}_{abc} + T^{(1) \perp}_{ai}\Delta^{i}_{~bc} + T^{(2) \perp}_{bi}\Delta^{i}_{~ca} + T^{(3) \perp}_{ci}\Delta^{i}_{~ab} + T^{(0)}_{i}(\Delta^2)^i_{~abc}.
\end{align}

For a four-tensor $T_{abcd}$ we define the diagonals analogously
\begin{align*}
T^{(0)}_{a} &\equiv T_{aaaa}, \quad
T^{(1)}_{ab} \equiv T_{abbb}, \quad
T^{(2)}_{ab} \equiv T_{babb}, \quad
T^{(3)}_{ab} \equiv T_{bbab}, \quad
T^{(4)}_{ab} \equiv T_{bbba}, \\
T^{(5)}_{ab} &\equiv T_{aabb}, \quad
T^{(6)}_{ab} \equiv T_{abab}, \quad
T^{(7)}_{ab} \equiv T_{abba}, \quad
T^{(1,2)}_{abc} \equiv T_{abcc}, \quad
T^{(1,3)}_{abc} \equiv T_{acbc}, \quad \\
T^{(1,4)}_{abc} &\equiv T_{accb}, \quad
T^{(2,3)}_{abc} \equiv T_{cabc}, \quad
T^{(2,4)}_{abc} \equiv T_{cacb}, \quad
T^{(3,4)}_{abc} \equiv T_{ccab}.
\end{align*}
The diagonals can be extracted by $\Delta$ maps \eqref{eq:DefDelta}, \eqref{eq:DefPowersOfDelta}, for instance $T^{(0)}_{a} = T_{ijkl} (\Delta^3)_{a}^{~ijkl}$ or $T^{(5)}_{ab} = T_{ijkl} \Delta^{~ij}_a \Delta^{~kl}_{b}$. For reference, we list all diagonal maps here:

\begin{align}
\begin{split} \label{eq:DiagonalMaps}
\{
& (\Delta^3)^{i}_{~abcd}, 
(\Delta^2)^{i}_{~abc}, (\Delta^2)^{i}_{~abd}, (\Delta^2)^{i}_{~acd}, (\Delta^2)^{i}_{~bcd}, 
\Delta^{i}_{~ab}\Delta^{j}_{~cd}, \\
& \; \Delta^{i}_{~ac}\Delta^{j}_{~bd}, \Delta^{i}_{~ad}\Delta^{j}_{~bc}, \Delta^{i}_{~ab}, \Delta^{i}_{~ac}, \Delta^{i}_{~ad}, \Delta^{i}_{~bc}, \Delta^{i}_{~bd}, \Delta^{i}_{~cd}
\}.
\end{split}
\end{align}

Every partial diagonal contains again diagonals with higher number of repeated indices, which we have to extract. In agreement with formula \eqref{eq:DiagonalsT} we write
\begin{align} \label{eq:Decomposition2Tensor}
T^{(I)}_{ab} = T^{(I)\perp}_{ab} + T^{(0)}_{i}\Delta^{i}_{~ab}, \qquad I \in \{1,\dots 7\}.
\end{align}
For decomposition of the partial diagonals with three indices, we use the formula \eqref{eq:DecT}, which gives
\begin{align} \label{eq:Decomposition3Tensor}
T^{(I,J)}_{abc} = T^{(I,J)\perp}_{abc} + T^{(I)\perp}_{ai} \Delta^i_{~bc} + T^{(J)\perp}_{bi} \Delta^i_{~ac} + \mathfrak{E}_{ci} T^{(K)\perp}_{ci} \Delta^i_{~ab} + T^{(0)}_{i} (\Delta^2)^i_{~abc}, 
\end{align}
where $K = -|I+J-5|+7$ and $\mathfrak{E}_{ab}$ is the identity for $I=1$ and exchanges the order of indices for $I>1$, namely $\mathfrak{E}_{ab} T_{ab} = T_{ba}$ iff $I>1$. 

Analogously to our previous example of a three-tensor, we can derive for a four-tensor
\begin{align}
T_{abcd} &= T^\perp_{abcd} + T^{(1,2) \perp}_{abi}\Delta^i_{~cd} + T^{(1,3) \perp}_{aci}\Delta^i_{~bd} + T^{(1,4) \perp}_{adi}\Delta^i_{~bc} + T^{(2,3) \perp}_{bci}\Delta^i_{~ad} + T^{(2,4) \perp}_{bdi}\Delta^i_{~ac} \nonumber \\
&\quad + T^{(3,4) \perp}_{cdi}\Delta^i_{~ab}
+ T^{(1) \perp}_{ai}(\Delta^2)^i_{~bcd} + T^{(2) \perp}_{bi}(\Delta^2)^i_{~acd} + T^{(3) \perp}_{ci}(\Delta^2)^i_{~abd}  \label{eq:Decomposition4Tensor} \\
&\quad + T^{(4) \perp}_{di}(\Delta^2)^i_{~abc} + T^{(5) \perp}_{ij}\Delta^{i}_{~ab}\Delta^j_{cd} + T^{(6) \perp}_{ij}\Delta^{i}_{~ac}\Delta^j_{bd} + T^{(7) \perp}_{ij}\Delta^{i}_{~ad}\Delta^j_{bc} + T^{(0)}_{i}(\Delta^3)^i_{~abcd}. \nonumber
\end{align}
The decomposition is valid for \emph{every} four-tensor. In specific cases, when $T_{abcd}$ has non-trivial index symmetries, the formula \eqref{eq:Decomposition4Tensor} simplifies, because some of the diagonals vanish or can become dependent. The simplest example is a four-form for which all the diagonal parts are zero $T_{[abcd]} = T^\perp_{abcd}$.
See the discussion around Tables~\ref{tab:Diagonals} and~\ref{tab:DiagonalsE0} for a detailed treatment of a more complicated case.

The formulas \eqref{eq:DiagonalsT}, \eqref{eq:DecT} and \eqref{eq:Decomposition4Tensor} can be understood from a more abstract level as linear maps (for $k=2,3,4$)
\begin{equation} \label{eq:D-maps}
	\mathfrak{D}^k \colon \mathcal{D}^k = \bigoplus_{(d)} \mathcal{D}^k_{(d)} \rightarrow \mathbb{T}^k ,
	\quad
	\mathfrak{D}^k \colon \big(T_k^{(d)\perp}\big) \mapsto T = \sum_{(d)} T_k^{(d)\perp} (\Delta_k^{(d)}).
\end{equation}
which take for each $k$ a tuple of \emph{all} diagonals $\big(T_k^{(d)\perp}\big)$, $T_k^{(d)\perp} \in \mathcal{D}^k_{(d)}$, and map them back to the corresponding tensor $T$ of rank $k$.
Relations \eqref{eq:DiagonalsT}, \eqref{eq:DecT} and \eqref{eq:Decomposition4Tensor} are fully general parametrizations, in the sense that the corresponding $\mathfrak{D}^k$ maps in~\eqref{eq:D-maps} are bijections. The reverse maps $\bar{\mathfrak{D}}^k\colon \mathbb{T}^k \rightarrow \mathcal{D}^k$ defined by composing $T$ with the appropriate diagonal extraction maps from~\eqref{eq:DiagonalMaps} satisfy $\bar{\mathfrak{D}}^k\circ \mathfrak{D}^k = \text{id}_{\mathcal{D}^k}$, witnessing the bijectivity. The inversion formula holds because of the $\Delta$-composition identities (Lemma~\ref{lem:ContractionsOfVDelta}) and the fact that each of the sub-diagonals $T_k^{(d)\perp}$ (with more than one index) has been defined itself to be diagonal-less.

\section{Symmetry of the Kerr-NUT-(A)dS spacetime \\ and the stabilizer group $U(1)^n \rtimes S_n$} \label{sec:InvarProperties}

It is useful to study transformations preserving the Darboux basis, namely all the transformations in which the metric is diagonal and the principal tensor has the special block-diagonal form \eqref{eq:PrincipalTensor}. Such transformations are a subgroup of $O(D)$ (recall the spacetime dimension $D=2n+\eps$ from section~\ref{sec:prelim}), which leave the $2 \times 2$ blocks of $h$ unchanged. The most general transformation preserving orientation and orthogonality of vectors in these 2-dimensional subspaces is a \emph{rotation}. The rotations can be independent, each by a different angle $\fii_1, \ldots, \fii_n$. This means that the whole transformation belongs to a direct product of $n$ copies of $U(1)$, that is $U(1)^n$. When each $\fii_\mu$ is the same for all subspaces, we identify the transformation with an element of a \emph{common} copy of $U(1) \subset U(1)^n$.
 
Another independent kind of transformation leaving the block diagonal structure of the principal tensor the same is a \emph{permutation} of the blocks, or the 2-dimensional subspaces of the tangent space respectively, acting as
\begin{equation} \label{eq:ActionOfSn}
e^\mu \mapsto e^{\sigma(\mu)}, \quad
\hat{e}^\mu \mapsto \hat{e}^{\sigma(\mu)}, \quad
\forall \mu \in \{1,\dots,n\}, \quad \sigma \in S_n,
\end{equation}
where $S_n$ is the group of permutations of $n$ elements.
It does not leave the principal tensor pointwise invariant, rather it preserves the whole $n$-parameter family of block diagonal matrices in which the parameters are the coordinate functions $x_{\mu}$, but it becomes an honest symmetry when the coordinates are also permuted.%
	\footnote{This transformation can be understood as coming from a coordinate transformation - permutation of the coordinates $x_\mu \mapsto x_{\sigma(\mu)}$, while the Killing coordinates $\psi_k$ are unaffected. The basis forms d$x_\mu$ and the metric functions $U_\mu, X_\mu, A^{(k)}_\mu$ are also naturally permuted in the same way, e.g.\ $A^{(k)}_\mu \mapsto A^{(k)}_{\sigma(\mu)}$, which can be easily checked. Consequently, covectors $e^\mu, \hat{e}^\mu$ of the Darboux frame \eqref{eq:DarbouxFrameDef} are permuted too $e^\mu, \hat{e}^\mu \mapsto e^{\sigma(\mu)}, \hat{e}^{\sigma(\mu)}$. Because $A^{(k)}$ is invariant under permutations, $\hat{e}^{0}$ remains the same.} %
Of course, the action of $S_n$ on $e^\mu$ and $\hat{e}^\mu$ induces an action on higher order tensors, namely acting on the indices of the scalar coefficients contracted with a product of some number of copies of $e^\mu$s and $\hat{e}^\mu$s.

It is not hard to see that these transformations generate all the symmetries of $h$, its \emph{stabilizer group} (or \emph{little group}). Ignoring the coordinate permutations, they generate a subgroup of $O(D)$ that has the semi-direct product structure $U(1)^n \rtimes S_n$, since the two kinds of transformations do not commute. 
Consider a rotation $r(\fii)_\mu$ of the $\mu$-th subspace and a permutation of the subspaces $\pi$. The composed operation $\pi \triangleright r(\fii)_\mu \equiv \pi r(\fii)_\mu \pi^{-1}$ is a rotation of a certain subspace (not necessarily of the same space labeled by $\mu$), so is an element of $U(1)^n$, which defines the automorphism action specifying the semi-direct product structure (section~\ref{sec:GroupTheory}). 

The discussed symmetry properties are better revealed by slightly adapting the Darboux frame. Let us define for every 2-dimensional subspace new basis vectors
\begin{align} \label{eq:BasisEpm}
e^{\pm \mu} = e^{\mu} \pm \ii \hat{e}^{\mu},
\end{align}
which transform by the \emph{phase factor} $e^{\pm \ii \fii_\mu}$ under the corresponding copy of $U(1)$.  It allows us to associate $e^{\pm \mu }$ with certain weights.

\begin{definition} \label{def:Weights}
Consider vectors $e^{\pm \mu}, \; \mu \in \{1,\dots, n\}$ which transform under the action of $U(1)^n$ in the following way
\begin{align*}
e^{\pm \mu} \mapsto e^{\pm \ii \fii_\mu} e^{\pm \mu}, \quad \forall \mu .
\end{align*}
We assign \emph{weight} $(0,\ldots,\pm 1,\ldots,0)$ non-vanishing in position $\mu$ to $e^{\pm \mu}$. In \emph{odd} dimensions we assign weight $(0,\ldots,0)$ to $\hat{e}^{0}$, which is not affected by the transformation $\hat{e}^{0} \mapsto \hat{e}^{0}$. 

The weight of a homogeneous monomial term $e^{+ \mu}_a \dots e^{+ \mu}_b e^{- \mu}_c \dots e^{- \mu}_d $ is then the sum of all vector weights. The weight of a mixed monomial term $e^{+ \mu}_a \dots e^{+ \mu}_b e^{- \mu}_c \dots e^{- \mu}_d $ \; $e^{+ \nu}_e \dots e^{+ \nu}_f e^{- \nu}_g \dots e^{- \nu}_h \dots $ is defined as a tuple of weights $(w_1,\dots,w_n)$ corresponding to all subspaces $\mu=1,\ldots,n$. The length of the weight tuple is always $n$. 

We say that a monomial has (global) \emph{weight zero}, when the sum of the weights $\sum_i w_i$ is equal to zero. On the other hand, we say that a monomial is \emph{totally zero weight}, when all weights are equal to zero, so its weight tuple is $(0,\dots,0)$.
Tensors have a specific weight if they are a sum of monomials with that weight, this extends also to total or global zero weight properties.
\end{definition}
The metric and the principal tensor in the new basis are
\begin{equation}
g_{ab} = \frac{1}{2}\sum_\mu ( e^{+ \mu}_a e^{- \mu}_b +  e^{- \mu}_a e^{+ \mu}_b), \quad
h_{ab} = \sum_\mu x_\mu h_{\mu\,ab} , \quad
h_{\mu\,ab} = \frac{\ii}{2} e^{+ \mu}_a \wedge e^{- \mu}_b.
\end{equation}
Notice that the metric is not diagonal anymore, just block diagonal similar to $h$. 

Let us state a useful lemma which governs multiplication properties of the new basis vectors.

\begin{lemma} \label{lem:ProductsOfEpm}
Contractions of the $e^{\pm \mu}$ basis vectors
\begin{align*}
e^{+\mu}_a e^{+\nu a} &= 0 = e^{-\mu}_a e^{-\nu a}, \quad \mathrm{for}\; \mu \ne \nu, & 
e^{+\mu}_a e^{-\nu a} &= 2 \delta^{\mu\nu},  \\
h_{\mu \,a}^{~~~b} e^{\pm\nu}_b &= \pm \ii \delta_{\mu\nu} e^{\pm\mu}_a, &
h_a^{~b} e^{\pm\mu}_b &= \pm \ii x_{\mu} e^{\pm\mu}_a.
\end{align*}
\end{lemma}

Both $g$ and $h$ are manifestly invariant under the action of $U(1)^n$, since 
\begin{align} \label{eq:DerOfZeroWeight}
e^{+ \mu}_a e^{- \mu}_b \mapsto e^{\ii \fii_\mu} e^{+ \mu}_a e^{- \ii \fii_\mu} e^{- \mu}_b = e^{+ \mu}_a e^{- \mu}_b, \; \forall \mu.
\end{align}
This observation has important consequences for our construction of the Riemann tensor in section \ref{sec:ZeroWeightSubspace}.

\begin{remark} \label{rem:SubgroupHierarchyStrategy}
Some of our main results (see section~\ref{sec:GeneralizedIC}) will
involve the analysis of a large linear system on Riemann-symmetric
tensors that is invariant with respect to the stabilizer group $U(1)^n
\rtimes S_n$ (theorem~\ref{the:ICZeroWeightTheorem}). One approach would
be to decompose the domain and the codomain of the linear system into
irreducible representations and thus block diagonalize the system by
Schur's lemma. However, even though the irreps of a semi-direct product
group can be identified using a standard result
(section~\ref{sec:GroupTheory}), we have found that the complexity
involved is prohibitive.

We have found that an alternative practical approach is to rely on a
decomposition into invariant subspaces (not necessarily irreducible)
that is simpler to realize and is organized in a hierarchy corresponding
to the subgroup relations
\[
\begin{array}{cccccc}
	U(1) & \subset & U(1)^n & \subset & U(1)^{2n} \\
	& & \cap & & \cap \\
	& & U(1)^n \rtimes S_n & \subset & GL(n)\times GL(n)\times GL(\varepsilon) & \subset GL(D) \\
	& & \cup & & \cup \\
	& & S_n & \subset & GL(n)
\end{array}
\]
Tensors on a $D=2n+\eps$ manifold naturally transform under $GL(D)$. The
$GL(n)\times GL(n) \times GL(\eps)$ subgroup corresponds to decomposing
the tangent space into subspaces by splitting the Darboux basis into the
$(e^{+\mu}, e^{-\mu}, \hat{e}^0)$ groups. The $GL(n)$ subgroup is
diagonally embedded in the product of the two copies of $GL(n)$, and its
subgroup $S_n$ coresponds to the action of permuting the $\mu$ labels on
the $e^\mu$, $\hat{e}^\mu$, or $e^{\pm\mu}$ basis elements. $U(1)
\subset U(1)^n$ is the subgroup of common phase rotation, whereas
$U(1)^n \subset U(1)^{2n}$ acts diagonally on the $e^{\pm\mu}$ basis
vectors according to definition~\ref{def:Weights}. The chain of larger
subgroups containing $U(1)^n \rtimes S_n$ preserve only part of the
structure of $h_{ab}$ or the linear system in
theorem~\ref{the:ICZeroWeightTheorem}, so instead of a strict Schur
diagonalization, they only provide a progressively sparser block
simplification. On the other hand, the smaller subgroups of $U(1)^n
\rtimes S_n$ are true symmetries, but their individual actions are too
weak to yield the desired simplifications. They need to be used in
conjunction.  
\end{remark}

\subsection{Decomposition of a Riemann-symmetric tensor} \label{sec:RiemannWeightSubspaces}

In this section, we will decompose Riemann-symmetric (RS) tensors under restriction of symmetry to $GL(n) \times GL(n) ( \times GL(\eps)) \subset GL(D)$. The $GL(\eps)$ factor is trivial for our purposes, so we mostly ignore it.
Let us demonstrate the idea on a simple example of symmetric two-tensors labeled by the Young diagram (partition) $(2)$, when $\eps=0$. Its restriction is a direct sum of representations
\begin{align*}
\mathrm{Res}[e^{+}, e^-](2) &\cong (2,0) \oplus (0,2) \oplus (1,1).
\end{align*}
The term (2,0) represents a symmetric two-tensor made of linear combinations of $e^{+\mu}e^{+\nu}$, (0,2) represents a symmetric two-tensor from $e^{-\mu}e^{-\nu}$s and (1,1) represents combinations $e^{+\mu}e^{-\nu}$. Irreps are labeled by pairs of Young diagrams, whose size is equal to the number of $e^{+\mu}$s and $e^{-\nu}$s in the tensor product. 

Let us denote a space of Riemann-symmetric tensors by $\mathfrak{R}$. A general term appearing in any four-tensor is a product of four basis vectors 
\begin{align} \label{eq:GeneralTerm4vectors}
\tilde{e}_1^\mu \otimes \tilde{e}_2^\nu \otimes \tilde{e}_3^\rho \otimes \tilde{e}_4^\sigma ,
\end{align}
where $\tilde{e}_i$ could be any of $e^+, e^-, \hat{e}^0$ (not necessarily all are the same). Using the branching rule for $GL(n) \times GL(n) \subset GL(2n)$~\cite[Eq.(8.20)]{Fulton1996} and the computer program \emph{Lie}~\cite{Lie}%
	\footnote{See the documentation for the \texttt{branch} function. It was enough to compute the decomposition in sufficiently high dimension, for instance branching from $GL(8)$ to $GL(4) \times GL(4)$.} %
to evaluate it, we found restrictions of the following necessary diagrams with respect to $(e^+, e^-)$:
{\small
\begin{align} 
\mathrm{Res}[e^{+}, e^-](111) &\cong (111,0,1)\oplus(0,111,1)\oplus(11,1,1)\oplus(1,11,1), \nonumber \\
\mathrm{Res}[e^{+}, e^-](21) &\cong (21,0) \oplus (0,21) \oplus(2,1) \oplus(1,2) \oplus (11,1) \oplus (1,11), \nonumber \\
\mathrm{Res}[e^{+}, e^-](211) &\cong (211,0)\oplus(0,211)\oplus(21,1)\oplus(1,21)\oplus(111,1) \label{eq:YoungDecOf22}  \\
& \quad \oplus(1,111)\oplus(11,11)\oplus(11,2)\oplus(2,11), \nonumber \\
\mathrm{Res}[e^{+}, e^-](22) &\cong (22,0)\oplus(0,22)\oplus(21,1)\oplus(1,21)\oplus(2,2)\oplus(11,11). \nonumber
\end{align}
}
An interested reader can check that dimensions on both sides agree using the formula \eqref{eq:ShiftedHookLengthFormula}.
For restrictions into irreducibles with respect to the sectors $(e^+, e^-, \hat{e}^0)$ we can use the previous formulas recursively
\begin{align*}
\mathrm{Res}[e^+ + e^-, \hat{e}^0](211) &\cong (211,0) \oplus (21,1) \oplus (111,1) \oplus (11,2), \\
\mathrm{Res}[e^+ + e^-, \hat{e}^0](22) &\cong (22,0) \oplus (21,1) \oplus (2,2),
\end{align*}
because products of $\hat{e}^0$s are always fully symmetric (trivial representation) and higher number of $\hat{e}^0$s would yield zero due to antisymmetrization of columns of the (22), (211) and (21) diagrams. Further we split the $(e^+, e^-)$ part to get
{\small
\begin{align}\mathrm{Res}[e^+, e^-, \hat{e}^0](211) &\cong \big[ (211,0,0)\oplus(0,211,0)\oplus(21,1,0)\oplus(1,21,0)\oplus(111,1,0) \nonumber \\
& \quad \oplus(1,111,0)\oplus(11,11,0)\oplus(11,2,0)\oplus(2,11,0)\big] \label{eq:Res211} \\ 
& \quad \oplus \big[ (21,0,1) \oplus (0,21,1) \oplus(2,1,1) \oplus(1,2,1) \oplus (11,1,1) \oplus (1,11,1)\big] \nonumber \\
& \quad \oplus \big[(111,0,1)\oplus(0,111,1)\oplus(11,1,1)\oplus(1,11,1)\big] \nonumber \\
& \quad \oplus \big[ (11,0,2) \oplus (0,11,2) \oplus (1,1,2)\big], \nonumber \\
\mathrm{Res}[e^+, e^-, \hat{e}^0](22) &\cong \big[ (22,0,0)\oplus(0,22,0)\oplus(21,1,0)\oplus(1,21,0)\oplus(2,2,0)\oplus(11,11,0)\big] \nonumber \\
& \quad \oplus \big[ (21,0,1) \oplus (0,21,1) \oplus(2,1,1) \oplus(1,2,1) \oplus (11,1,1) \oplus (1,11,1)\big] \nonumber \\
& \quad \oplus \big[ (2,0,2) \oplus (0,2,2) \oplus (1,1,2)\big]. \label{eq:Res22}
\end{align}
}
The above decomposition comes from abstract representation-theoretic considerations. To use it practically, we need to define (for each direct summand) actual projection and embedding maps. Because the restriction of the symmetry group respects our adapted basis $(e^+_\mu, e^-_\nu, \hat{e}^0)$, these maps can be defined simply by applying Young projections to appropriate groups of indices on the monomials
\[
	E^{\tilde{\mu}\tilde{\nu}\tilde{\rho}\tilde{\sigma}}_{abcd}
	= \tilde{e}^{\mu}_a \tilde{e}^{\nu}_b \tilde{e}^{\rho}_c \tilde{e}^{\sigma}_d,
\]
where an upper index slot like $\tilde{\mu}$ could stand for $\overset{+}{\mu}$, $\overset{-}{\mu}$, or $0$, meaning that the corresponding $\tilde{e}^{\mu}_a$ is replaced by $e^{+\mu}_a$, $e^{-\mu}_a$ or $\hat{e}^0_a$, respectively.
More concretely, let us define linear maps
\begin{equation} \label{eq:P-map}
\begin{aligned}
	\mathfrak{P}^R\colon \bigoplus_{(\lambda)} \mathfrak{T}_{(\lambda)} \equiv \mathfrak{T}^R &\rightarrow \mathfrak{R} , &
	\mathfrak{P}^R\colon \big(\TT^{(\lambda)}\big) \mapsto R &= \sum_{(\lambda)} \TT^{(\lambda)} E^{(\lambda)} ,
	\\
	\bar{\mathfrak{P}}^R\colon \mathfrak{R} &\rightarrow \mathfrak{T}^R , &
	\bar{\mathfrak{P}}^R\colon R \mapsto \big(\TT^{(\lambda)}\big) &= \bigoplus_{(\lambda)} P^{(\lambda)} R ,
\end{aligned}
\end{equation}
where $E^{(\lambda)}$ and $P^{(\lambda)}$ are the individual embedding and projection maps, with the summation labels $(\lambda)$ ranging over to the triples of Young diagrams appearing in~\eqref{eq:Res22},
which map between the parameters $(\TT^{(\lambda)}) = (\TT^{(22,0,0)}_{\mu\nu\rho\sigma}, \TT^{(21,1,0)}_{\mu\nu\rho\sigma}, \dots)$, $\TT^{(\lambda)} \in \mathfrak{T}_{(\lambda)}$.
The summands of $\mathfrak{T}^R$ are all of the form $\mathfrak{T}_{(\lambda_+, \lambda_-,\lambda_0)} = \lambda_+ \otimes \lambda_-$, where on the right-hand side $\lambda_\pm$ stands for the representation labelled by the corresponding Young diagram, where the tensor product structure is inherent in the corresponding groupings of the Greek indices. Since all the representations labelled by $\lambda_0$ are implicitly fully symmetric and 1-dimensional, they are omitted from the tensor product. 
The tensors underlying the projection maps take the explicit form
\begin{small}
\begin{equation} \label{eq:ICProjectors22}
\begin{aligned} 
(P^{(22,0,0)})^{\mu\nu\rho\sigma}_{abcd} &= P_{\,\young(ac,bd)} P_{\,\young(\mu\rho,\nu\sigma)}\, E^{{\overset{+}{\mu}}{\overset{+}{\nu}}{\overset{+}{\rho}}{\overset{+}{\sigma}}}_{abcd} , &
(P^{(0,22,0)})^{\mu\nu\rho\sigma}_{abcd} &= P_{\,\young(ac,bd)} P_{\,\young(\mu\rho,\nu\sigma)}\, E^{{\overset{-}{\mu}}{\overset{-}{\nu}}{\overset{-}{\rho}}{\overset{-}{\sigma}}}_{abcd} , \\
(P^{(21,1,0)})^{\mu\nu\rho\sigma}_{abcd} &= P_{\,\young(ac,bd)} P_{\,\young(\mu\rho,\nu)}\, E^{{\overset{+}{\mu}}{\overset{+}{\nu}}{\overset{+}{\rho}}{\overset{-}{\sigma}}}_{abcd} , &
(P^{(1,21,0)})^{\mu\nu\rho\sigma}_{abcd} &= P_{\,\young(ac,bd)} P_{\,\young(\mu\rho,\sigma)}\, E^{{\overset{-}{\mu}}{\overset{+}{\nu}}{\overset{-}{\rho}}{\overset{-}{\sigma}}}_{abcd} ,
\\
(P^{(2,2,0)})^{\mu\nu\rho\sigma}_{abcd} &= P_{\,\young(ac,bd)} P_{\,\young(\mu\rho)} P_{\,\young(\nu\sigma)} \, E^{{\overset{+}{\mu}}{\overset{-}{\nu}}{\overset{+}{\rho}}{\overset{-}{\sigma}}}_{abcd} , &
(P^{(11,11,0)})^{\mu\nu\rho\sigma}_{abcd} &= P_{\,\young(ac,bd)} P_{\,\young(\mu,\nu)} P_{\,\young(\rho,\sigma)} \, E^{{\overset{+}{\mu}}{\overset{+}{\nu}}{\overset{-}{\rho}}{\overset{-}{\sigma}}}_{abcd} ,
\\ \\
(P^{(21,0,1)})^{\mu\nu\rho}_{abcd} &= P_{\,\young(ac,bd)} P_{\,\young(\mu\rho,\nu)}\, E^{{\overset{+}{\mu}}{\overset{+}{\nu}}{\overset{+}{\rho}}{0}}_{abcd} , &
(P^{(0,21,1)})^{\mu\nu\rho}_{abcd} &= P_{\,\young(ac,bd)} P_{\,\young(\mu\rho,\nu)}\, E^{{\overset{-}{\mu}}{\overset{-}{\nu}}{\overset{-}{\rho}}{0}}_{abcd} , \\
(P^{(2,1,1)})^{\mu\nu\rho}_{abcd} &= P_{\,\young(ac,bd)} P_{\,\young(\mu\rho)} \, E^{{\overset{+}{\mu}}{\overset{-}{\nu}}{\overset{+}{\rho}}{0}}_{abcd} , &
(P^{(1,2,1)})^{\mu\nu\rho}_{abcd} &= P_{\,\young(ac,bd)} P_{\,\young(\mu\rho)} \, E^{{\overset{-}{\mu}}{\overset{+}{\nu}}{\overset{-}{\rho}}{0}}_{abcd} ,
\\
(P^{(11,1,1)})^{\mu\nu\rho}_{abcd} &= P_{\,\young(ac,bd)} P_{\,\young(\mu,\nu)} \, E^{{\overset{+}{\mu}}{\overset{+}{\nu}}{\overset{-}{\rho}}{0}}_{abcd} , &
(P^{(1,11,1)})^{\mu\nu\rho}_{abcd} &= P_{\,\young(ac,bd)} P_{\,\young(\mu,\nu)} \, E^{{\overset{-}{\mu}}{\overset{-}{\nu}}{\overset{+}{\rho}}{0}}_{abcd} , 
\\ \\
(P^{(2,0,2)})^{\mu\rho}_{abcd} &= P_{\,\young(ac,bd)} P_{\,\young(\mu\rho)}\, E^{{\overset{+}{\mu}}{0}{\overset{+}{\rho}}{0}}_{abcd} , &
(P^{(0,2,2)})^{\mu\rho}_{abcd} &= P_{\,\young(ac,bd)} P_{\,\young(\mu\rho)}\, E^{{\overset{-}{\mu}}{0}{\overset{-}{\rho}}{0}}_{abcd} , \\
(P^{(1,1,2)})^{\mu\rho}_{abcd} &= P_{\,\young(ac,bd)} E^{{\overset{+}{\mu}}{0}{\overset{-}{\rho}}{0}}_{abcd} .
\end{aligned}
\end{equation}
\end{small}
Because the matrix of scalar products of $e^{\pm}$ is off-diagonal (lemma~\ref{lem:ProductsOfEpm}), to every projection the corresponding embedding is defined with $+$ and $-$ exchanged, for instance 
\begin{equation*}
(E^{(21,1,0)})^{\mu\nu\rho\sigma}_{abcd} = P_{\,\young(ac,bd)} P_{\,\young(\mu\rho,\nu)}\, E^{{\overset{-}{\mu}}{\overset{-}{\nu}}{\overset{-}{\rho}}{\overset{+}{\sigma}}}_{abcd}.
\end{equation*}

The $\mathfrak{P}^R$ map~\eqref{eq:P-map} allows us to write an explicit decomposition of $\mathfrak{R}$ into subspaces irreducible with respect to the product group $GL(n)\times GL(n) \subset GL(2n+\varepsilon)$ labeled by triples of Young diagrams:
\begin{align} 
\begin{split} \label{eq:DecompositionOfRtoWeightSectors}
\RR_{abcd} &=
\TT^{(22,0,0)}_{\mu\nu\rho\sigma} (E^{(22,0,0)})^{\mu\nu\rho\sigma}_{abcd}
	\\
& \quad
+ \TT^{(21,0,1)}_{\mu\nu\rho} (E^{(21,0,1)})^{\mu\nu\rho}_{abcd}
	\\
& \quad
+ \TT^{(21,1,0)}_{\mu\nu\rho\sigma} (E^{(21,1,0)})^{\mu\nu\rho\sigma}_{abcd}
+ \TT^{(2,0,2)}_{\mu\nu} (E^{(2,0,2)})^{\mu\nu}_{abcd}
	\\
& \quad
+ \TT^{(2,1,1)}_{\mu\nu\rho} (E^{(2,1,1)})^{\mu\nu\rho}_{abcd}
+  \TT^{(11,1,1)}_{\mu\nu\rho} (E^{(11,1,1)})^{\mu\nu\rho}_{abcd}
	\\
& \quad
+ \TT^{(2,2,0)}_{\mu\nu\rho\sigma} (E^{(2,2,0)})^{\mu\nu\rho\sigma}_{abcd}
+ \TT^{(11,11,0)}_{\mu\nu\rho\sigma} (E^{(11,11,0)})^{\mu\nu\rho\sigma}_{abcd}
+ \TT^{(1,1,2)}_{\mu\nu} (E^{(1,1,2)})^{\mu\nu}_{abcd},
	\\
& \quad
+ \TT^{(1,2,1)}_{\mu\nu\rho} (E^{(1,2,1)})^{\mu\nu\rho}_{abcd}
+  \TT^{(1,11,1)}_{\mu\nu\rho} (E^{(1,11,1)})^{\mu\nu\rho}_{abcd}
	\\
& \quad
+ \TT^{(1,21,0)}_{\mu\nu\rho\sigma} (E^{(1,21,0)})^{\mu\nu\rho\sigma}_{abcd}
+ \TT^{(0,2,2)}_{\mu\nu} (E^{(0,2,2)})^{\mu\nu}_{abcd}
	\\
& \quad
+ \TT^{(0,21,1)}_{\mu\nu\rho} (E^{(0,21,1)})^{\mu\nu\rho}_{abcd}
	\\
& \quad
+ \TT^{(0,22,0)}_{\mu\nu\rho\sigma} (E^{(0,22,0)})^{\mu\nu\rho\sigma}_{abcd}
	,
\end{split}
\end{align}
in which $\TT$s are (undetermined) scalar coefficients sharing the same symmetry with embeddings in the Greek indices. We have arranged the terms on each line by their \emph{global weight}, which for each sector can be calculated from the number of $e^+$ and $e^-$ indicated by its label; e.g.\ $E^{(2,1,1)}$ contains two $e^+$s, one $e^-$ and one $\hat{e}^0$, so the global weight is $+1$. We will say that the weight of a coefficient $\TT^{({\cdots})}$ is \emph{equal} to the weight of the map with which it is contracted, e.g.\ weight of $\TT^{(2,1,1)}$ is $+1$.

As can be checked by explicit (computer) calculation, the composition $\bar{\mathfrak{P}}^R\circ \mathfrak{P}^R$ is block diagonal on $\mathfrak{T}^R = \bigoplus_{(\lambda)} \mathfrak{T}_{(\lambda)}$, each block proportional to the identity $\text{id}_{\mathfrak{T}_{(\lambda)}}$ with a non-zero coefficient. Hence $\mathfrak{P}^R$ is injective and we also know that it must be surjective, because by the abstract isomorphism \eqref{eq:Res22} the dimension of the image of $\mathfrak{P}^R$ coincides with the dimension of $\mathfrak{R}$. Thus, we can conclude that both $\mathfrak{P}^R$ and $\bar{\mathfrak{P}}^R$ are bijections.

\subsection{Decomposition of coefficients into diagonals} \label{sec:Diagonals}

\subsubsection*{Symmetry constraints on the diagonals}

The calculations in sections \ref{sec:CalculationOfProjectionsIC} and \ref{sec:CalculationsNoe0} will require us to solve equations for the coefficients $\TT^{({\cdots})}$.
Coefficients of different weights  will not be mixed in the equations, only the different terms with the same weight. Further simplification can be achieved by decomposition under the restriction to the symmetry group $S_n \subset GL(n) \subset GL(n) \times GL(n)$.
A full decomposition into $S_n$ irreps could be employed to break the equations into irreducible pieces that can be analyzed separately. However, as alluded earlier, for practical calculations it is too complicated and in our case not even necessary. We will be satisfied with partial decomposition into diagonals, section \ref{sec:DecompositionIntoDiagonals}.
For all the coefficients $\TT^{({\cdots})}$ of parametrization \eqref{eq:DecompositionOfRtoWeightSectors} we use the notation of (partial) diagonals introduced in section \ref{sec:DecompositionIntoDiagonals}. Since $\TT^{({\cdots})}$s have certain symmetries, their diagonals are dependent and formulas \eqref{eq:DecT}--\eqref{eq:Decomposition4Tensor} simplify according to the Young diagram label.

Let us now investigate these dependencies. Because there are quite a few coefficients and relations, we summarize the results in table \ref{tab:Diagonals}, and write only one example with important relations. 

For brevity, we will \emph{omit} $\perp$ sign in the notation of diagonals, e.g.\ $\TT^{(2,2,0)(6)\perp}_{\mu\nu} \equiv \TT^{(2,2,0)(6)}_{\mu\nu}$. The remaining off-diagonal part of a tensor is still denoted with $\perp$ to distinguish it from the tensor itself, e.g.\ $\TT^{(11,1,1)\perp}_{\mu\nu\rho}$. All quantities are assumed to be diagonal-less.

Mono-term symmetries of the coefficient $\TT^{(22,0,0)}_{\mu\nu\rho\sigma}$ imply 
\begin{align*}
\TT^{(22,0,0)\,(0)}_{\mu} &= \TT^{(22,0,0)\,(1)}_{\mu\nu} = \TT^{(22,0,0)\,(2)}_{\mu\nu} = \TT^{(22,0,0)\,(3)}_{\mu\nu} = \TT^{(22,0,0)\,(4)}_{\mu\nu} = \TT^{(22,0,0)\,(5)}_{\mu\nu} = 0 , \\
\TT^{(22,0,0)\,(6)}_{\mu\nu} &= \TT^{(22,0,0)\,(6)}_{\nu\mu} = - \TT^{(22,0,0)\,(7)}_{\mu\nu}, \\
\TT^{(22,0,0)\,(1,2)}_{\mu\nu\rho} &= \TT^{(22,0,0)\,(3,4)}_{\mu\nu\rho} = 0 , \\
\TT^{(22,0,0)\,(1,3)}_{\mu\nu\rho} &= - \TT^{(22,0,0)\,(1,4)}_{\mu\nu\rho} = - \TT^{(22,0,0)\,(2,3)}_{\mu\nu\rho} = \TT^{(22,0,0)\,(2,4)}_{\mu\nu\rho}.
\end{align*}
The multi-term symmetry related to (22) Young symmetry does not imply any additional constraints. We have exhausted all symmetry properties, so the only independent diagonals are $\TT^{(22,0,0)\,(6)}_{(\mu\nu)}, \TT^{(22,0,0)\,(1,3)}_{\mu\nu\rho}$. 

Even though we did not get any further restrictions from the Young symmetry (22) in the example, it is not always true. For $\TT^{(21,1,0)}_{\mu\nu\rho\sigma}$ the multi-term symmetry of (21) diagram implies
\begin{align*}
\TT^{(21,1,0)\,(1,2)}_{\mu\nu\rho} = 2\TT^{(21,1,0)\,(1,3)}_{[\mu\nu]\rho}.
\end{align*}

After accounting for all symmetries, besides the remaining off-diagonal parts of the 6 original $\TT^{({\cdots})}$ coefficients, we find 22 independent diagonals.

\begin{table}[h]
\begin{center}
\begin{tabular}{ |c|c|c|c|c|c|c| } 
 \hline
  & $\TT^{(22,0,0)}$ & $\TT^{(0,22,0)}$ & $\TT^{(21,1,0)}$ & $\TT^{(1,21,0)}$ & $\TT^{(2,2,0)}$ & $\TT^{(11,11,0)}$ \\
 \hline
 \hline
 $T^{(0)}_\mu$ & 0 & 0 & 0 & 0 & $\bm{\TT}^{(0)}_\mu$ & 0 \\
 \hline 
 $T^{(1)}_{\mu\nu}$ & 0 & 0 & $\bm{\TT}^{(1)}_{\mu\nu}$ & $\bm{\TT}^{(1)}_{\mu\nu}$ & $\bm{\TT}^{(1)}_{\mu\nu}$ & 0 \\ 
 $T^{(2)}_{\mu\nu}$ & 0 & 0 & $-\TT^{(1)}_{\mu\nu}$ & 0 & $\bm{\TT}^{(2)}_{\mu\nu}$ & 0 \\ 
 $T^{(3)}_{\mu\nu}$ & 0 & 0 & 0 & 0 & $\TT^{(1)}_{\mu\nu}$ & 0 \\ 
 $T^{(4)}_{\mu\nu}$ & 0 & 0 & 0 & $-\TT^{(1)}_{\mu\nu}$ & $\TT^{(2)}_{\mu\nu}$ & 0 \\ 
 $T^{(5)}_{\mu\nu}$ & 0 & 0 & 0 & $\bm{\TT}^{(5)}_{\mu\nu}$ & $\bm{\TT}^{(5)}_{(\mu\nu)}$ & 0 \\ 
 $T^{(6)}_{\mu\nu}$ & $\bm{\TT}^{(6)}_{(\mu\nu)}$ & $\bm{\TT}^{(6)}_{(\mu\nu)}$ & $\bm{\TT}^{(6)}_{\mu\nu}$ & $-\TT^{(5)}_{\nu\mu}$ & $\bm{\TT}^{(6)}_{\mu\nu}$ & $\bm{\TT}^{(6)}_{(\mu\nu)}$ \\ 
 $T^{(7)}_{\mu\nu}$ & $-\TT^{(6)}_{\mu\nu}$ & $-\TT^{(6)}_{\mu\nu}$ & $-\TT^{(6)}_{\nu\mu}$ & 0 & $\TT^{(5)}_{\mu\nu}$ & $-\TT^{(6)}_{\nu\mu}$ \\ 
 \hline
 $T^{(1,2)}_{\mu\nu\rho}$ & 0 & 0 & $2\TT^{(1,3)}_{[\mu\nu]\rho}$ & $\bm{\TT}^{(1,2)}_{\mu\nu\rho}$ & $\bm{\TT}^{(1,2)}_{\mu\nu\rho}$ & 0 \\ 
 $T^{(1,3)}_{\mu\nu\rho}$ & $\bm{\TT}^{(1,3)}_{(\mu\nu)\rho}$ & $\bm{\TT}^{(1,3)}_{(\mu\nu)\rho}$ & $\bm{\TT}^{(1,3)}_{\mu\nu\rho}$ & $\bm{\TT}^{(1,3)}_{\mu\nu\rho}$ & $\bm{\TT}^{(1,3)}_{(\mu\nu)\rho}$ & $\bm{\TT}^{(1,3)}_{\mu\nu\rho}$ \\ 
 $T^{(1,4)}_{\mu\nu\rho}$ & $-\TT^{(1,3)}_{\mu\nu\rho}$ & $-\TT^{(1,3)}_{\mu\nu\rho}$ & $\bm{\TT}^{(1,4)}_{\mu\nu\rho}$ & $2\TT^{(1,3)}_{[\mu\nu]\rho}$ & $\TT^{(1,2)}_{\mu\nu\rho}$ & $-\TT^{(1,3)}_{\mu\nu\rho}$ \\ 
 $T^{(2,3)}_{\mu\nu\rho}$ & $-\TT^{(1,3)}_{\mu\nu\rho}$ & $-\TT^{(1,3)}_{\mu\nu\rho}$ & $-\TT^{(1,3)}_{\mu\nu\rho}$ & 0 & $\TT^{(1,2)}_{\nu\mu\rho}$ & $-\TT^{(1,3)}_{\mu\nu\rho}$ \\ 
 $T^{(2,4)}_{\mu\nu\rho}$ & $\TT^{(1,3)}_{\mu\nu\rho}$ & $\TT^{(1,3)}_{\mu\nu\rho}$ & $-\TT^{(1,4)}_{\mu\nu\rho}$ & $-\TT^{(1,2)}_{\nu\mu\rho}$ & $\bm{\TT}^{(2,4)}_{(\mu\nu)\rho}$ & $\TT^{(1,3)}_{\mu\nu\rho}$ \\ 
 $T^{(3,4)}_{\mu\nu\rho}$ & 0 & 0 & 0 & $-\TT^{(1,3)}_{\nu\mu\rho}$ & $\TT^{(1,2)}_{\mu\nu\rho}$ & 0 \\ 
 \hline 
 $T^{\perp}_{\mu\nu\rho\sigma}$ & $\mathfrak{Y}^{(\lambda)}\bm{\TT}^{\perp}_{\mu\nu\rho\sigma}$ & $\mathfrak{Y}^{(\lambda)}\bm{\TT}^{\perp}_{\mu\nu\rho\sigma}$ & $\mathfrak{Y}^{(\lambda)}\bm{\TT}^{\perp}_{\mu\nu\rho\sigma}$  & $\mathfrak{Y}^{(\lambda)}\bm{\TT}^{\perp}_{\mu\nu\rho\sigma}$  & $\mathfrak{Y}^{(\lambda)}\bm{\TT}^{\perp}_{\mu\nu\rho\sigma}$ & $\mathfrak{Y}^{(\lambda)}\bm{\TT}^{\perp}_{\mu\nu\rho\sigma}$  \\
 \hline
\end{tabular}
\caption{Mutual relations between diagonals of the coefficients with four indices. Independent diagonals are highlighted in \textbf{bold}. Symbol $\mathfrak{Y}^{(\lambda)}$ stands for symmetrization of the tensor with appropriate partition $\lambda$.}
\label{tab:Diagonals}
\end{center}
\end{table}

The situation is simpler with three-index coefficients related to embeddings with one $\hat{e}^0$, because at most four diagonals can be independent. We summarize the relations in table \ref{tab:DiagonalsE0}. Besides the remaining off-diagonal parts of the 6 $\TT^{({\cdots})}$ coefficients, we find 10 independent diagonals.

\begin{table}[h]
\begin{center}
\begin{tabular}{ |c|c|c|c|c|c|c| } 
 \hline
  & $\TT^{(21,0,1)}$ & $\TT^{(0,21,1)}$ & $\TT^{(2,1,1)}$ & $\TT^{(1,2,1)}$ & $\TT^{(11,1,1)}$ & $\TT^{(1,11,1)}$ \\
 \hline
 $T^{(0)}_\mu$ & 0 & 0 & $\bm{\TT}^{(0)}_{\mu}$ & $\bm{\TT}^{(0)}_{\mu}$ & 0 & 0 \\
 $T^{(1)}_{\mu\nu}$  & $\bm{\TT}^{(1)}_{\mu\nu}$  & $\bm{\TT}^{(1)}_{\mu\nu}$ & $\bm{\TT}^{(1)}_{\mu\nu}$ & $\bm{\TT}^{(1)}_{\mu\nu}$ & $\bm{\TT}^{(1)}_{\mu\nu}$ & $\bm{\TT}^{(1)}_{\mu\nu}$ \\
 $T^{(2)}_{\mu\nu}$  & $-\TT^{(1)}_{\mu\nu}$ & $-\TT^{(1)}_{\mu\nu}$ & $\bm{\TT}^{(2)}_{\mu\nu}$ & $\bm{\TT}^{(2)}_{\mu\nu}$ & $-\TT^{(1)}_{\mu\nu}$ & $-\TT^{(1)}_{\mu\nu}$ \\
 $T^{(3)}_{\mu\nu}$  & 0 & 0 & $\TT^{(1)}_{\mu\nu}$ & $\TT^{(1)}_{\mu\nu}$ & 0 & 0 \\
 \hline
 $T^{\perp}_{\mu\nu\rho}$ & $\mathfrak{Y}^{(\lambda)}\bm{\TT}^{\perp}_{\mu\nu\rho}$ & $\mathfrak{Y}^{(\lambda)}\bm{\TT}^{\perp}_{\mu\nu\rho}$ & $\mathfrak{Y}^{(\lambda)}\bm{\TT}^{\perp}_{\mu\nu\rho}$  & $\mathfrak{Y}^{(\lambda)}\bm{\TT}^{\perp}_{\mu\nu\rho}$  & $\mathfrak{Y}^{(\lambda)}\bm{\TT}^{\perp}_{\mu\nu\rho}$ & $\mathfrak{Y}^{(\lambda)}\bm{\TT}^{\perp}_{\mu\nu\rho}$  \\
 \hline
\end{tabular}
\caption{Relations between diagonals of the coefficients with three indices. Independent diagonals are indicated in \textbf{bold}.  Symbol $\mathfrak{Y}^{(\lambda)}$ stands for symmetrization of the tensor with appropriate partition $\lambda$.}
\label{tab:DiagonalsE0}
\end{center}
\end{table}

The last group of coefficients containing two $\hat{e}^0$s is the simplest.
We have the remaining off-diagonal parts $\TT^{(2,0,2)\perp}_{\mu\nu}$, $\TT^{(0,2,2)\perp}_{\mu\nu}$, $\TT^{(1,1,2)\perp}_{\mu\nu}$, and their independent diagonals $\TT^{(2,0,2)(0)}_{\mu}$, $\TT^{(0,2,2)(0)}_{\mu}$, $\TT^{(1,1,2)(0)}_{\mu}$.

\subsubsection*{Parametrization by independent diagonals}

In section \ref{sec:RiemannWeightSubspaces} we mentioned that the parametrization of the Riemann-symmetric tensor \eqref{eq:DecompositionOfRtoWeightSectors} can be viewed as a linear map, which takes a tuple of coefficients $\TT^{(\lambda)}$ and builds the tensor \eqref{eq:DecompositionOfRtoWeightSectors}. Here we explore the parametrization further, taking into account the diagonal decomposition of $\TT$s and mutual relations between diagonal parts due to their symmetries studied in the previous section. 

We want to apply to each $\TT^{(\lambda)}$ coefficient the diagonal decomposition implemented by the $\mathfrak{D}^k$ and $\bar{\mathfrak{D}}^k$ defined in section~\ref{sec:DecompositionIntoDiagonals}, where $k$ is the number of indices of the tensor. We will denote the corresponding maps by $\mathfrak{D}^{(\lambda)} \colon \mathcal{D}_{(\lambda)} \to \mathfrak{T}_{(\lambda)}$ and $\bar{\mathfrak{D}}^{(\lambda)} \colon \mathfrak{T}_{(\lambda)} \to \mathcal{D}_{(\lambda)}$. In one direction it is enough to set $\bar{\mathfrak{D}}^{(\lambda)} = \bar{\mathfrak{D}}^{k}|_{\mathfrak{T}_{(\lambda)} \subset \mathbb{T}^{k}}$, while in the other direction we project to the right subspace by a Young symmetrizer, $\mathfrak{D}^{(\lambda)} = \mathfrak{Y}^{(\lambda)} \circ \mathfrak{D}^{k}$. Since no symmetry properties are assumed in the construction of the domain of $\mathfrak{D}^k$, the domains $\mathcal{D}_{(\lambda)}$ corresponding to tensors with the same number of indices are isomorphic, even if the $\lambda$ labels differ.

Taking into account the explicit structure of the components of the $\mathfrak{P}^R$ and $\bar{\mathfrak{P}}^R$ maps \eqref{eq:ICProjectors22}, the Young symmetrizers have to be defined as
\begin{align*}
	\mathfrak{Y}^{(22,0,0)}
		&= P_{\,\young(\mu\rho,\nu\sigma)} \otimes \mathrm{Id}
	, &
	\mathfrak{Y}^{(21,1,0)}
		&= P_{\,\young(\mu\rho,\nu)} \otimes P_{\,\young(\sigma)}
	, &
	\mathfrak{Y}^{(2,2,0)}
		&= P_{\,\young(\mu\rho)} \otimes P_{\,\young(\nu\sigma)}
	, \\
	\mathfrak{Y}^{(00,22,0)}
		&= \mathrm{Id} \otimes P_{\,\young(\mu\rho,\nu\sigma)}
	, &
	\mathfrak{Y}^{(1,21,0)}
		&= P_{\,\young(\nu)} \otimes P_{\,\young(\mu\rho,\sigma)}
	, &
	\mathfrak{Y}^{(11,11,0)}
		&= P_{\,\young(\mu,\nu)} \otimes P_{\,\young(\rho,\sigma)}
	, \\
	\mathfrak{Y}^{(21,0,1)}
		&= P_{\,\young(\mu\rho,\nu)} \otimes \mathrm{Id}
	, &
	\mathfrak{Y}^{(2,1,1)}
		&= P_{\,\young(\mu\rho)} \otimes P_{\,\young(\nu)}
	, &
	\mathfrak{Y}^{(11,1,1)}
		&= P_{\,\young(\mu,\nu)} \otimes P_{\,\young(\rho)}
	, \\
	\mathfrak{Y}^{(0,21,1)}
		&= \mathrm{Id} \otimes P_{\,\young(\mu\rho,\nu)}
	, &
	\mathfrak{Y}^{(1,2,1)}
		&= P_{\,\young(\nu)} \otimes P_{\,\young(\mu\rho)}
	, &
	\mathfrak{Y}^{(1,11,1)}
		&= P_{\,\young(\rho)} \otimes P_{\,\young(\mu,\nu)}
	, \\
	\mathfrak{Y}^{(2,0,2)}
		&= P_{\,\young(\mu\rho)} \otimes \mathrm{Id}
	, &
	\mathfrak{Y}^{(0,2,2)}
		&= \mathrm{Id} \otimes P_{\,\young(\mu\rho)}
	, &
	\mathfrak{Y}^{(1,1,2)}
		&= P_{\,\young(\mu)} \otimes P_{\,\young(\rho)}
	.
\end{align*} 
We combine everything into a new map
\begin{equation} \label{eq:D-map}
	\mathfrak{D} = \bigoplus_{(\lambda)} \mathfrak{D}^{(\lambda)}
		\colon \bigoplus_{(\lambda)} \mathcal{D}_{(\lambda)}
			= \bigoplus_{(\lambda)} \bigoplus_{(d)} \mathcal{D}_{(\lambda),(d)}
		\to \bigoplus_{(\lambda)} \mathfrak{T}_{(\lambda)}
			= \mathfrak{T}^R ,
\end{equation}
which takes all tuples of diagonals of all coefficients $\TT^{(\lambda)}$, now individually labelled $\mathcal{D}_{(\lambda)(d)}$, and maps them to a direct sum of tensors (of different ranks and symmetries). This can then be composed with $\mathfrak{P}^R$ to give the parametrization 
\begin{equation} \label{eq:PD-map}
\mathfrak{P}^R\circ \mathfrak{D}\colon
	\bigoplus_{(\lambda)} \mathcal{D}_{(\lambda)} \rightarrow \mathfrak{R},
\end{equation}
which represents the formula \ref{sec:RiemannWeightSubspaces} written in diagonal (and off-diagonal) terms of the coefficients $\TT^{(\lambda)(d)} \in \mathcal{D}_{(\lambda),(d)}$.
 
However now, because of the Young symmetrization, not all diagonals are independent and the maps $\mathfrak{D}^{(\lambda)}$ and $\bar{\mathfrak{D}}^{(\lambda)}$ are no longer bijections (instead they are respectively surjective and injective). To capture the independent parts summarized in the tables \ref{tab:Diagonals}, \ref{tab:DiagonalsE0},  we define a map, analogous to $\mathfrak{D}^r$
\begin{equation} \label{eq:J-map}
	\mathfrak{J}^{(\lambda)}: \mathcal{J}_{(\lambda)} = \bigoplus_{(d)} \mathcal{J}_{(\lambda),(d)} \rightarrow \mathcal{D}_{(\lambda)} ,
	\quad \text{and} \quad
	\mathfrak{J} = \bigoplus_{(\lambda)} \mathfrak{J}^{(\lambda)}
		\colon \bigoplus_{(\lambda)} \mathcal{J}_{(\lambda)}
			\to \bigoplus_{(\lambda)} \mathcal{D}_{(\lambda)},
\end{equation}
which take the \emph{independent} diagonals of $\TT^{(\lambda)}$s and maps them to \emph{all} diagonals using the prescriptions in tables  \ref{tab:Diagonals}, \ref{tab:DiagonalsE0}. In the case of two $\hat{e}^{0}$s it is an identity, because there is only one diagonal, which is always independent. 

For instance,
\begin{align*}
\mathfrak{J}^{(22,0,0)} &\colon \mathcal{J}_{(22,0,0)(6)} \oplus \mathcal{J}_{(22,0,0)(1,3)} \oplus \mathcal{J}_{(22,0,0)\perp} \to \mathcal{D}_{(22,0,0)}
\end{align*}
maps the off-diagonal component $\TT^{(22,0,0)\perp}$ to itself and the independent parts $\TT^{(22,0,0)(6)}_{(\mu\nu)}, \TT^{(22,0,0)(1,3)}_{(\mu\nu)\rho}$ to the set of all diagonals by the formulas listed in the first column of the table \ref{tab:Diagonals}, namely
\[
	\mathfrak{J}^{(22,0,1)} \colon
	\begin{pmatrix}
		\TT^{(22,0,0)(6)}_{(\mu\nu)} \\
		\TT^{(22,0,0)(1,3)}_{(\mu\nu)\rho} \\
		\mathfrak{Y}^{(22,0,0)} \TT^{(22,0,0)\perp}_{\mu\nu\rho\sigma}
	\end{pmatrix}
	\mapsto
	\begin{pmatrix}
		T^{(0)}_\mu \\
		T^{(1)}_{\mu\nu} \\
		T^{(2)}_{\mu\nu} \\
		T^{(3)}_{\mu\nu} \\
		T^{(4)}_{\mu\nu} \\
		T^{(5)}_{\mu\nu} \\
		T^{(6)}_{\mu\nu} \\
		T^{(7)}_{\mu\nu} \\
		T^{(1,2)}_{\mu\nu\rho} \\
		T^{(1,3)}_{\mu\nu\rho} \\
		T^{(1,4)}_{\mu\nu\rho} \\
		T^{(2,3)}_{\mu\nu\rho} \\
		T^{(2,4)}_{\mu\nu\rho} \\
		T^{(3,4)}_{\mu\nu\rho} \\
		T^{\perp}_{\mu\nu\rho\sigma}
	\end{pmatrix}
	=
	\begin{pmatrix}
		0 \\
		0 \\
		0 \\
		0 \\
		0 \\
		0 \\
		\TT^{(22,0,0)(6)}_{(\mu\nu)} \\
		-\TT^{(22,0,0)(6)}_{(\mu\nu)} \\
		0 \\
		\TT^{(22,0,0)(1,3)}_{(\mu\nu)\rho} \\
		-\TT^{(22,0,0)(1,3)}_{\mu\nu\rho} \\
		-\TT^{(22,0,0)(1,3)}_{\mu\nu\rho} \\
		\TT^{(22,0,0)(1,3)}_{(\mu\nu)\rho} \\
		0 \\
		\mathfrak{Y}^{(22,0,0)} \TT^{(22,0,0)\perp}_{\mu\nu\rho\sigma}
	\end{pmatrix}
	.
\]

Altogether, we have defined a convenient parametrization of Riemann-symmetric tensors:

\begin{lemma}[Well-defined parametrization] \label{lem:Parametrization}
Consider maps $\mathfrak{P}^R$, $\mathfrak{D}$, $\mathfrak{J}$ described above in equations \eqref{eq:P-map}, \eqref{eq:D-map}, \eqref{eq:J-map}, whose composition
\[
	\mathfrak{P}^R\circ\mathfrak{D}\circ\mathfrak{J}\colon \bigoplus_{(\lambda)} \mathcal{J}_{(\lambda)} = \bigoplus_{(\lambda)} \bigoplus_{(d)} \mathcal{J}_{(\lambda),(d)} \rightarrow \mathfrak{R}
\]
maps the independent parameters $\TT^{(\lambda)(d)} \in \mathfrak{J}_{(\lambda),(d)}$ to a Riemann-symmetric tensor \eqref{eq:DecompositionOfRtoWeightSectors}, ultimately using the formulas \eqref{eq:DiagonalsT}, \eqref{eq:DecT} and \eqref{eq:Decomposition4Tensor}. The map $\mathfrak{P}^R\circ\mathfrak{D}\circ\mathfrak{J}$ is a \emph{bijection}. 
\end{lemma}
\begin{proof}
When it was introduced, we noted that the map $\mathfrak{P}^R$ is a bijection, hence it is enough to show that $\mathfrak{D} \circ \mathfrak{J}$ is a bijection. It will be sufficient to show that $\mathfrak{D} \circ \mathfrak{J}$ is injective, with bijectivity following from the equality of the dimensions of its domain and codomain.

The composition $\mathfrak{D} \circ \mathfrak{J}$ is somewhat complicated to analyze directly. Instead, recalling that the linear maps $\bar{\mathfrak{D}}^{(\lambda)}$ extracting all the diagonals are injective, cf.\ the discussion before~\eqref{eq:J-map}, it is sufficient to prove that $(\bigoplus_{(\lambda)} \bar{\mathfrak{D}}^{(\lambda)}) \circ \mathfrak{D} \circ \mathfrak{J}$ is injective. The advantage of that last composition is that it comes out to be rather sparse%
	\footnote{All the maps involved intertwine the action of the $U(1)^n \rtimes S_n$ stabilizer group and hence its $S^n$ subgroup. So their composition would be exactly block diagonal if the decomposition into diagonals by the operators~$\mathfrak{D}^k$ was fine enough to result in $S^n$-irreducible subspaces. But that would have required projecting out all the contractions with the $v_\mu$ isotropic vector, not just the restrictions to diagonals, thus greatly increasing the complexity of the decomposition. The decomposition that we have chosen to use merely resulted in a sufficiently simple sparse composition.} %
in block form. The total number of block equations to check $15\times 15 = 225$ is quite large (the codomain $\bigoplus_{(\lambda)} \mathfrak{T}_{(\lambda)}$ of $\mathfrak{D}$ has 15 summands, each of which is decomposed into 15 diagonals by $\bar{\mathfrak{D}}^{(\lambda)}$), but each equation is fairly simple. One example is
\[
	\TT^{(22,0,0)(2)}_{\sigma\rho} + \TT^{(22,0,0)(0)}_\alpha \Delta^\alpha_{\sigma\rho} = 0 .
\]
Starting with the simplest equations with the lowest number of indices, we can easily show that the trivial solution is the only possible one, implying that $\ker ((\bigoplus_{(\lambda)} \bar{\mathfrak{D}}^{(\lambda)}) \circ \mathfrak{D} \circ \mathfrak{J}) = 0$ and hence that $\mathfrak{D} \circ \mathfrak{J}$ is injective. That is, all constraints among the diagonals, passing from the map $\mathfrak{D}^4$~\eqref{eq:D-maps} to $\mathcal{D}$~\eqref{eq:D-map}, were already exhausted by the construction of the map $\mathfrak{J}$, reflected in relations in tables \ref{tab:Diagonals}, \ref{tab:DiagonalsE0}.
\end{proof}

\subsection{Parametrization of the zero weight subspace} \label{sec:ZeroWeightSubspace}

Formula \eqref{eq:DecompositionOfRtoWeightSectors} represents a general decomposition of tensor with (2,2) Young symmetry into subspaces with different weights. Here, we look more closely at the subspace with \emph{(global) zero weight}, its basis and significance for our investigations. 

The most general Riemann-symmetric tensor of zero weight is parametrized by three $\TT$ coefficients and reads
\begin{align} \label{eq:ZeroWeightRiemann}
\RR_{abcd} &= \TT^{(2,2,0)}_{\mu\nu\rho\sigma} (E^{(2,2,0)})^{\mu\nu\rho\sigma}_{abcd} + \TT^{(11,11,0)}_{\mu\nu\rho\sigma} (E^{(11,11,0)})^{\mu\nu\rho\sigma}_{abcd} + \TT^{(1,1,2)}_{\mu\nu} (E^{(1,1,2)})^{\mu\nu}_{abcd}.
\end{align}
In section \ref{sec:RiemannInBasis} we show that the Kerr-NUT-(A)dS spacetime Riemann tensor belongs to this subspace. Furthermore, in section \ref{sec:TwoFormFromRiemann}, we show implications of this form for construction of potential candidates on principal tensor and IDEAL characterization. First, we have a look at a special basis of this subspace, which is more useful in our calculations and is directly related to the principal tensor itself.

\subsubsection*{Construction of the $h$-basis in $D$ dimensions} \label{sec:hBasis}

For the construction of the Riemann-symmetric basis we decompose the principal tensor in the Darboux basis \eqref{eq:PrincipalTensor} or \eqref{eq:BasisEpm} into separate parts according to the block diagonal structure 
\begin{align} \label{eq:DecomposedH}
h = \sum_\mu \, x_\mu h_\mu, \quad
h_\mu \equiv \frac{\ii}{2} e^{+\mu} \wedge e^{-\mu} = e^\mu \wedge \hat{e}^\mu. 
\end{align}

Analogously to $h$, we also decompose the conformal Killing tensor
\begin{align} \label{eq:DecomposedQ}
Q = \sum_\mu x_{\mu}^2 Q_\mu, \quad
Q_\mu \equiv e^\mu \otimes e^\mu + \hat{e}^\mu \otimes \hat{e}^\mu
	= \tfrac{1}{2} (e^{+\mu} \otimes e^{+\mu} + e^{-\mu} \otimes e^{-\mu}) .
\end{align}
In the case of odd dimensions $\epsilon \neq 0$, $h$ does not contain the vector $\hat{e}^0$, so we have to include it into the basis in a different way. A simple choice is
\begin{align*}
Q_\epsilon \equiv \varepsilon \, g.
\end{align*}
Another possible choice is just $Q_\epsilon \equiv \varepsilon \, \hat{e}^0 \otimes \hat{e}^0$, but we will use the first possibility. 

To ensure proper symmetries, especially the first Bianchi identity, tensor products of the basis tensors have to be \emph{Young-symmetrized}, because the Riemann tensor has Young symmetry (22). For instance $h_{ab}h_{cd}$ has correct monoterm index symmetries, but does not fulfill the Bianchi identity, which is a multiterm symmetry. Hence we define Young-symmetrized (YS) product for antisymmetric tensors\footnote{The product $\odot$ differs from the projector $P_{(22)}$ defined by \eqref{eq:YoungSymmetrizer} just by an overall numerical factor.}
\begin{align} \label{eq:ODotProd1}
(h \odot f)_{abcd} = \tfrac{1}{4} h_{ab} f_{cd} + \tfrac{1}{4} f_{ab} h_{cd} + \tfrac{1}{8} h_{ac} f_{bd} + \tfrac{1}{8} f_{ac} h_{bd} - \tfrac{1}{8} h_{ad} f_{bc} - \tfrac{1}{8} f_{ad} h_{bc},
\end{align}
and for symmetric tensors (aka the \emph{Kulkarni-Nomizu} product)
\begin{align} \label{eq:ODotProd2}
(Q \odot k)_{abcd} = \tfrac{1}{8} Q_{ad} k_{bc} + \tfrac{1}{8} k_{ad} Q_{bc} - \tfrac{1}{8} Q_{ac} k_{bd} - \tfrac{1}{8} k_{ac} Q_{bd},
\end{align}
while the YS product of a symmetric and an antisymmetric tensors identically vanishes. 
When $h=f$ or $Q=k$ the formulae further simplify.

We define the following set of tensors in $D$ dimensions
\begin{align} \label{eq:hBasis}
\mathcal{B}_D \equiv \{(h_\mu \odot h_\mu),\, (h_\mu \odot h_\nu),\, (Q_\mu \odot Q_\nu),\, (Q_\mu \odot Q_\epsilon)\}, \quad \forall \mu \neq \nu.
\end{align}
and $\mathfrak{R}_0$ the linear space (of Riemann-symmetric tensors) spanned by this set.
Since every element contains different combinations of vectors $e^\mu, \hat{e}^\mu, \hat{e}^0$, they are manifestly linearly independent. Inclusion of terms $(Q_\mu\odot Q_\mu)$ does not bring any new information, because
\begin{align} \label{eq:RelationBetweenDiagonalQQAndHH}
(Q_\mu\odot Q_\mu) = -\frac{1}{3}(h_\mu\odot h_\mu).
\end{align}
Counting the number of elements in $\mathcal{B}_D$ gives the dimension of $\mathfrak{R}_0$ \begin{align*}
\mathrm{dim}(\mathfrak{R}_0) = n + \frac{n^2-n}{2} + \frac{n^2-n}{2} + \varepsilon n = n^2 + \varepsilon n.
\end{align*}

For future reference, let us state here lemmas which govern various contractions of the basis tensors. They can be easily checked by direct calculation, so we omit explicit proofs for brevity.  

\begin{lemma} \label{lem:ProductsOfHsQs}
Products of $h$s and $Q$s. 
\begin{align} 
\begin{alignedat}{4} \label{eq:ProductsOfHsQs}
h_{\mu\, a}^{~~~c} h_{\nu\, bc} &= \delta_{\mu \nu} Q_{\mu\, ab} & \qquad\qquad\qquad
Q_{\mu\, a}^{~~~c} h_{\nu\, bc} &= - \delta_{\mu \nu} h_{\mu\, ab} \\
Q_{\mu\, a}^{~~~c} Q_{\nu\, bc} &= \delta_{\mu \nu} Q_{\mu\, ab} &
Q_{\epsilon\, a}^{~~~c} h_{\mu\, bc} &= - h_{\mu\, ab} \\
Q_{\epsilon\, a}^{~~~c} Q_{\mu\, bc} &= Q_{\mu\, ab} &
Q_{\epsilon\, a}^{~~~c} Q_{\epsilon\, bc} &= Q_{\epsilon\, ab} \\
h_{a}^{~c} h_{\mu\, bc} &= x_\mu Q_{\mu\, ab} & 
h_{a}^{~c} Q_{\mu\, bc} &= x_\mu h_{\mu\, ab} \\
Q_{a}^{~c} h_{\mu\, bc} &= - x_\mu^2 h_{\mu\, ab} & 
Q_{a}^{~c} Q_{\mu\, bc} &= x_\mu^2 Q_{\mu\, ab}
\end{alignedat}  
\end{align}
\end{lemma}
In the calculations needed below, it is useful to use the decomposition $Q_\epsilon = \sum_\mu Q_\mu + e^0 e^0$ and the bilinearity of the $\odot$-product.
\begin{lemma} \label{lem:TracesOfHsQs}
Traces of $\mathcal{B}_D$ basis tensors.
\begin{align} \label{eq:TracesOfHsQs}
(h_{\mu }\odot h_\nu)^i_{~aib} &= \frac{3}{4} \delta_{\mu \nu} Q_{\mu ab} &
(h_{\mu }\odot h_\nu)^{ij}_{~~ij} &= \frac{3}{2} \delta_{\mu \nu} \nonumber \\
(Q_{\mu }\odot Q_\nu)^i_{~aib} &= \frac{1}{4} \delta_{\mu \nu} Q_{\mu ab} - \frac{1}{4}( Q_{\mu ab} + Q_{\nu ab})  &
(Q_{\mu }\odot Q_\nu)^{ij}_{~~ij} &= \frac{1}{2} \delta_{\mu \nu} - 1 \\
(Q_{\mu }\odot Q_\epsilon)^i_{~aib} &= -\frac{D-2}{8} Q_{\mu ab} -\frac{1}{4} Q_{\epsilon ab} &
(Q_{\mu }\odot Q_\epsilon)^{ij}_{~~ij} &= -\frac{D-1}{2} \nonumber 
\end{align}
\end{lemma}

\begin{lemma} \label{lem:ProductsOfYSHsQs}
Contractions of $\mathcal{B}_{D}$ tensors with $h$. 
\begin{align} \label{eq:ProductsOfYSHsQs}
(h_{\mu }\odot h_\nu)_{abce} h_{\rho ~d}^{~e} &= -\frac{1}{8}\delta_{\mu \rho} \left( Q_{\rho\, ad} h_{\nu\, cb} + Q_{\rho\, bd} h_{\nu\, ac} + 2 Q_{\rho\, cd} h_{\nu\, ab} \right) \nonumber \\
& \quad - \frac{1}{8}\delta_{\nu \rho} \left( Q_{\rho\, ad} h_{\mu\, cb} + Q_{\rho\, bd} h_{\mu\, ac} + 2 Q_{\rho\, cd} h_{\mu\, ab} \right) \nonumber \\
(h_{\mu }\odot h_\nu)_{abef} h_{\rho }^{~ef} &= \frac{1}{2} \delta_{\mu \rho} \delta_{\nu \rho} h_{\rho\, ab} + \frac{1}{2} \delta_{\mu \rho} h_{\nu \, ab} + \frac{1}{2} \delta_{\nu \rho} h_{\mu \, ab} \nonumber \\
(Q_{\mu }\odot Q_\nu)_{abce} h_{\rho ~d}^{~e} &= - \frac{1}{8} \delta_{\nu \rho} \left( Q_{\mu\, ac} h_{\rho\, bd} - Q_{\mu\, bc} h_{\rho\, ad} \right) - \frac{1}{8} \delta_{\mu \rho} \left( Q_{\nu\, ac} h_{\rho\, bd} - Q_{\nu\, bc} h_{\rho\, ad} \right) \nonumber \\
(Q_{\mu }\odot Q_\nu)_{abef} h_{\rho}^{~ef} &= - \frac{1}{2} \delta_{\mu \rho} \delta_{\nu \rho} h_{\rho\, ab} \\
(Q_{\mu }\odot Q_\epsilon)_{abce} h_{\rho ~d}^{~e} &= -\frac{1}{8} (Q_{\mu\, ac} h_{\rho\, bd} - Q_{\mu\, bc} h_{\rho\, ad}) - \frac{1}{8} \delta_{\mu \rho}(Q_{\epsilon\, ac} h_{\rho\, bd} - Q_{\epsilon\, bc} h_{\rho\, ad}) \nonumber \\
(Q_{\mu }\odot Q_\epsilon)_{abef} h_{\rho}^{~ef} &= -\frac{1}{2} \delta_{\mu \rho} h_{\rho\, ab} \nonumber 
%
%
\end{align}
\end{lemma}

We mentioned that $h$ is invariant under $U(1)^n$. From the derivation \eqref{eq:DerOfZeroWeight} it is obvious that even the $h_\mu$ alone are invariant. Since $Q_\mu$ are second powers of $h_\mu$, they are invariant too. All tensors from the set \eqref{eq:hBasis} thus have totally zero weight and they form a subset of tensors of global zero weight, which are of the form \eqref{eq:ZeroWeightRiemann}. A linear combination of tensors \eqref{eq:hBasis} is a Riemann-symmetric tensor, which plays a crucial role in our investigations.

\begin{theorem} \label{the:ZeroWeightBasis}
Every Riemann-symmetric tensor that has totally zero weight under $U(1)^n$ transformations (definition~\ref{def:Weights}) is expressible in the adapted  h-basis \eqref{eq:hBasis}.
\end{theorem}
\begin{proof}
A generating set of terms for a Riemann-symmetric tensor is obtained by
applying $(22)$ Young symmetrization to monomials of the form $\tilde{e}^{\mu}
\tilde{e}^{\nu} \tilde{e}^{\rho} \tilde{e}^{\sigma}$, where $\tilde{e}^{\mu}$ could stand for $e^{+\mu}$ or $e^{-\mu}$.
Thus the result follows from the following identities valid for any $\mu, \nu$ that can be checked by direct calculation:
\begin{align*}
P_{\, \young(ac,bd)} \Big(e^{+\mu}_a e^{-\mu}_b e^{+\nu}_c e^{-\nu}_d\Big) &= \frac{4}{3}\Big( (Q_\mu \odot Q_\nu)_{abcd} - (h_\mu \odot h_\nu)_{abcd} \Big), \\
P_{\, \young(ac,bd)}  \Big(e^{+\mu}_a e^{+\nu}_b e^{-\mu}_c e^{-\nu}_d\Big) &= - \frac{8}{3}(Q_\mu \odot Q_\nu)_{abcd}, \\
P_{\, \young(ac,bd)} \Big(e^{+\mu}_a e^{-\mu}_b \hat{e}^{0}_c \hat{e}^{0}_d\Big) &= \frac{4}{3} \Big( (Q_\mu \odot Q_\epsilon)_{abcd} - \sum_\alpha (Q_\mu \odot Q_\alpha)_{abcd} \Big).
\end{align*}
All other symmetrized monomials reduce to one of the possibilities on the left-hand side (the terms with $\hat{e}^0_a$ appear only in odd
dimensions). The identities show that this generating set, and hence all
Riemann-symmetric tensors, can be expressed in terms of the
$h$-basis~\eqref{eq:hBasis}. Note that when $\mu=\nu$ the first two projections
degenerate because of relation~\eqref{eq:RelationBetweenDiagonalQQAndHH}.
\end{proof}

\begin{remark} \label{rem:CommonU1AngleNecessity}
If we considered $U(1)$ transformations by a common angle $e^{\pm \mu} \mapsto e^{\ii \fii} e^{\pm \mu}$, for each $\mu$, the $+$ and $-$ vectors in the basis monomials could belong to different subspaces, namely $e^{+\mu} e^{+\nu} e^{-\rho} e^{-\sigma}$. In such a case theorem \ref{the:ZeroWeightBasis} would not be valid anymore, thus $U(1)^n$ transformations are the key ingredient. 
\end{remark}

From theorem \ref{the:ZeroWeightBasis} it follows that a totally zero weight Riemann tensor is a linear combination of tensors from $\mathcal{B}_D$
\begin{multline}
\label{eq:RiemannInBasis}
R = \sum_{\substack{\mu,\nu \\ \mu<\nu}} \left[ \Omega_{\mu \nu} (h_\mu \odot h_\nu) + \alpha_{\mu \nu} (Q_\mu \odot Q_\nu) \right]
	\\
	+ \sum_{\mu} \left[ \tfrac{1}{2}\Omega_{\mu \mu} (h_\mu \odot h_\mu) + \, \alpha_{\mu}^{(\epsilon)} (Q_\mu \odot Q_\epsilon) \right] ,
\end{multline}
where $\Omega_{\mu\nu}, \alpha_{\mu\nu}$ and $\alpha_{\mu}^{(\epsilon)}$ are (undetermined) coefficient functions. We define $\Omega_{\mu\nu}, \alpha_{\mu\nu}$ to be symmetric matrices with respect to the Greek indices. We include also $\alpha_{\mu\mu}$s in the matrix, even though the formula \eqref{eq:RiemannInBasis} does not contain them.

\subsubsection*{Relation to the global zero weight subspace}

\begin{lemma} \label{lem:RelationBetweenParametrizations}
Consider the parametrization of the Riemann tensor of global zero weight \eqref{eq:ZeroWeightRiemann}.
The totally zero weight subspace~\eqref{eq:RiemannInBasis} corresponds to the sum of the diagonal subspaces (cf.~section~\ref{sec:DecompositionIntoDiagonals}) with coefficients $\TT^{(2,2,0)(0)}_\mu, \TT^{(2,2,0)(5)}_{\mu\nu},\TT^{(11,11,0)(6)}_{\mu\nu}$, plus $\TT^{(1,1,2)(0)}_\mu$ in odd dimensions.
The mutual relations between both sets of parameters are $(\mu \neq \nu)$
\begin{align}
\begin{alignedat}{2} \label{eq:RelationBetweenParametrizations}
\TT^{(2,2,0)(0)}_{\mu} &= \frac{3}{32}(2\alpha^{(\epsilon)}_\mu - 3\Omega_{\mu\mu}), &
\TT^{(2,2,0)(5)}_{\mu\nu} &= \frac{3}{64}(\alpha_{\mu\nu} - \frac{3}{2}\Omega_{\mu\nu} + \alpha^{(\epsilon)}_\mu + \alpha^{(\epsilon)}_\nu), \\
\TT^{(1,1,2)(0)}_{\mu} &= -\frac{3}{8} \alpha^{(\epsilon)}_\mu, &
\TT^{(11,11,0)(6)}_{\mu\nu} &= -\frac{3}{32}(\alpha_{\mu\nu} + \Omega_{\mu\nu} + \alpha^{(\epsilon)}_\mu + \alpha^{(\epsilon)}_\nu).
\end{alignedat}
\end{align}
\end{lemma}
\begin{proof}
A long, but straightforward calculation. We take the formula \eqref{eq:ZeroWeightRiemann} and use the relations \eqref{eq:RelationBetweenParametrizations} and definitions of embeddings from the formula \eqref{eq:DecompositionOfRtoWeightSectors}. By a suitable rearranging of terms, using definitions of $h_\mu, Q_\mu, Q_\epsilon$ we arrive at formula \eqref{eq:RiemannInBasis}.
\end{proof}

Taking into account the number of independent diagonals of all global zero weight Riemann-symmetric tensors in table \ref{tab:Diagonals}, we see that its subspace of totally zero weight tensors is much smaller. Compared to the general decomposition \eqref{eq:DecompositionOfRtoWeightSectors}, formula \eqref{eq:RiemannInBasis} represents a very limited number of possibilities.

\subsection{The Riemann tensor in the $h$-basis} \label{sec:RiemannInBasis}

In the previous section we discussed a linear space of Riemann-symmetric tensors of global zero weight and introduced its special subspace of totally zero weight tensors, which are described by the adapted $h$-basis \eqref{eq:hBasis} constructed from the principal tensor. Its physical relevance is justified in the following theorem.

\begin{theorem} \label{the:RiemannInBasis}
Let $R_{abcd}$ be the Riemann curvature tensor of Kerr-NUT-(A)dS
spacetime~\eqref{eq:HigherDimensionalKNAdSmetric}, with $h_{ab}$ its principal tensor~\eqref{eq:PrincipalTensor}. Then $R_{abcd}$ belongs to the subspace
of totally zero weight Riemann-symmetric tensors \eqref{eq:RiemannInBasis},
namely
\begin{align*}
R &= \sum_{\substack{\mu,\nu \\ \mu<\nu}} \left[ \Omega_{\mu \nu} (h_\mu \odot h_\nu) + \alpha_{\mu \nu} (Q_\mu \odot Q_\nu) \right] + \sum_{\mu} \left[ \tfrac{1}{2}\Omega_{\mu \mu} (h_\mu \odot h_\mu) + \, \alpha_{\mu}^{(\epsilon)} (Q_\mu \odot Q_\epsilon) \right].
\end{align*}
The coefficients $\Omega_{\mu\nu}, \alpha_{\mu\nu}$ and $\alpha_{\mu}^{(\epsilon)}$ are explicitly expressed as $(\mu\neq\nu)$
\begin{align}
\begin{alignedat}{2} \label{eq:BasisCoefsToHamFcions}
\Omega_{\mu \mu} &= \frac{8}{3}(D_\mu - 2\varepsilon \, E_\mu), & \qquad\qquad\qquad
\alpha_{\mu \nu} &= -8 S_{\mu \nu} + 8 \varepsilon \, (E_\mu + E_\nu), \\
\Omega_{\mu \nu} &= 8 F_{\mu \nu} & \qquad\qquad\qquad
\alpha_{\mu}^{(\epsilon)} &= -8 \varepsilon \, E_\mu, 
\end{alignedat}
\end{align}
in terms of the functions $S_{\mu \nu}, F_{\mu \nu}, D_{\mu}, E_{\mu}$
defined in \eqref{eq:CurvatureFunctions}.
\end{theorem}
\begin{proof}
The proof is a straightforward calculation, so we omit it here for brevity. It can be found in \cite[Theorem 2]{MatejovThesis}.
\end{proof}

The relation \eqref{eq:RiemannInBasis} represents a simple compact form of the Riemann tensor, which captures its special structure following from existence of the principal tensor. Notice that \eqref{eq:RiemannInBasis} contains only combinations of the special products $e^\mu\wedge \hat{e}^\mu, e^\mu\otimes e^\mu, \hat{e}^\mu\otimes \hat{e}^\mu$, via $h_\mu$ and $Q_\mu$, while curvature two forms $\RR_{\mu\nu}$ \eqref{eq:HamamotoFormulas} contain all other combinations like $e^\mu \wedge \hat{e}^\nu$, etc.\ as well. Thus, it is not obvious at first sight why $R \in \mathfrak{R}_0$ at all. But in the end, the terms from $\RR_{\mu\nu}$ miraculously combine into \eqref{eq:RiemannInBasis}. 

\begin{remark} \label{rem:TypeDConditionAnalogy}
The successful IDEAL characterization of the Kerr metric \cite{FerrandoSaez2009} relied heavily on the type~D condition for the Weyl tensor in 4D written with the help of the self-dual canonical bivector $\mathcal{U}$. In higher dimensions no analogous formula is known. In \cite[section 2.1.5]{MatejovThesis} we demonstrated that $\mathcal{U}$ is directly related to the decomposed principal tensor parts. Perhaps, the formula \eqref{eq:RiemannInBasis} could be regarded as a generalization of the type~D condition in 4D.
\end{remark}

\section{Construction of a 2-form from the Riemann tensor} \label{sec:TwoFormFromRiemann}

If we could covariantly construct the Kerr-NUT-(A)dS principal tensor $h_{ab}$ directly from the curvature, due to the strong uniqueness properties from proposition \ref{th:Uniqueness}, we would already achieve an IDEAL characterization of the Kerr-NUT-(A)dS family.
Thus, it is worth exploring the possibility of construction a two-form from the Riemann tensor. It is well-know that contraction of two indices of the Riemann tensor yields zero or the Ricci tensor (up to a sign). Thus, we cannot construct a two-form directly from itself. Thus, at the very least, we must employ more than one copy of the Riemann tensor. After some attempts, any 2-form that we were able to build in this way turned out to be identically zero, which we were able to explain by the following no-go theorem:
\begin{theorem} \label{th:NoTwoFormFromRiemann}
Let $R_{abcd}$ be the Riemann curvature tensor of totally zero weight \eqref{eq:RiemannInBasis} and $H_{ab}$ a two-form constructed from its covariant monomials (tensor products of any number of $R_{abcd}$ with various indices contracted by the metric). Then $H_{ab} \equiv 0$. 
\end{theorem}
\begin{proof}
The Riemann tensor of totally zero weight can be expressed in the adapted basis \eqref{eq:RiemannInBasis} according to theorem \ref{the:ZeroWeightBasis}. It is solely composed of $h_{\mu\, ab}, \, Q_{\mu\, ab}, \, Q_{\epsilon\, ab}$ and their tensor products, so is every monomial $R_{a...} \dots R_{b...}$, prior to contractions. According to \eqref{eq:ProductsOfHsQs} every contraction of these tensors results again in $h_{\mu\, ab},\, Q_{\mu\, ab}$ or $Q_{\epsilon\, ab}$. Thus, taking antisymmetry into account, every two-form built from monomials of $R_{abcd}$ must be of the form
\begin{align}
H_{ab} = \sum_{\mu} C_\mu h_{\mu\, ab} ,
\end{align}
for some coefficients $C_\mu$.
On the other hand, every element of adapted $h$-basis~\eqref{eq:hBasis} contains either 0 or 2 $h_\mu$s, so the total number of $h_\mu$s in a monomial is even.
The $Q\cdot h$ contractions keep the number of $h_\mu$'s the same, while $h_{\mu, a}{}^a = 0$ and $h_{\mu\, a}{}^c h_{\nu\, bc} = \delta_{\mu\nu} Q_{\mu\, ab}$ reduce the number by $2$. So no non-vanishing expression with an odd number of $h_\mu$s can result from these operations.
\end{proof}

The immediate consequence of the theorem is that the simple approach to characterization using only the Riemann tensor does not work. Since the Riemann tensor itself cannot produce a two-form, we would have to try to employ its covariant derivative. This possibility was left to future invetigation.

\section{Integrability condition of the CKY equation} \label{sec:GeneralizedIC} 

An IDEAL characterization requires the spacetime Riemann tensor $R$ (and its concomitants) to satisfy certain equations, so from its point of view it is reasonable to search for and investigate various constraints, which the Riemann tensor obeys. One of those constraints is the integrability condition of the conformal Killing-Yano (CKY) equation \cite[Eq.~(5.77)]{FKK}, namely
\begin{align} \label{eq:IntegrabilityConditionCKYE}
(D-2)R^{ab}{}_{e[c} h^e{}_{d]} - \delta^{[a}_{[c} R^{b]}{}_{d]}{}^{ef} h_{ef} - 2 R^{[a}{}_{e} \delta^{b]}_{[c} h^{e}{}_{d]} = 0,
\end{align}
which, on a given background geometry, is a linear homogeneous algebraic condition on $h_{ab}$, $\mathcal{J}[h]=0$, where $\mathcal{J}\colon {\yng(1,1)} \to {\yng(1,1) \otimes \yng(1,1)}$. However, we can also interpret it as a linear homogeneous algebraic condition for $R_{ijkl}$, if the CKY tensor $h_{ab}$ is already known.
We define $\mathcal{I}\colon {\yng(2,2)} \mapsto {\yng(1,1) \otimes \yng(1,1)}$, with
\begin{equation} \label{eq:OperatorIC}
\mathcal{I}_{abcd}{}^{ijkl} \equiv P_{\young(ik,jl)}
	\left( -(D-2)\delta^i_{[a} \delta^j_{b]}h^k_{~[c}\delta^l_{d]} + h^{ij}\delta^k_{[a}g_{b][c}\delta^l_{d]} - 2 g^{ik} \delta^l_{[a} g_{b][c} h^j_{~d]} \right) ,
\end{equation}
so that the equation \eqref{eq:IntegrabilityConditionCKYE} reads $\mathcal{I}_{abcd}{}^{ijkl} R_{ijkl} = 0$. This equation makes sense not only on higher-dimensional Kerr-NUT-(A)dS spacetimes, but also on any spacetime with a given generic 2-form $h_{ab}$ (in the sense of definition~\ref{def:GeneralH} below). Understanding the strength of this condition might be important in the search for IDEAL characterization equations.
Our results also suggest that the strength of analogous integrability conditions for other variations on Killing tensors~\cite{Batista} should not be underestimated when studying other uniqueness or rigidity problems.

\begin{remark} \label{rem:TowardsIdealIC}
Although equation~\eqref{eq:IntegrabilityConditionCKYE} is not IDEAL in nature (it depends on an external tensor field $h_{ab}$), it is not out of the question that it could lead to some more IDEAL condition. For instance, treating the integrability condition $\mathcal{J}[h]=0$ as a rectangular system of equations for $h_{ab}$, its solvability could be expressed purely in terms of the vanishing of some matrix minors of this system (analogous to the vanishing of the determinant for a square linear system), which should have a covariant expression purely in terms of $R_{abcd}$, albeit a complicated one.
\end{remark}

\begin{definition} \label{def:GeneralH}
Consider a $(2n+\varepsilon)$-dimensional spacetime $(\man,g)$ equipped with an orthonormal frame $\{e^\mu\}$ and $n$ functionally independent functions $x_\mu$ defining the following tensor field
\begin{align} \label{eq:CKYlikeTensor}
h_{ab} = \sum_{\mu} x_{\mu} e^{\mu} \wedge e^{n+\mu}.
\end{align}
Then we define the operator of the integrability condition of CKY equation $\mathcal{I}$ by formula \eqref{eq:OperatorIC} with $h_{ab}$ defined above. 
\end{definition}

The functions $x_\mu$ have no longer the meaning of coordinates, hence $\mathcal{I}$ can be defined in every spacetime. The choice is, of course, not unique nor is the splitting of the tangent space by the orthogonal frame. 
It turns out that the operator $\mathcal{I}$ has an important property, which will be useful in our further calculations.

\begin{theorem} \label{the:ICZeroWeightTheorem}
The operator $\mathcal{I}$ \eqref{eq:OperatorIC} intertwines the action of the
stabilizer group $U(1)^n \rtimes S_n$ on its domain and codomain. Hence, in
particular, it preserves the global and all partial weight subspaces.
\begin{proof}
Once observed, this result is fairly obvious. Each element of the stabilizer
group naturally acts on tensors as a subgroup of $GL(2n+\eps)$, hence
preserving the operations of index permutation and contraction. It also
preserves $h_{ab}$ and $g_{ab}$, hence also index raising and lowering. Since
these are the only operations and tensors used in formula~\eqref{eq:OperatorIC}, the operator $\mathcal{I}$ itself intertwines the action of the stabilizer group.
\end{proof}
\end{theorem}

Equipped with this knowledge, we split the space of four-tensors into subspaces of different weights as we described in section \ref{sec:RiemannWeightSubspaces}. The theorem \ref{the:ICZeroWeightTheorem} says that $\mathcal{I}$ is block-diagonal in a suitable basis. Tensors from different (weight) subspaces are not mixed by the action of $\mathcal{I}$. This remarkable property is a consequence of presence of the symmetry group $U(1)^n \rtimes S_n$.

To simplify the analysis of $\mathcal{I}$, we decompose~\cite{Lie} its codomain into irreducible subspaces of four-tensors labeled by Young diagrams, see section \ref{sec:YoungDiagrams}:
\begin{equation} \label{eq:CodomainICParts}
{\yngbig(1,1)\otimes\yngbig(1,1) \cong \yngbig(1,1,1,1) \oplus \yngbig(2,1,1) \oplus \yngbig(2,2)}.
\end{equation}
With respect to this decomposition we define corresponding parts of $\mathcal{I}$ 
\begin{align}
\mathcal{I}^{(1111)}{}_{abcd}{}^{ijkl} &\equiv \mathcal{I}_{[abcd]}{}^{ijkl}, \nonumber \\ 
\mathcal{I}^{(211)}{}_{abcd}{}^{ijkl} &\equiv P_{\,\young(ac,b,d)}\,\big(\mathcal{I}_{abcd}{}^{ijkl}\big), \label{eq:OperatorICParts}\\
\mathcal{I}^{(22)}{}_{abcd}{}^{ijkl} &\equiv P_{\,\young(ac,bd)}\,\big(\mathcal{I}{}_{abcd}{}^{ijkl}\big). \nonumber
\end{align}

\begin{lemma} \label{lem:IC1111}
As defined by \eqref{eq:OperatorICParts},
the fully antisymmetric part $\mathcal{I}^{(1111)} \equiv 0$.
\end{lemma}
\begin{proof}
Direct calculation in \texttt{xAct}~\cite{xAct}.
\end{proof}

The remaining parts are in general non-zero, however, the $(22)$ part vanishes on the totally zero weight tensors to which to Kerr-NUT-(A)dS Riemann tensor belong.

\begin{lemma} \label{lem:IC22OnZeroWeightRiemann}
Let $\mathcal{I}^{(22)}$ be defined by \eqref{eq:OperatorICParts} and $R_{abcd}$ be a Riemann-symmetric tensor of totally zero weight \eqref{eq:RiemannInBasis}. Then
\begin{align*}
\mathcal{I}^{(22)}{}_{abcd}{}^{ijkl} R_{ijkl} = 0
\end{align*}
\end{lemma}
\begin{proof}
By direct calculation in \texttt{xAct}~\cite{xAct}. We use the defining formula for totally zero weight Riemann-symmetric tensors \eqref{eq:RiemannInBasis} and the formulas for contractions of $h_\mu$s and $Q_\mu$s listed in lemmas~\ref{lem:ProductsOfHsQs}, \ref{lem:TracesOfHsQs}, and \ref{lem:ProductsOfYSHsQs}.
\end{proof}

The integrability condition for a \emph{general} Riemann-symmetric tensor is now split into two separate equations corresponding to the non-vanishing parts of $\mathcal{I}$
\begin{align} 
\mathcal{I}^{(22)}{}_{abcd}{}^{ijkl} R_{ijkl} &= 0, \label{eq:OICR22}\\
\mathcal{I}^{(211)}{}_{abcd}{}^{ijkl} R_{ijkl} &= 0. \label{eq:OICR211}
\end{align}

Direct calculation is too complicated, we have to use more sophisticated methods.
In the next section we use the decomposition of Riemann-symmetric tensors that we obtained in lemma~\ref{lem:Parametrization}.
Since that decomposition respects the stabilizer symmetry group, we can
take advantage of theorem \ref{the:ICZeroWeightTheorem}, which allows us to investigate the sectors of different weights separately, and their sub-diagonals sparsely decoupled.

To achieve these simplifications, we need to apply a similar decomposition to the codomains of $\mathcal{I}^{(22)}$ and $\mathcal{I}^{(211)}$. For the former, we can reuse lemma~\ref{lem:Parametrization}, while for the latter, we need to define an analogous decomposition. For our purposes, it is sufficient to define the maps

\begin{align} \label{eq:PL-map}
	\mathfrak{P}^L\colon \bigoplus_{(\lambda)} \mathfrak{T}_{(\lambda)} = \mathfrak{T}^L &\rightarrow \mathfrak{L} , &
	\mathfrak{P}^L\colon \big(\TT^{(\lambda)}\big) \mapsto L &= \sum_{(\lambda)} \TT^{(\lambda)} \tilde{E}^{(\lambda)} ,
	\\ \label{eq:PLbar-map}
	\bar{\mathfrak{P}}^L\colon \mathfrak{L} &\rightarrow \mathfrak{T}^L , &
	\bar{\mathfrak{P}}^L\colon L \mapsto \big(\TT^{(\lambda)}\big) &= \bigoplus_{(\lambda)} \tilde{P}^{(\lambda)} L ,
\end{align}
where $\mathfrak{L}$ is the space of spacetime tensors  of $(211)$ Young type and the summation over $(\lambda)$ ranges over the triples of Young diagrams appearing in the decomposition~\eqref{eq:Res211}. The corresponding projections, defined analogously to those in \eqref{eq:ICProjectors22}, are given by the following explicit formulas: first without $\hat{e}^0$s
{\small
\begin{equation} \label{eq:ICProjectors211}
\begin{aligned}
(\tilde{P}^{(211,0,0)})^{\mu\nu\rho\sigma}_{abcd} &= P_{\,\young(ac,b,d)} \, P_{\,\young(\mu\rho,\nu,\sigma)}\, E_{abcd}^{{\overset{+}{\mu}}{\overset{+}{\nu}}{\overset{+}{\rho}}{\overset{+}{\sigma}}} , &
(\tilde{P}^{(0,211,0)})^{\mu\nu\rho\sigma}_{abcd} &= P_{\,\young(ac,b,d)} \, P_{\,\young(\mu\rho,\nu,\sigma)}\, E_{abcd}^{{\overset{-}{\mu}}{\overset{-}{\nu}}{\overset{-}{\rho}}{\overset{-}{\sigma}}} ,  \\
(\tilde{P}^{(111,1,0)})^{\mu\nu\rho\sigma}_{abcd} &= P_{\,\young(ac,b,d)} \, P_{\,\young(\mu,\nu,\sigma)}\, E_{abcd}^{{\overset{+}{\mu}}{\overset{+}{\nu}}{\overset{-}{\rho}}{\overset{+}{\sigma}}} , &
(\tilde{P}^{(1,111,0)})^{\mu\nu\rho\sigma}_{abcd} &= P_{\,\young(ac,b,d)} \, P_{\,\young(\mu,\nu,\sigma)}\, E_{abcd}^{{\overset{-}{\mu}}{\overset{-}{\nu}}{\overset{+}{\rho}}{\overset{-}{\sigma}}} ,  \\
(\tilde{P}^{(11,11,0)})^{\mu\nu\rho\sigma}_{abcd} &= P_{\,\young(ac,b,d)} \, P_{\,\young(\rho,\sigma)} \, P_{\,\young(\mu,\nu)} \, E_{abcd}^{{\overset{+}{\mu}}{\overset{+}{\nu}}{\overset{-}{\rho}}{\overset{-}{\sigma}}} , & &
\\
(\tilde{P}^{(11,2,0)})^{\mu\nu\rho\sigma}_{abcd} &= P_{\,\young(ac,b,d)} \, P_{\,\young(\nu,\sigma)} \ P_{\,\young(\mu\rho)} \, E_{abcd}^{{\overset{-}{\mu}}{\overset{+}{\nu}}{\overset{-}{\rho}}{\overset{+}{\sigma}}} , &
(\tilde{P}^{(2,11,0)})^{\mu\nu\rho\sigma}_{abcd} &= P_{\,\young(ac,b,d)} \, P_{\,\young(\nu,\sigma)} \ P_{\,\young(\mu\rho)} \, E_{abcd}^{{\overset{+}{\mu}}{\overset{-}{\nu}}{\overset{+}{\rho}}{\overset{-}{\sigma}}} ,  \\
(\tilde{P}^{(21,1,0)})^{\mu\nu\rho\sigma}_{abcd} &= P_{\,\young(ac,b,d)} \, P_{\,\young(\mu\rho,\nu)}\, E_{abcd}^{{\overset{+}{\mu}}{\overset{+}{\nu}}{\overset{+}{\rho}}{\overset{-}{\sigma}}} , &
(\tilde{P}^{(1,21,0)})^{\mu\nu\rho\sigma}_{abcd} &= P_{\,\young(ac,b,d)} \, P_{\,\young(\mu\rho,\sigma)}\, E_{abcd}^{{\overset{-}{\mu}}{\overset{+}{\nu}}{\overset{-}{\rho}}{\overset{-}{\sigma}}} , 
\end{aligned}
\end{equation}
}
then with one $\hat{e}^0$
\begin{equation} \label{eq:ICProjectors211E0}
\begin{aligned} 
(\tilde{P}^{(21,0,1)})^{\mu\nu\rho}_{abcd} &= P_{\,\young(ac,b,d)} \, P_{\,\young(\mu\rho,\nu)}\, E_{abcd}^{{\overset{+}{\mu}}{\overset{+}{\nu}}{\overset{+}{\rho}}{0}} , &
(\tilde{P}^{(0,21,1)})^{\mu\nu\rho}_{abcd} &= P_{\,\young(ac,b,d)} \, P_{\,\young(\mu\rho,\nu)}\, E_{abcd}^{{\overset{-}{\mu}}{\overset{-}{\nu}}{\overset{-}{\rho}}{0}} ,  \\
(\tilde{P}^{(2,1,1)})^{\mu\nu\rho}_{abcd} &= P_{\,\young(ac,b,d)} \, P_{\,\young(\mu\rho)}\, E_{abcd}^{{\overset{+}{\mu}}{\overset{-}{\nu}}{\overset{+}{\rho}}{0}} , &
(\tilde{P}^{(1,2,1)})^{\mu\nu\rho}_{abcd} &= P_{\,\young(ac,b,d)} \, P_{\,\young(\mu\rho)}\, E_{abcd}^{{\overset{-}{\mu}}{\overset{+}{\nu}}{\overset{-}{\rho}}{0}} ,  \\
(\tilde{P}^{(11,1,1)(1)})^{\mu\nu\rho}_{abcd} &= P_{\,\young(ac,b,d)} \, P_{\,\young(\mu,\nu)}\, E_{abcd}^{{\overset{+}{\mu}}{\overset{+}{\nu}}{\overset{-}{\rho}}{0}} , &
(\tilde{P}^{(11,1,1)(2)})^{\mu\nu\rho}_{abcd} &= P_{\,\young(ac,b,d)} \, P_{\,\young(\mu,\nu)}\, E_{abcd}^{{\overset{+}{\mu}}{\overset{+}{\nu}}{0}{\overset{-}{\rho}}} ,
\\
(\tilde{P}^{(1,11,1)(1)})^{\mu\nu\rho}_{abcd} &= P_{\,\young(ac,b,d)} \, P_{\,\young(\mu,\nu)}\, E_{abcd}^{{\overset{-}{\mu}}{\overset{-}{\nu}}{\overset{+}{\rho}}{0}} , &
(\tilde{P}^{(1,11,1)(2)})^{\mu\nu\rho}_{abcd} &= P_{\,\young(ac,b,d)} \, P_{\,\young(\mu,\nu)}\, E_{abcd}^{{\overset{-}{\mu}}{\overset{-}{\nu}}{0}{\overset{+}{\rho}}} ,  \\
(\tilde{P}^{(111,0,1)(1)})^{\mu\nu\rho}_{abcd} &= P_{\,\young(ac,b,d)} \, P_{\,\young(\mu,\nu,\rho)}\, E_{abcd}^{{\overset{+}{\mu}}{\overset{+}{\nu}}{0}{\overset{+}{\rho}}} , &
(\tilde{P}^{(0,111,1)(2)})^{\mu\nu\rho}_{abcd} &= P_{\,\young(ac,b,d)} \, P_{\,\young(\mu,\nu,\rho)}\, E_{abcd}^{{\overset{-}{\mu}}{\overset{-}{\nu}}{0}{\overset{-}{\rho}}} , 
\end{aligned}
\end{equation}
and with two $\hat{e}^0$s
\begin{equation} \label{eq:ICProjectors211E0E0}
\begin{aligned}
(\tilde{P}^{(11,0,2)})^{\mu\nu}_{abcd} &= P_{\,\young(ac,b,d)} \, P_{\,\young(\mu,\nu)}\, E_{abcd}^{{\overset{+}{\mu}}{\overset{+}{\nu}}{0}{0}} , \qquad & 
(\tilde{P}^{(0,11,2)})^{\mu\nu}_{abcd} &= P_{\,\young(ac,b,d)} \, P_{\,\young(\mu,\nu)}\, E_{abcd}^{{\overset{-}{\mu}}{\overset{-}{\nu}}{0}{0}} , \\
(\tilde{P}^{(1,1,2)})^{\mu\nu}_{abcd} &= P_{\,\young(ac,b,d)} \, E_{abcd}^{{\overset{+}{\mu}}{\overset{-}{\nu}}{0}{0}} . & &
\end{aligned}
\end{equation}
As before, because the matrix of scalar products of $e^{\pm}$ is off-diagonal (lemma~\ref{lem:ProductsOfEpm}), to every projection the corresponding embedding is defined with $+$ and $-$ exchanged, for instance 
\[
(\tilde{E}^{(11,2,0)})^{\mu\nu\rho\sigma}_{abcd} = P_{\,\young(ac,b,d)} \, P_{\,\young(\nu,\sigma)} \ P_{\,\young(\mu\rho)} \, E_{abcd}^{{\overset{+}{\mu}}{\overset{-}{\nu}}{\overset{+}{\rho}}{\overset{-}{\sigma}}} .
\]

\subsection{Calculation of the projections} \label{sec:CalculationOfProjectionsIC}

After we substitute \eqref{eq:DecompositionOfRtoWeightSectors} into \eqref{eq:OICR22}, reps.\ \eqref{eq:OICR211}, we project the result with all possible projectors \eqref{eq:ICProjectors22} resp.\ \eqref{eq:ICProjectors211}--\eqref{eq:ICProjectors211E0E0}. This corresponds to computing the block components of the compositions $\bar{\mathfrak{P}}^R \circ \mathcal{I}^{(22)} \circ \mathfrak{P}^R$ and $\bar{\mathfrak{P}}^L \circ \mathcal{I}^{(211)} \circ \mathfrak{P}^R$. The results typically contain a number of terms (on the order of 10), which can be simplified by grouping them and finding a representative term, which generates the whole group after the corresponding symmetrizer (used to define the projector) is applied. We state here such simplified equations. They are sorted according to the global weight of the coefficients in descending order. Notice that coefficients with different weights do not mix in the equations, it is a consequence of theorem \ref{the:ICZeroWeightTheorem}.  

\allowdisplaybreaks

The projected equation \eqref{eq:OICR22} by projectors \eqref{eq:ICProjectors22} (without $e^0$): 
{\footnotesize
\begin{align}
0 &= P_{\,\young(\mu\rho,\nu\sigma)} \big[ x_\mu \TT^{(22,0,0)}_{\mu\nu\rho\sigma} \big], \label{eq:ICProjections22-1} \\
0 &= P_{\,\young(\mu\rho,\nu)} \big[ (D-2)(3x_\mu - x_\nu) \TT^{(21,1,0)}_{\mu\nu\rho\sigma} + 3x_\mu \delta_{\nu\sigma}\sum_\alpha(\TT^{(21,1,0)}_{\mu\rho\alpha\alpha} + 2\TT^{(21,1,0)}_{\alpha\mu\rho\alpha}) - 2x_\mu \delta_{\nu\sigma} \TT^{(2,0,2)}_{\mu\rho} \big], \label{eq:ICProjections22-2} \\
0 &=  P_{\,\young(\mu,\rho)}  P_{\,\young(\nu,\sigma)} \big[ (x_\mu - x_\sigma)\big(-4(D-2)\TT^{(2,2,0)}_{\mu\nu\rho\sigma} +  2\delta_{\nu\rho}\sum_\alpha(2\TT^{(2,2,0)}_{\mu\sigma\alpha\alpha} + 3\TT^{(11,11,0)}_{\mu\alpha\alpha\sigma}) - \delta_{\nu\rho} \TT^{(1,1,2)}_{\mu\sigma} \big)\big], \label{eq:ICProjections22-3} \\
0 &= P_{\,\young(\mu\nu)}  P_{\,\young(\rho\sigma)} \big[ (x_\mu - x_\sigma)\big(-2(D-2)\TT^{(11,11,0)}_{\mu\nu\rho\sigma}  + 2\delta_{\nu\rho}\sum_\alpha(2\TT^{(2,2,0)}_{\mu\sigma\alpha\alpha} + 3\TT^{(11,11,0)}_{\mu\alpha\alpha\sigma}) - \delta_{\nu\rho} \TT^{(1,1,2)}_{\mu\sigma} \big)\big], \label{eq:ICProjections22-4} \\
0 &= P_{\,\young(\mu\rho,\sigma)} \big[ (D-2)(3x_\mu - x_\nu) \TT^{(1,21,0)}_{\mu\nu\rho\sigma} + 3x_\mu \delta_{\nu\sigma}\sum_\alpha(\TT^{(1,21,0)}_{\mu\alpha\alpha\rho} + 2\TT^{(1,21,0)}_{\alpha\alpha\rho\mu}) - 2x_\mu \delta_{\nu\sigma} \TT^{(0,2,2)}_{\mu\rho} \big], \label{eq:ICProjections22-5} \\
0 &= P_{\,\young(\mu\rho,\nu\sigma)} \big[ x_\mu \TT^{(0,22,0)}_{\mu\nu\rho\sigma} \big]. \label{eq:ICProjections22-6}
\end{align}
}
The projected equation \eqref{eq:OICR22} by projectors \eqref{eq:ICProjectors22} (with one $e^0$):
{\footnotesize
\begin{align}
0 &= P_{\,\young(\mu\rho,\nu)} \big[ x_\mu \TT^{(21,0,1)}_{\mu\nu\rho} \big], \label{eq:ICProjections22E0-1} \\
0 &= P_{\,\young(\mu\rho)} \big[ (D-2)(x_\nu - 2x_\rho)\TT^{(2,1,1)}_{\mu\nu\rho} +  \delta_{\mu\nu}x_\rho\sum_\alpha(\TT^{(2,1,1)}_{\alpha\alpha\rho} + \tfrac{3}{2}\TT^{(11,1,1)}_{\rho\alpha\alpha}) \big], \label{eq:ICProjections22E0-2} \\
0 &= P_{\,\young(\mu,\nu)} \big[ (D-2)(x_\rho - 2x_\nu)\TT^{(11,1,1)}_{\mu\nu\rho} -  \delta_{\mu\rho}x_\nu\sum_\alpha(3\TT^{(11,1,1)}_{\nu\alpha\alpha} + 2\TT^{(2,1,1)}_{\alpha\alpha\nu}) \big], \label{eq:ICProjections22E0-3} \\
0 &= P_{\,\young(\mu\rho)} \big[ (D-2)(x_\nu - 2x_\rho)\TT^{(1,2,1)}_{\mu\nu\rho} +  \delta_{\mu\nu}x_\rho\sum_\alpha(\TT^{(1,2,1)}_{\alpha\alpha\rho} + \tfrac{3}{2}\TT^{(1,11,1)}_{\rho\alpha\alpha}) \big], \label{eq:ICProjections22E0-4} \\
0 &= P_{\,\young(\mu,\nu)} \big[ (D-2)(x_\rho - 2x_\nu)\TT^{(1,11,1)}_{\mu\nu\rho} -  \delta_{\mu\rho}x_\nu\sum_\alpha(3\TT^{(1,11,1)}_{\nu\alpha\alpha} + 2\TT^{(1,2,1)}_{\alpha\alpha\nu}) \big], \label{eq:ICProjections22E0-5} \\
0 &= P_{\,\young(\mu\rho,\nu)} \big[ x_\mu \TT^{(0,21,1)}_{\mu\nu\rho} \big]. \label{eq:ICProjections22E0-6}
\end{align}
}
The projected equation \eqref{eq:OICR22} by projectors \eqref{eq:ICProjectors22} (with two $e^0$)s:
{\footnotesize
\begin{align}
0 &= P_{\,\young(\mu\nu)} \big[ x_\mu\big(2(D-3)\TT^{(2,0,2)}_{\mu\nu} + 3\sum_\alpha(\TT^{(21,1,0)}_{\alpha\mu\nu\alpha} + \TT^{(21,1,0)}_{\alpha\nu\mu\alpha})\big)\big], \label{eq:ICProjections22E0E0-1} \\
0 &= (x_\mu - x_\nu)\big((D-3)\TT^{(1,1,2)}_{\mu\nu} + 2\sum_\alpha(3\TT^{(11,11,0)}_{\mu\alpha\alpha\nu} + 2\TT^{(2,2,0)}_{\mu\nu\alpha\alpha})\big), \label{eq:ICProjections22E0E0-2} \\
0 &= P_{\,\young(\mu\nu)} \big[ x_\mu\big(2(D-3)\TT^{(0,2,2)}_{\mu\nu} + 3\sum_\alpha(\TT^{(1,21,0)}_{\alpha\alpha\mu\nu} + \TT^{(1,21,0)}_{\alpha\alpha\nu\mu})\big)\big]. \label{eq:ICProjections22E0E0-3} 
\end{align}
}

The projected equation \eqref{eq:OICR211} by projectors \eqref{eq:ICProjectors211} (without $e^0$):
{\footnotesize
\begin{align}
0 &= P_{\,\young(\mu\rho,\nu,\sigma)} \big[ x_\mu \TT^{(22,0,0)}_{\mu\nu\rho\sigma} \big], \label{eq:ICProjections211-1} \\
0 &= P_{\,\young(\mu\rho,\nu)} \big[ -\delta_{\mu\sigma} \sum_{\alpha}\big( 2x_\alpha \TT^{(21,1,0)}_{\rho\nu\alpha\alpha} + (x_\nu - x_\rho)(\TT^{(21,1,0)}_{\alpha\rho\nu\alpha}+ \TT^{(21,1,0)}_{\alpha\nu\rho\alpha}) \big) \nonumber \\
& \qquad \qquad \qquad + (D-2)(3x_\mu - 2x_\rho + x_\sigma)\TT^{(21,1,0)}_{\mu\nu\rho\sigma} + \tfrac{2}{3}\delta_{\mu\sigma}(x_\nu - x_\rho) \TT^{(2,0,2)}_{\nu\rho} \big], \label{eq:ICProjections211-2} \\
0 &= P_{\,\young(\mu,\nu,\rho)} \big[ (D-2)x_\mu \TT^{(21,1,0)}_{\nu\mu\sigma\rho} + \delta_{\nu\rho}\sum_\alpha(x_\alpha \TT^{(21,1,0)}_{\mu\sigma\alpha\alpha} - x_\mu( \TT^{(21,1,0)}_{\alpha\sigma\mu\alpha} + \TT^{(21,1,0)}_{\alpha\mu\sigma\alpha})) \nonumber \\
& \qquad \qquad \qquad + \tfrac{2}{3}\delta_{\nu\rho} x_\mu \TT^{(2,0,2)}_{\mu\sigma} \big], \label{eq:ICProjections211-3} \\
0 &= P_{\,\young(\mu,\nu)}P_{\,\young(\rho,\sigma)} \big[ 2(D-2)(x_\mu + x_\sigma)\TT^{(11,11,0)}_{\mu\nu\rho\sigma} - 2\delta_{\nu\rho}\sum_\alpha(2x_\alpha + 3x_\mu + 3x_\sigma)\TT^{(11,11,0)}_{\mu\alpha\alpha\sigma} \nonumber \\
& \qquad \qquad \qquad + 4\delta_{\nu\rho}\sum_\alpha (2x_\alpha - x_\mu - x_\sigma) \TT^{(2,2,0)}_{\mu\sigma\alpha\alpha} + \delta_{\nu\rho}(x_\mu + x_\sigma) \TT^{(1,1,2)}_{\mu\sigma} \big], \label{eq:ICProjections211-4} \\
0 &= P_{\,\young(\nu,\sigma)}P_{\,\young(\mu\rho)} \big[ 2(D-2)x_\mu \TT^{(11,11,0)}_{\nu\sigma\rho\mu} + 2\delta_{\nu\rho}\sum_\alpha(2x_\alpha + 3x_\mu + 3x_\sigma)\TT^{(11,11,0)}_{\sigma\alpha\alpha\mu} - 4(D-2)x_\sigma \TT^{(2,2,0)}_{\nu\rho\sigma\mu} \nonumber \\
& \qquad \qquad \qquad - 4\delta_{\nu\rho}\sum_\alpha (2x_\alpha - x_\mu - x_\sigma) \TT^{(2,2,0)}_{\sigma\mu\alpha\alpha} - \delta_{\nu\rho}(x_\mu + x_\sigma) \TT^{(1,1,2)}_{\sigma\mu}\big], \label{eq:ICProjections211-5} \\
0 &=  P_{\,\young(\nu,\sigma)}P_{\,\young(\mu\rho)} \big[ 2(D-2)x_\mu \TT^{(11,11,0)}_{\rho\mu\nu\sigma} + 2\delta_{\nu\rho}\sum_\alpha(2x_\alpha + 3x_\mu + 3x_\sigma)\TT^{(11,11,0)}_{\mu\alpha\alpha\sigma} - 4(D-2)x_\sigma \TT^{(2,2,0)}_{\mu\nu\rho\sigma} \nonumber \\
& \qquad \qquad \qquad - 4\delta_{\nu\rho}\sum_\alpha (2x_\alpha - x_\mu - x_\sigma) \TT^{(2,2,0)}_{\mu\sigma\alpha\alpha} - \delta_{\nu\rho}(x_\mu + x_\sigma) \TT^{(1,1,2)}_{\mu\sigma}\big], \label{eq:ICProjections211-6} \\
0 &= P_{\,\young(\mu,\nu,\rho)} \big[ (D-2)x_\mu \TT^{(1,21,0)}_{\mu\rho\sigma\nu} - \delta_{\nu\rho}\sum_\alpha(x_\alpha \TT^{(1,21,0)}_{\mu\alpha\alpha\sigma} - x_\mu( \TT^{(1,21,0)}_{\alpha\alpha\sigma\mu} + \TT^{(1,21,0)}_{\alpha\alpha\mu\sigma})) \nonumber \\
& \qquad \qquad \qquad - \tfrac{2}{3}\delta_{\nu\rho} x_\mu \TT^{(2,0,2)}_{\mu\sigma} \big], \label{eq:ICProjections211-7} \\
0 &= P_{\,\young(\mu\rho,\sigma)} \big[ -\delta_{\mu\nu} \sum_{\alpha}\big( 2x_\alpha \TT^{(1,21,0)}_{\rho\alpha\alpha\sigma} + (x_\sigma - x_\rho)(\TT^{(1,21,0)}_{\alpha\alpha\sigma\rho} + \TT^{(1,21,0)}_{\alpha\alpha\rho\sigma}) \big) \nonumber \\
& \qquad \qquad \qquad + (D-2)(3x_\mu - 2x_\rho + x_\nu)\TT^{(1,21,0)}_{\mu\nu\rho\sigma} + \tfrac{2}{3}\delta_{\mu\nu}(x_\sigma - x_\rho) \TT^{(0,2,2)}_{\sigma\rho} \big], \label{eq:ICProjections211-8} \\
0 &= P_{\,\young(\mu\rho,\nu,\sigma)} \big[ x_\mu \TT^{(0,22,0)}_{\mu\nu\rho\sigma} \big]. \label{eq:ICProjections211-9}
\end{align}
}
The projected equation \eqref{eq:OICR211} by projectors \eqref{eq:ICProjectors211E0} (with one $e^0$):
{\footnotesize
\begin{align}
0 &= P_{\,\young(\mu\rho,\nu)} \big[ (3x_\nu - x_\rho) \TT^{(21,0,1)}_{\mu\nu\rho} \big], \label{eq:ICProjections211E0-1} \\
0 &= P_{\,\young(\mu,\nu,\rho)} \big[ x_\mu \TT^{(21,0,1)}_{\mu\nu\rho} \big], \label{eq:ICProjections211E0-2} \\
0 &= P_{\,\young(\mu\rho)}  \big[ (D-2)x_\mu \TT^{(11,1,1)}_{\rho\mu\nu} + \tfrac{1}{2} \delta_{\mu\nu}\sum_\alpha(2x_\alpha + 3x_\rho)\TT^{(11,1,1)}_{\rho\alpha\alpha} + (D-2)x_\nu \TT^{(2,1,1)}_{\mu\nu\rho}  \nonumber \\
& \qquad \qquad \qquad - \delta_{\mu\nu}\sum_\alpha(2x_\alpha - x_\rho)\TT^{(2,1,1)}_{\rho\alpha\alpha} \big], \label{eq:ICProjections211E0-3} \\
0 &= P_{\,\young(\mu,\nu)} \big[ (D-2)x_\mu \TT^{(11,1,1)}_{\mu\nu\rho} + \delta_{\mu\rho}\sum_\alpha(2x_\alpha + 3x_\nu)\TT^{(11,1,1)}_{\nu\alpha\alpha} - 2(D-2)x_\nu \TT^{(2,1,1)}_{\mu\rho\nu} \nonumber \\
& \qquad \qquad \qquad - 2\delta_{\mu\rho}\sum_\alpha(2x_\alpha - x_\nu)\TT^{(2,1,1)}_{\nu\alpha\alpha} \big], \label{eq:ICProjections211E0-4} \\
0 &= P_{\,\young(\mu,\nu)} \big[ (D-2)(x_\mu + x_\rho) \TT^{(11,1,1)}_{\mu\nu\rho} + 2(D-2)x_\nu \TT^{(2,1,1)}_{\mu\rho\nu} \big], \label{eq:ICProjections211E0-5} \\
0 &= P_{\,\young(\mu\rho)} \big[ (D-2)x_\mu \TT^{(1,11,1)}_{\rho\mu\nu} + \tfrac{1}{2} \delta_{\mu\nu}\sum_\alpha(2x_\alpha + 3x_\rho)\TT^{(1,11,1)}_{\rho\alpha\alpha} + (D-2)x_\nu \TT^{(1,2,1)}_{\mu\nu\rho} \nonumber \\
& \qquad \qquad \qquad - \delta_{\mu\nu}\sum_\alpha(2x_\alpha - x_\rho)\TT^{(1,2,1)}_{\rho\alpha\alpha} \big], \label{eq:ICProjections211E0-6} \\
0 &= P_{\,\young(\mu,\nu)} \big[ (D-2)x_\mu \TT^{(1,11,1)}_{\mu\nu\rho} + \delta_{\mu\rho}\sum_\alpha(2x_\alpha + 3x_\nu)\TT^{(1,11,1)}_{\nu\alpha\alpha} - 2(D-2)x_\nu \TT^{(1,2,1)}_{\mu\rho\nu} \nonumber \\
& \qquad \qquad \qquad - 2\delta_{\mu\rho}\sum_\alpha(2x_\alpha - x_\nu)\TT^{(1,2,1)}_{\nu\alpha\alpha} \big], \label{eq:ICProjections211E0-7} \\
0 &= P_{\,\young(\mu,\nu)} \big[ (D-2)(x_\mu + x_\rho) \TT^{(1,11,1)}_{\mu\nu\rho} + 2(D-2)x_\nu \TT^{(1,2,1)}_{\mu\rho\nu} \big], \label{eq:ICProjections211E0-8}\\
0 &= P_{\,\young(\mu,\nu,\rho)} \big[ x_\mu \TT^{(0,21,1)}_{\mu\nu\rho} \big], \label{eq:ICProjections211E0-9} \\
0 &= P_{\,\young(\mu\rho,\nu)} \big[ (3x_\nu - x_\rho) \TT^{(0,21,1)}_{\mu\nu\rho} \big]. \label{eq:ICProjections211E0-10}
\end{align}
}
The projected equation \eqref{eq:OICR211} by projectors \eqref{eq:ICProjectors211E0E0} (with two $e^0$s):
{\footnotesize
\begin{align}
0 &= P_{\,\young(\mu,\nu)} \big[ 2(D-3)x_\mu \TT^{(2,0,2)}_{\mu\nu} + 3\sum_\alpha(2x_\alpha + x_\mu - x_\nu)\TT^{(21,1,0)}_{\alpha\mu\nu\alpha}\big], \label{eq:ICProjections211E0E0-1} \\
0 &=  (D-3)(x_\mu + x_\nu)\TT^{(1,1,2)}_{\mu\nu} - 4\sum_\alpha (2x_\alpha - x_\mu - x_\nu) \TT^{(2,2,0)}_{\mu\nu\alpha\alpha}  \nonumber \\
& \qquad \qquad \qquad + 2\sum_\alpha (2x_\alpha + 3x_\mu + 3x_\nu) \TT^{(11,11,0)}_{\mu\alpha\alpha\nu}, \label{eq:ICProjections211E0E0-2} \\
0 &= P_{\,\young(\mu,\nu)} \big[ 2(D-3)x_\mu \TT^{(2,0,2)}_{\mu\nu} - 3\sum_\alpha(2x_\alpha + x_\nu - x_\mu)\TT^{(1,21,0)}_{\alpha\alpha\mu\nu}\big]. \label{eq:ICProjections211E0E0-3}
\end{align}
}

\allowdisplaybreaks[0]

The equations \eqref{eq:ICProjections22-1}--\eqref{eq:ICProjections211E0E0-3} represent a system of linear equations for unknown tensor coefficients $\TT^{({\cdots})}$s. As we can see, the equations contain terms which contribute only to certain diagonal parts. Moreover, $\TT^{({\cdots})}$s have non-trivial symmetries inherited from the embeddings, which means that their own diagonals are not all independent, see section \ref{sec:Diagonals}.

To save space, we just describe a general way how to investigate these equations\footnote{A detailed calculation can be found in \cite[Sec.2.7.3]{MatejovThesis}.}. As we discussed in section \ref{sec:InvarProperties} the group $S_n$ acts on the frame Greek indices. To obtain independent conditions, one should decompose the coefficients $\TT^{({\cdots})}$ into parts belonging to irreducible subspaces with respect to $S_n$. However, this full decomposition is in practical calculations very tedious for tensors of higher rank than two with a huge number of irreducible subspaces. In the case of the system \eqref{eq:ICProjections22-1}--\eqref{eq:ICProjections211E0E0-3} it is sufficient to employ the partial decomposition into diagonals described in section \ref{sec:DecompositionIntoDiagonals}. Each coefficient $\TT^{({\cdots})}$ is broken into the diagonal and off-diagonal parts using formulas \eqref{eq:Decomposition2Tensor}--\eqref{eq:Decomposition4Tensor}. After that, the equations themsleves are broken into the diagonal and off-diagonal parts using maps \eqref{eq:DiagonalMaps}. Abstractly these steps correspond to computing the block components of the compositions
\[\textstyle
	\left(\bigoplus_{(\lambda)} \bar{\mathfrak{D}}^{(\lambda)}\right) \circ \left(\bar{\mathfrak{P}}^R \circ \mathcal{I}^{(22)} \circ \mathfrak{P}^R\right) \circ \mathfrak{D} \circ \mathfrak{J}
	\quad \text{and} \quad
	\left(\bigoplus_{(\lambda)} \bar{\mathfrak{D}}^{(\lambda)}\right) \circ \left(\bar{\mathfrak{P}}^L \circ \mathcal{I}^{(211)} \circ \mathfrak{P}^R\right) \circ \mathfrak{D} \circ \mathfrak{J}
	.
\]
The resulting sparse equations can now be solved with standard methods of linear algebra. Let us recall that the sparseness was achieved by identifying the stabilizer group of the $\mathcal{I}^{(22)}$ and $\mathcal{I}^{(211)}$ operators (theorem~\ref{the:ICZeroWeightTheorem}) and leveraging Schur's lemma, as discussed in section~\ref{sec:GroupTheory}.

After we perform the calculations we find that all coefficients have to vanish except the diagonals $\TT^{(2,2,0)(0)}, \TT^{(2,2,0)(5)},\TT^{(11,11,0)(6)}$ to which we add $\TT^{(1,1,2)(0)}$ in odd dimensions. The first one, $\TT^{(2,2,0)(0)}$, is \emph{not constrained} by the system, while the other three has to fulfill the only independent equations which are left $(\lambda\neq\kappa)$:
\begin{footnotesize}
\begin{align} 
0 &= 2\sum_\alpha \big( (x_\lambda - x_\alpha)\TT^{(2,2,0)(5)}_{\lambda\alpha} + (x_\kappa -  x_\alpha)\TT^{(2,2,0)(5)}_{\kappa\alpha} \big) \nonumber \\
& \quad - \sum_\alpha \big( (x_\alpha + 3x_\lambda)\TT^{(11,11,0)(6)}_{\lambda\alpha} + (x_\alpha + 3 x_\kappa)\TT^{(11,11,0)(6)}_{\kappa\alpha} \big) + 2(D-2) (x_\kappa + x_\lambda)\TT^{(11,11,0)(6)}_{\lambda\kappa} \nonumber \\
& \quad + \frac{D}{2}x_\lambda\TT^{(1,1,2)(0)}_\lambda + \frac{D}{2}x_\kappa\TT^{(1,1,2)(0)}_\kappa, \label{eq:EquationsForTs} \\
0 &= 2\sum_\alpha \big( -(x_\lambda - x_\alpha)\TT^{(2,2,0)(5)}_{\lambda\alpha} + (x_\kappa -  x_\alpha)\TT^{(2,2,0)(5)}_{\kappa\alpha} \big) + 2(D-2) (x_\lambda - x_\kappa)\TT^{(2,2,0)(5)}_{\lambda\kappa} \nonumber \\
& \quad + \sum_\alpha \big( (x_\alpha + 3x_\lambda)\TT^{(11,11,0)(6)}_{\lambda\alpha} - (x_\alpha + 3 x_\kappa)\TT^{(11,11,0)(6)}_{\kappa\alpha} \big) + (D-2) (x_\kappa - x_\lambda)\TT^{(11,11,0)(6)}_{\lambda\kappa} \nonumber \\
& \quad - \frac{D}{2}x_\lambda\TT^{(1,1,2)(0)}_\lambda + \frac{D}{2}x_\kappa\TT^{(1,1,2)(0)}_\kappa, \nonumber \\
0 &= (D-1)x_\lambda \TT^{(1,1,2)(0)}_\lambda + \sum_\alpha \big( 2(x_\lambda - x_\alpha)\TT^{(2,2,0)(5)}_{\lambda\alpha} - (3x_\lambda + x_\alpha)\TT^{(11,11,0)(6)}_{\lambda\alpha} \big). \nonumber
\end{align} 
\end{footnotesize}

We summarize the calculation in the following theorem and investigate the derived constraints in the next section. 

\begin{theorem}[Integrability condition] \label{the:GeneralizedIntegrability}
Let $\RR_{abcd}$ be the Riemann-sym\-metric tensor
\eqref{eq:DecompositionOfRtoWeightSectors} in a $D=2n+\varepsilon$ dimensional
spacetime equipped with a 2-form $h_{ab}$ as in definition~\ref{def:GeneralH},
which satisfies the integrability condition of CKY equation
$\mathcal{I}_{abcd}{}^{ijkl} \RR_{ijkl} = 0$, where $\mathcal{I}$ is defined
by \eqref{eq:OperatorIC}. 

(a) The integrability condition restricts $\RR_{abcd}$ to have zero global weight with respect to the action of $U(1)^n \rtimes S_n$, hence have the form~\eqref{eq:ZeroWeightRiemann}. 

(b) In the parametrization~\eqref{eq:ZeroWeightRiemann}, most of the $\TT^{(\lambda)}$ coefficients vanish except the following diagonals corresponding to the totally zero weight subspace (lemma~\ref{lem:RelationBetweenParametrizations})
\begin{align*}
\TT^{(2,2,0)(0)}_\mu, \TT^{(2,2,0)(5)}_{\mu\nu},\TT^{(11,11,0)(6)}_{\mu\nu}, \TT^{(1,1,2)(0)}_\mu,
\end{align*}
which are subject to constraints \eqref{eq:EquationsForTs}.
\end{theorem}

\subsection{Integrability condition for the totally zero weight Riemann tensor}

In the previous section we showed that the only possible solution of the integrability condition is a Riemann-symmetric tensor of \emph{global zero weight}, theorem \ref{the:GeneralizedIntegrability}. Moreover, during the calculations we derived non-trivial conditions for the remaining diagonal coefficients \eqref{eq:EquationsForTs}. Instead of trying to solve the constraints directly, we use the parametrization \eqref{eq:RiemannInBasis} in the $h$-basis, which is more convenient to use.

For simplicity of notation let us introduce a new product of two tensors.

\begin{definition}\label{def:StarProduct} Let $A_{ab}, B_{ab}$ be tensors. We define a \emph{star product}
\begin{align} \label{eq:StarProduct} 
(A\star B)_{abcd} \equiv \frac{1}{4} (A_{ac} B_{bd} + A_{bd} B_{ac} - A_{bc} B_{ad} - A_{ad} B_{bc})
\end{align}
\end{definition}
This product can be obtained by double antisymmetrization of $A_{ac} B_{bd}$ in $(a,b)$ and $(c,d)$.
From the definition it is obvious that the product is linear in both arguments and that it is symmetric $(A\star B) = (B\star A)$. 
If $A$ is a symmetric and $B$ an antisymmetric tensor, $(A\star B)$ belongs to the (211) subspace of the $(2) \otimes (2)$ tensor product in the decomposition~\eqref{eq:CodomainICParts}. 
In lemmas \ref{lem:IC1111} and \ref{lem:IC22OnZeroWeightRiemann} we showed that only $\mathcal{I}^{(211)}$ is non-zero on totally zero weight RS tensors. It is therefore possible to project the integrability condition \eqref{eq:IntegrabilityConditionCKYE} with (211) Young projector and introduce corresponding YS products of tensors that appear in the resulting equation. With the help of $\star$ product one can jump over this step and approach the calculation directly. The expressions are also easier to handle and recognize. 

The following two lemmas are essential for the calculation:

\begin{lemma} \label{lem:DiagonalStarProds}
The star product of $Q_\mu$ and $h_\mu$ identically vanishes
\begin{align*}
(Q_\mu\star h_\mu) \equiv 0, \qquad \forall \mu \in \{1,\dots,n\}.
\end{align*}
\end{lemma}

\begin{lemma} \label{lem:LinearIndependenceStarProd}
Tensors $(Q_{\mu} \star h_\nu)$ for $\mu\ne\nu$ are linearly independent
\begin{align*}
\sum_{\substack{\mu, \nu \\ \mu \neq \nu}} a_{\mu \nu} (Q_{\mu} \star h_\nu) + \varepsilon \sum_{\mu} b_\mu ((\hat{e}^0 \hat{e}^0)\star h_\mu) = 0, \quad \Leftrightarrow \quad a_{\mu \nu}= 0 = b_\mu, \quad \forall \mu,\nu.
\end{align*}
\end{lemma}

When we take \eqref{eq:RiemannInBasis}, substitute in into \eqref{eq:IntegrabilityConditionCKYE}, and use the previous lemmas, we arrive at constraints summarized in the lemma below.

\begin{lemma} \label{lem:ICZeroWeightConstraints}
Consider a totally zero weight RS tensor $\mathcal{R}_{abcd}$ in the parametrization \eqref{eq:RiemannInBasis}. Then the integrability
condition \eqref{eq:IntegrabilityConditionCKYE} is equivalent to the following
system for the coefficients $\Omega_{\mu\nu}, \alpha_{\mu\nu},
\alpha^{(\epsilon)}_{\mu}$:
{\footnotesize
\begin{align}
\left(n+\frac{\varepsilon - 2}{2}\right)(\Omega_{\mu \nu} x_\mu + (\alpha_{\mu \nu} + \alpha^{(\epsilon)}_{\mu}) x_\nu) - \sum_{\rho\neq \nu}(\Omega_{\rho \nu} x_\rho + (\alpha_{\rho \nu} + \alpha^{(\epsilon)}_{\rho}) x_\nu) &= 0, \quad (\mu \neq \nu) \label{eq:ICZeroWeightConstraints} \\
\sum_{\rho\neq \nu}(\Omega_{\rho \nu} x_\rho + (\alpha_{\rho \nu} + \alpha^{(\epsilon)}_{\rho}) x_\nu) &= 0, \quad (\varepsilon \neq 0). \label{eq:ICZeroWeightConstraintsEps}
\end{align}
}
\end{lemma}

The constraints on the lemma are compatible with the ones we derived before \eqref{eq:EquationsForTs}.
Consider the first equations from \eqref{eq:EquationsForTs}, say $E_1, E_2$. When we take their difference $E_1 - E_2$ and  use the relations \eqref{eq:RelationBetweenParametrizations} we get 
\begin{align*}
\frac{1}{2}(D-2)(\Omega_{\mu \nu} x_\mu + (\alpha_{\mu \nu} + \alpha^{(\epsilon)}_{\mu}) x_\nu) - \sum_{\rho\neq \nu}(\Omega_{\rho \nu} x_\rho + (\alpha_{\rho \nu} + \alpha^{(\epsilon)}_{\rho}) x_\nu) &= 0, \quad (\mu \neq \nu),
\end{align*}
while the last equation from \eqref{eq:EquationsForTs} present in odd dimensions reads
\begin{align*}
\sum_{\rho\neq \mu}(\Omega_{\rho \mu} x_\rho + (\alpha_{\rho \mu} + \alpha^{(\epsilon)}_{\rho}) x_\mu) &= 0,
\end{align*}
which precisely agrees with the equations \eqref{eq:ICZeroWeightConstraints}.
Let us mention that the sum of the equations is not independent, but it is equal to the difference with exchanged indices, namely $(E_1 + E_2)_{\mu\nu} = (E_1 - E_2)_{\nu\mu}$. The constraints \eqref{eq:EquationsForTs} are fully exhausted.

Solving of the equations \eqref{eq:ICZeroWeightConstraints} is a rather non-trivial task covered in the following theorem.

\begin{theorem}[Integrability condition, totally zero weight] \label{th:IntegrabilityCondition}
With the notation and assumptions of theorem~\ref{the:GeneralizedIntegrability},
the integrability condition $\mathcal{I}_{abcd}{}^{ijkl} \RR_{ijkl} = 0$ has
rank $n^2 + (\varepsilon-2) n$ when restricted to the subspace totally zero
weight Riemann-symmetric tensors $\RR_{ijkl}$. Moreover, when parametrized as
in lemma~\ref{lem:RelationBetweenParametrizations}, the general solution for $n
\geqslant 2$ and $(\mu \neq \nu)$ can be written explicitly as
\begin{align} \label{eq:ExplicitCoefficientsIntCon}
\alpha_{\mu \nu} = \frac{x_\mu d_\mu - x_\nu d_\nu}{x^2_\mu - x^2_\nu} - \varepsilon \left( \frac{d_\mu}{x_\mu} + \frac{d_\nu}{x_\nu}\right), \quad 
\Omega_{\mu \nu} = \frac{x_\mu d_\nu - x_\nu d_\mu}{x^2_\mu - x^2_\nu}, \quad
\alpha^{(\epsilon)}_\mu = \varepsilon \frac{d_\mu}{x_\mu},
\end{align}
where $\{d_1, \dots, d_n\}$ are $n$ independent real parameters.
\end{theorem}

\begin{proof}
We need to find the complete solution for the equations
\eqref{eq:ICZeroWeightConstraints},
\eqref{eq:ICZeroWeightConstraintsEps} from lemma
\ref{lem:ICZeroWeightConstraints}.

First, consider the even dimensions ($\eps=0$). Let us introduce the parameters
\begin{equation} \label{eq:d-param0}
	(n-1) d_\nu = \sum_{\rho\neq \nu}(\Omega_{\rho \nu} x_\rho + (\alpha_{\rho \nu} + \alpha^{(\eps)}_{\rho}) x_\nu)
		.
\end{equation}
For a given pair $\mu \ne \nu$, using both
equation~\eqref{eq:ICZeroWeightConstraints} and its index swapped
version, together with the symmetry of $\alpha_{\nu\mu} =
\alpha_{\mu\nu}$ and $\Omega_{\nu\mu} = \Omega_{\mu\nu}$ gives 
\[
	\begin{pmatrix} x_\mu & x_\nu \\ x_\nu & x_\mu \end{pmatrix}
	\begin{pmatrix} \Omega_{\mu\nu} \\ \alpha_{\mu\nu} \end{pmatrix}
	= \begin{pmatrix} d_\nu \\ d_\mu \end{pmatrix}
	\quad \iff \quad
	\begin{pmatrix} \Omega_{\mu\nu} \\ \alpha_{\mu\nu} \end{pmatrix}
	= \begin{pmatrix}
		\frac{x_\mu d_\mu - x_\nu d_\nu}{x^2_\mu - x^2_\nu} \\
		\frac{x_\mu d_\nu - x_\nu d_\mu}{x^2_\mu - x^2_\nu}
	\end{pmatrix} .
\]
Further constraints on the $d_\mu$ parameters can be obtained only by
plugging the above solution into the defining
equation~\eqref{eq:d-param0}. Direct calculation shows that the
right-hand side evaluates exactly to $(n-1) d_\nu$, which means that the
$d_\mu$'s are not constrained and hence parametrize the general solution
of~\eqref{eq:ICZeroWeightConstraints}.

Next, consider odd dimensions ($\eps=1$). Let us introduce the parameters
\begin{equation} \label{eq:d-param1}
	d_\nu = x_\nu \alpha^{(\eps)}_\nu .
\end{equation}
Adding~\eqref{eq:ICZeroWeightConstraintsEps}
to~\eqref{eq:ICZeroWeightConstraints} cancels the last term with the
summation over $\rho$. Summing the result over $\mu\ne\nu$ with $\nu$
fixed, gives $(n-\frac{1}{2})$
times~\eqref{eq:ICZeroWeightConstraintsEps}, which is then no longer an
independent equation. As above, using the simplified equation in the
original and index swapped versions gives
\[
	\begin{pmatrix} x_\mu & x_\nu \\ x_\nu & x_\mu \end{pmatrix}
	\begin{pmatrix} \Omega_{\mu\nu} \\ \alpha_{\mu\nu} \end{pmatrix}
	= \begin{pmatrix}
			-d_\mu \frac{x_\nu}{x_\mu} \\
			-d_\nu \frac{x_\mu}{x_\nu}
		\end{pmatrix}
	\quad \iff \quad
	\begin{pmatrix} \Omega_{\mu\nu} \\ \alpha_{\mu\nu} \end{pmatrix}
	= \begin{pmatrix}
		\frac{x_\mu d_\mu - x_\nu d_\nu}{x^2_\mu - x^2_\nu}
		- \left(\frac{d_\mu}{x_\mu} + \frac{d_\nu}{x_\nu}\right) \\
		\frac{x_\mu d_\nu - x_\nu d_\mu}{x^2_\mu - x^2_\nu}
	\end{pmatrix} .
\]
There are no more independent equations, so the $d_\mu$ are again
unconstrained and parametrize the general solution
of~\eqref{eq:ICZeroWeightConstraints}
and~\eqref{eq:ICZeroWeightConstraintsEps}.

While the origin of the $d_\mu$ parameters in even and odd
dimensions is somewhat different, the particular choices we made
above allow the final form of the solutions to have a desirable uniform
structure, which is the main
conclusion~\eqref{eq:ExplicitCoefficientsIntCon} of this theorem.
\end{proof}

\begin{remark}
The proof works for every $n > 1$, however, the original system of equations \eqref{eq:ICZeroWeightConstraints}, \eqref{eq:ICZeroWeightConstraintsEps} is identically satisfied for $n=2$ and doesn't put any restrictions on the parameters. The number of free parameters coincides, thus we can extend the formulas \eqref{eq:ExplicitCoefficientsIntCon} also to $n=2$ as a reparametrization. 
\end{remark}

We can immediately notice that the form of the off-diagonal coefficients $\alpha_{\mu \nu}, \Omega_{\mu \nu}$ and $\alpha_{\mu}^{(\epsilon)}$ \eqref{eq:ExplicitCoefficientsIntCon} is strikingly similar to functions \eqref{eq:BasisCoefsToHamFcions}, which we derived for the coefficients from the calculation of the Riemann tensor from the curvature two-forms \eqref{eq:HamamotoFormulas}. After a~short comparison we conclude.

\begin{proposition} \label{th:PreliminaryFormOfdmu}
The coefficients $\alpha_{\mu \nu}$, $\Omega_{\mu \nu}$, and $\alpha_{\mu}^{(\epsilon)}$ $(\mu\neq\nu)$ specified by formulas \eqref{eq:ExplicitCoefficientsIntCon} agree precisely with formulas \eqref{eq:BasisCoefsToHamFcions} when we choose
\begin{align} \label{eq:dAndPDQT}
d_\mu = 4 \pd_\mu \mathsf{M}.
\end{align}
\end{proposition}

It is now natural to ask, what further conditions on $R_{abcd}$ would force~\eqref{eq:dAndPDQT} to hold, so that it would exactly reproduce the Kerr-NUT-(A)dS curvature tensor. One covariant differential equation, that would be identically satisfied by an actual Riemann tensor, is right at hand, namely, the II. Bianchi identity \eqref{eq:IIBianchiIdentity}. Preliminary consideration suggests that it will give restrictions on the diagonal omegas $\Omega_{\mu \mu}$, because there is no special symmetry, which would eliminate these terms and it contains derivatives of the basis coefficients in \eqref{eq:RiemannInBasis}, that is derivatives of the parameters $d_\mu$. On the other hand, the II. Bianchi identity has many more independent components. So we may expect to obtain further independent differential constraints on the $d_\mu$ from it. Section \ref{sec:IIBianchiIdentity} is dedicated to verification of these hypotheses, where under some reasonable conditions we will in fact obtain differential constraints that force $d_\mu$ to take the desired form~\eqref{eq:dAndPDQT}. 

\section{II. Bianchi identity for the totally zero weight Riemann tensor} \label{sec:IIBianchiIdentity}

The Riemann tensor has by definition Young symmetry (22). It is reflected in mono-term symmetries related to exchange of its indices, while the multi-term symmetry is represented by the first Bianchi identity. In our construction of formula \eqref{eq:RiemannInBasis}, parametrizing the totally weight zero subspace of Riemann-symmetric tensors, all these properties are already incorporated. However, a true Riemann tensor $R_{abcd}$ has another symmetry property associated with (32) Young symmetry of its covariant derivative, which is captured in the \emph{second Bianchi identity} 
\begin{align} \label{eq:IIBianchiIdentity}
\nabla_{[e} R_{ab]cd} = 0.
\end{align}
As opposed to the first Bianchi identity, it is not guaranteed by our construction. In this section we investigate the consequences of~\eqref{eq:IIBianchiIdentity} for the totally zero weight RS tensor in the $h$-basis \eqref{eq:RiemannInBasis}.  

When we substitute \eqref{eq:RiemannInBasis} into \eqref{eq:IIBianchiIdentity} and start to substitute and evaluate all necessary relations, we soon discover than the number of terms grows rapidly with every substitution so that the calculation becomes unmanageable. In order to derive conditions the equation implies, we choose analogous procedure to the one used in section \ref{sec:GeneralizedIC} - we decompose it into sectors of products of basis vectors $e^\mu, \hat{e}^\mu, \hat{e}^0$ from which each one belongs to an irreducible space under the group $GL(n)\times GL(n)$. Further decompositions that respect only the smaller $U(1)^n \rtimes S_n$ group are not relevant in this case, and we use only decomposition into diagonals, section \ref{sec:DecompositionIntoDiagonals}.

\begin{remark} \label{rem:IIBianchiInterpretation}
We must make a very important remark about the meaning of equation~\eqref{eq:IIBianchiIdentity}, which highlights the main shortcoming of the results that will be obtained in this section. To correctly interpret all the formulas and constraints that will follow in this section from the second Bianchi identity~\eqref{eq:IIBianchiIdentity}, the reader should note that $R_{abcd}$ refers to an arbitrary 4-tensor with Riemann symmetries that is in the totally zero weight form~\eqref{eq:RiemannInBasis}, with respect to a Darboux basis on a specific Kerr-NUT-(A)dS spacetime, while $\nabla_e$ is the Levi-Civita covariant derivative of the same geometry. 
Although we will eventually conclude (theorem~\ref{the:IIBianchiIdentity} and theorem~\ref{th:SolutionOfdEquation}) that the resulting constraints do force $R_{abcd}$ to coincide with some Riemann tensor of the Kerr-NUT-(A)dS family (not necessarily coinciding with the background geometry).
From the point of view of an IDEAL characterization, this conclusion has the shortcoming that it had assumed certain features of the background geometry (namely the interaction of the covariant derivative with the Darboux basis) and so cannot be taken as a sufficient condition for a characterization.
A more comprehensive analysis, which we did not have time to carry out, could for instance interpret the covariant derivative in~\eqref{eq:IIBianchiIdentity} as a general torsion-free affine connection and see if some of the sectors of that equation (which in our calculations are automatically satisfied) would force this connection to coincide with the background Levi-Civita connection. In that case, the pairing of theorems \ref{the:IIBianchiIdentity} and \ref{th:SolutionOfdEquation} with theorems~\ref{th:IntegrabilityCondition} and~\ref{the:GeneralizedIntegrability} would constitute sufficient conditions for the characterization of the Riemann tensor of a Kerr-NUT-(A)dS geometry involving a 2-form $h_{ab}$ as the only background geometric field. If any explicit reference to $h_{ab}$ could be eliminated then our results would indeed constitute an IDEAL characterization of the Kerr-NUT-(A)dS family of spacetimes. Of course, such considerations should be part of future investigations.
\end{remark}

In section \ref{sec:GeneralizedIC} we introduced the new vectors \eqref{eq:BasisEpm}, which allowed us to do the decomposition into sectors according to weights and use the full potential of the local group of symmetry $U(1)^n \rtimes S_n$, see e.g.\ theorem \ref{the:ICZeroWeightTheorem}. However, here we operate only with the totally zero weight formula \eqref{eq:RiemannInBasis} and its covariant derivative, in which it would be disadvantageous to switch to $e^{\pm}$. Thus, we remain with $e^\mu, \hat{e}^\mu, \hat{e}^0$.

Young symmetry of the covariant derivative of the Riemann tensor according to the Pieri's rule~\cite[Eq.(6.8)]{Fulton1996} is
\begin{align*}
\nabla R \leftrightarrow \yngbig(1) \otimes \yngbig(2,2) \cong \yngbig(3,2) \oplus \yngbig(2,2,1) \, ,
\end{align*}
from which the second part (221) corresponds to the second Bianchi identity and vanishes. Therefore, we define  projectors onto subspace of this symmetry. As before, we have to take into account that the (co)tangent space is split into two-dimensional subspaces defined by vectors - $e$ and $\hat{e}$, to which we have to add $\hat{e}^0$ in odd dimensions. 

Analogously to formulas in section \ref{sec:RiemannWeightSubspaces} we can determine decomposition of the (221) diagram:
\begin{small}
\begin{align}
\mathrm{Res}[e, \hat{e}, \hat{e}^0](221)
	&\cong \mathrm{Res}[e, \hat{e}](221)
		\oplus \mathrm{Res}_0[e, \hat{e}, \hat{e}^0](221), \nonumber
	\\
	\mathrm{Res}[e, \hat{e}](221)
	&\cong [(221,0,0) \oplus (0,221,0) \oplus (22,1,0) \oplus (1,22,0) \nonumber
	\\ & \quad {}
		\oplus (211,1,0) \oplus (1,211,0) \oplus (111,11,0) \oplus (11,111,0) \nonumber
	\\ & \quad {}
		\oplus (21,11,0) \oplus (11,21,0) \oplus (21,2,0) \oplus (2,21,0)], \label{eq:YoungDecOf221}
	\\
\mathrm{Res}_0[e, \hat{e}, \hat{e}^0](221)
	&\cong [(2 2, 0,1) \oplus (0, 2 2,1) \oplus (2 1, 1,1) \oplus (1, 2 1,1) \oplus (1 1, 1 1,1) \oplus (2, 2,1)] \nonumber
	\\ &\quad {}
		\oplus [(0,211,1) \oplus (211,0,1)\oplus (111,1,1) \oplus (11,11,1) \nonumber
	\\ &\quad {}
		 \oplus (11,2,1) \oplus (2,11,1) \oplus (1,111,1) \oplus (1,21,1) \oplus (21,1,1)] \nonumber
	\\ &\quad {}
		\oplus [(21, 0,2) \oplus (0, 21,2) \oplus (11,1,2) \oplus (1,11,2) \oplus (2,1,2) \oplus (1,2,2)], \nonumber
\end{align}
\end{small}
where the first summand corresponds to those sectors present in all dimensions and the second one to the extra sectors present in odd dimensions, when $\eps=1$.

By the same logic as in section~\ref{sec:RiemannWeightSubspaces}, we define the isomorphisms 
\begin{equation} \label{eq:P-Bianchi}
	\mathfrak{P}^B \colon \mathfrak{T}^B
		= \bigoplus_{(\lambda)} \mathfrak{T}_{(\lambda)}
	\to \mathfrak{B} , \quad
	\bar{\mathfrak{P}}^B \colon \mathfrak{B} \to \mathfrak{T}^B ,
\end{equation}
where $\mathfrak{B}$ denotes the space of tensors of Bianchi symmery type (221) with respect to $GL(D)$ and $\mathfrak{T}_{(\lambda)}$ is the $GL(n)\times GL(n) \times GL(\eps)$ representation corresponding to the tuple of Young diagrams $(\lambda)$. We will use some of the components of $\mathfrak{P}^B$ to explicitly project onto the corresponding $\mathfrak{T}_{(\lambda)}$, but not all of them. Those that are needed are listed explicitly in the next subsections.
We will also not need a complete list of the analogs of the diagonal maps $\mathfrak{D}^{(\lambda)}$, so we do not define them here.

\subsection{The part without \texorpdfstring{$\hat{e}^0$}{e0}} \label{sec:CalculationsNoe0}

We initiate our study by examining the part of the second Bianchi identity, which does not contain $\hat{e}^0$ and is present in all dimensions.
Up to this point, we have made of the invariance of our calculations under the $U(1)^n \rtimes S_n$ stabilizer subgroup. Now for the first time, we will find that this stabilizer subgroup must be further reduced to $S_n$, as is evidenced by a broken $e \leftrightarrow \hat{e}$ interchange symmetry in the following formulas. The symmetry is broken by the use of the covariant derivative in~\eqref{eq:IIBianchiIdentity}, which as can be seen explicitly from proposition~\ref{lem:RotOneForms} is not symmetric under the $e \leftrightarrow \hat{e}$ interchange.

We calculated that the expansion of~\eqref{eq:IIBianchiIdentity} contains only the following types of combinations of the basis vectors, up to permutations:
\[
e \otimes \hat{e} \otimes \hat{e} \otimes \hat{e} \otimes \hat{e}, \quad
e \otimes e \otimes e \otimes \hat{e} \otimes \hat{e}, \quad
e \otimes e \otimes e \otimes e \otimes e,
\]
so from the decomposition \eqref{eq:YoungDecOf221} of $\mathrm{Res}[e, \hat{e}](221)$ the only relevant subspaces are
\begin{align*}
(221, 0) \oplus (1, 22) \oplus (1, 211) \oplus (111, 11) \oplus (21, 11) \oplus (21, 2).
\end{align*}
To write the projectors explicitly, let us extend the earlier notation for basis monomials to
\[
	E^{\tilde{\lambda}\tilde{\mu}\tilde{\nu}\tilde{\rho}\tilde{\sigma}}_{iabcd}
	= \tilde{e}^{\lambda}_i \tilde{e}^{\mu}_a \tilde{e}^{\nu}_b \tilde{e}^{\rho}_c \tilde{e}^{\sigma}_d,
\]
where an upper index slot like $\tilde{\mu}$ could now stand for $\mu$, $\hat{\mu}$, or $0$, meaning that the corresponding $\tilde{e}^{\mu}_a$ is replaced by $e^{\mu}_a$, $\hat{e}^{\mu}_a$ or $\hat{e}^0_a$, respectively.
Then
\begin{small}
\begin{align*} 
(P^{(221,0)})^{\mu\nu\rho\sigma\lambda}_{iabcd} &= P_{\,\young(ac,bd,i)} P_{\,\young(\nu\sigma,\rho\lambda,\mu)}\, E^{{\mu}{\nu}{\rho}{\sigma}{\lambda}}_{iabcd} , &
(P^{(1,22)})^{\mu\nu\rho\sigma\lambda}_{iabcd} &= P_{\,\young(ac,bd,i)} P_{\,\young(\nu\sigma,\rho\lambda)}\, E^{{\mu}{\hat{\nu}}{\hat{\rho}}{\hat{\sigma}{\hat{\lambda}}}}_{iabcd} , \nonumber \\
(P^{(1,211)})^{\mu\nu\rho\sigma\lambda}_{iabcd} &= P_{\,\young(ac,bd,i)} P_{\,\young(\nu\sigma,\rho,\lambda)}\, E^{{\hat{\lambda}}{\hat{\nu}}{\hat{\rho}}{\hat{\sigma}}{\mu}}_{iabcd} , &
(P^{(111,11)})^{\mu\nu\rho\sigma\lambda}_{iabcd} &= P_{\,\young(ac,bd,i)} P_{\,\young(\mu,\nu,\rho)} P_{\,\young(\sigma,\lambda)}\, E^{{\mu}{\nu}{\rho}{\hat{\sigma}{\hat{\lambda}}}}_{iabcd}, \\
\\
(P^{(21,11)})^{\mu\nu\rho\sigma\lambda}_{iabcd} &= P_{\,\young(ac,bd,i)} P_{\,\young(\mu\rho,\nu)} P_{\,\young(\sigma,\lambda)}\, E^{{\hat{\lambda}}{\mu}{\nu}{\rho}{\hat{\sigma}}}_{iabcd} , &
(P^{(21,2)})^{\mu\nu\rho\sigma\lambda}_{iabcd} &= P_{\,\young(ac,bd,i)} P_{\,\young(\mu\rho,\nu)} P_{\,\young(\sigma\lambda)}\, E^{{\hat{\lambda}}{\mu}{\nu}{\rho}{\hat{\sigma}}}_{iabcd} . \nonumber
\end{align*}
\end{small}
Since the basis $e^\mu, \hat{e}^\mu, \hat{e}^0$ is \emph{orthogonal}, we can take the reverse embedding maps $(E^{(\lambda)})^{\mu\nu\rho\sigma\lambda}_{iabcd}$ to have exactly the same tensor components as the projection maps $(P^{(\lambda)})^{\mu\nu\rho\sigma\lambda}_{iabcd}$.

\subsubsection*{Projected equations}

In order to simplify the projected equations, we introduce the following shorthand notation
\begin{align}
\begin{split} \label{eq:AuxiliaryTensors}
t_{\nu\rho\mu} &\equiv \tfrac{1}{2}M_\nu \partial_\nu \Omega_{\rho\mu} + \Omega_{\rho \nu} U_{\mu\nu} + \Omega_{\mu\rho} V_{\nu\rho}, \\
F_{\mu\nu\rho} &\equiv \Omega_{\mu\nu} U_{\mu\rho} + \Omega_{\nu \rho} V_{\rho\mu} + (\alpha_{\mu\nu}-\alpha_{\mu\rho}) U_{\nu\rho} + (\alpha_{\nu}^{(\epsilon)}-\alpha_{\rho}^{(\epsilon)}) U_{\nu\rho}, \\
T_{\mu\nu\rho} &\equiv  \tfrac{1}{2}M_{\mu}\partial_\mu \alpha_{\nu\rho} - \alpha_{\mu \rho} V_{\mu\nu} + \alpha_{\nu \rho} V_{\mu\rho} + M_{\mu}\partial_\mu \alpha_{\rho}^{(\epsilon)} + (\alpha_{\rho}^{(\epsilon)}-\alpha_{\mu}^{(\epsilon)}) V_{\mu\rho}, \\
f_{\rho\mu\nu} &\equiv \Omega_{\mu\rho} U_{\nu\rho} - \Omega_{\nu \rho} U_{\mu\rho} + \Omega_{\mu\nu} V_{\rho\nu} - \Omega_{\mu \nu} V_{\rho\mu}.
\end{split}
\end{align}
The following identities are useful in the simplifications 
\begin{align*}
f_{\rho\mu\nu} &= f_{\rho[\mu\nu]}, \qquad \Leftrightarrow \qquad f_{\rho(\mu\nu)} = 0, \\
2t_{\nu(\rho\mu)} &= M_\nu \partial_\nu \Omega_{\rho\mu} + \Omega_{\rho \nu} U_{\mu\nu} + \Omega_{\mu\rho} V_{\nu\rho} + \Omega_{\mu \nu} U_{\rho\nu} + \Omega_{\mu\rho} V_{\nu\mu}.
\end{align*}
Than the projections are
\allowdisplaybreaks
{\small
\begin{align}
P^{(111,11)}: \; 0 &= P_{\,\young(\mu,\nu,\rho)} P_{\,\young(\sigma,\lambda)} \, \big[ \delta_{\mu\sigma} \delta_{\lambda\rho} t_{\nu\rho\mu} \big]
	\notag \\
	& = \tfrac{1}{6}\big(t_{\rho(\mu\nu)}(\delta_{\mu\lambda} \delta_{\nu\sigma} -\delta_{\mu\sigma} \delta_{\lambda\nu}) + t_{\nu(\mu\rho)} (\delta_{\mu\sigma} \delta_{\lambda\rho}-\delta_{\mu\lambda} \delta_{\rho\sigma}) \nonumber \\ 
& \qquad + t_{\mu(\nu\rho)} (\delta_{\rho\sigma} \delta_{\lambda\nu} - \delta_{\rho\lambda} \delta_{\nu\sigma})\big), \label{eq:ProjectionsB2-1} \\
P^{(1,211)}: \; 0 &= P_{\,\young(\nu\sigma,\rho,\lambda)} \, \big[ \delta_{\mu\rho} \delta_{\lambda\sigma} F_{\lambda\mu\nu} \big]
	\notag \\
	& = \tfrac{1}{8}\big(F_{\sigma\mu\rho}(\delta_{\mu\lambda} \delta_{\nu\sigma} -\delta_{\lambda\sigma} \delta_{\mu\nu}) + F_{\sigma\mu\nu} (\delta_{\lambda\sigma} \delta_{\mu\rho}-\delta_{\mu\lambda} \delta_{\rho\sigma}) \nonumber \\ 
& \qquad + F_{\sigma\mu\lambda} (\delta_{\mu\nu} \delta_{\rho\sigma} - \delta_{\mu\rho} \delta_{\nu\sigma})\big), \label{eq:ProjectionsB2-2} \\
P^{(221,0)}: \; 0 &= P_{\,\young(\nu\sigma,\rho\lambda,\mu)} \, \big[ \delta_{\nu\sigma} \delta_{\lambda\rho} T_{\mu\nu\rho} \big]
	\notag \\
	& = \tfrac{1}{6}\big(T_{\mu(\nu\rho)}(\delta_{\lambda\rho} \delta_{\nu\sigma} -\delta_{\lambda\nu} \delta_{\rho\sigma}) + T_{\rho(\mu\nu)} (\delta_{\lambda\nu} \delta_{\mu\sigma}-\delta_{\mu\lambda} \delta_{\nu\sigma}) \nonumber \\ 
& \qquad + T_{\nu(\mu\rho)} (\delta_{\mu\lambda} \delta_{\rho\sigma} - \delta_{\lambda\rho} \delta_{\mu\sigma})\big), \label{eq:ProjectionsB2-3} \\
P^{(1,22)}: \; 0 &= P_{\,\young(\nu\sigma,\rho\lambda)} \, \big[ -\delta_{\nu\sigma} \delta_{\lambda\rho} T_{\mu\nu\rho} - 3\delta_{\mu\sigma} \delta_{\lambda\rho} F_{\rho\mu\nu}  \big]
	\notag \\
	& = \tfrac{1}{2}T_{\mu(\nu\rho)}(\delta_{\lambda\nu} \delta_{\rho\sigma} -\delta_{\nu\sigma} \delta_{\lambda\rho}) + \tfrac{3}{8}\Big( F_{\rho\mu\nu} (\delta_{\lambda\mu} \delta_{\rho\sigma}-\delta_{\mu\sigma} \delta_{\lambda\rho}) \nonumber \\ 
& \qquad + F_{\nu\mu\rho} (\delta_{\mu\sigma} \delta_{\lambda\nu} - \delta_{\lambda\mu} \delta_{\nu\sigma}) + F_{\lambda\mu\sigma} (\delta_{\mu\rho} \delta_{\lambda\nu} - \delta_{\mu\nu} \delta_{\lambda\rho}) \nonumber \\
& \qquad + F_{\sigma\mu\lambda} (\delta_{\mu\nu} \delta_{\rho\sigma} - \delta_{\mu\rho} \delta_{\nu\sigma}) \Big), \label{eq:ProjectionsB2-4} \\
P^{(21,11)}: \; 0 &= P_{\,\young(\mu\rho,\nu)} P_{\,\young(\sigma,\lambda)} \,\big[ 4\delta_{\mu\sigma} \delta_{\lambda\nu} t_{\rho(\mu\nu)} + 12\delta_{\mu\sigma} \delta_{\lambda\nu} f_{\rho\mu\nu} + 3\delta_{\mu\rho} \delta_{\lambda\nu} F_{\mu\nu\sigma} \big]
	\notag \\
	& = \tfrac{4}{3}t_{\rho(\mu\nu)}(\delta_{\lambda\nu} \delta_{\mu\sigma} -\delta_{\nu\sigma} \delta_{\lambda\mu}) + \tfrac{2}{3} (t_{\nu(\mu\rho)} + 3f_{\nu\rho\mu})(\delta_{\lambda\rho} \delta_{\mu\sigma}-\delta_{\rho\sigma} \delta_{\lambda\mu}) \nonumber \\ 
& \qquad + \tfrac{2}{3} (t_{\mu(\nu\rho)} + 3f_{\mu\rho\nu})(\delta_{\rho\sigma} \delta_{\lambda\nu} - \delta_{\lambda\rho} \delta_{\nu\sigma}) + F_{\rho\lambda\sigma} (\delta_{\lambda\nu} \delta_{\mu\rho} - \delta_{\lambda\mu} \delta_{\nu\rho}) \nonumber \\
& \qquad + F_{\rho\sigma\lambda} (\delta_{\mu\sigma} \delta_{\nu\rho} - \delta_{\nu\sigma} \delta_{\mu\rho}), \label{eq:ProjectionsB2-5} \\
P^{(21,2)}: \; 0 &= P_{\,\young(\mu\rho,\nu)} P_{\,\young(\sigma\lambda)} \,\big[ 12\delta_{\mu\sigma} \delta_{\lambda\nu} t_{\rho(\mu\nu)} + 4\delta_{\mu\sigma} \delta_{\lambda\nu} f_{\rho\mu\nu} + 3\delta_{\mu\rho} \delta_{\lambda\nu} F_{\mu\nu\sigma} - 2\delta_{\mu\nu} \delta_{\lambda\sigma} T_{\rho\mu\lambda}
	\notag \\
& \qquad \qquad - 2\delta_{\nu\rho} \delta_{\lambda\sigma} T_{\mu\lambda\nu} \big]
	\notag \\
	& = -\tfrac{2}{3}(3t_{\nu(\mu\rho)} + f_{\nu\rho\mu})(\delta_{\lambda\rho} \delta_{\mu\sigma} + \delta_{\lambda\mu} \delta_{\rho\sigma}) + \tfrac{2}{3} (3t_{\mu(\nu\rho)} + f_{\mu\rho\nu})(\delta_{\lambda\rho} \delta_{\nu\sigma} + \delta_{\lambda\nu} \delta_{\rho\sigma}) \nonumber \\ 
& \qquad + \tfrac{4}{3} f_{\rho[\mu\nu]}(\delta_{\mu\sigma} \delta_{\lambda\nu} + \delta_{\lambda\mu} \delta_{\nu\sigma}) + F_{\rho\lambda\sigma} (\delta_{\lambda\nu} \delta_{\mu\rho} - \delta_{\lambda\mu} \delta_{\nu\rho}) \nonumber \\
& \qquad + F_{\rho\sigma\lambda} (\delta_{\nu\sigma} \delta_{\mu\rho} - \delta_{\mu\sigma} \delta_{\nu\rho}) + \tfrac{4}{3}T_{\nu(\mu\lambda)}\delta_{\lambda\sigma}\delta_{\mu\rho} - \tfrac{4}{3}T_{\mu(\nu\lambda)}\delta_{\lambda\sigma}\delta_{\nu\rho}. \label{eq:ProjectionsB2-6}
\end{align}
}
\allowdisplaybreaks[0]

Because of the presence of may $\delta_{\mu\nu}$ factors, a general approach to solve this system is to break it into diagonals using the maps \eqref{eq:DiagonalMaps} as in section~\ref{sec:DecompositionIntoDiagonals} and simplify until a smaller set of equations is obtained, which imply the original system. Without going into much detail, the smaller system that we identified is, for $\mu \neq \nu \neq \rho$,
\begin{equation} \label{eq:ConstraintsBianchi}
\begin{aligned}
f_{\mu\nu\rho} &= 0, &
F_{\mu\nu\rho} &= 0, \\
2 T_{\mu(\mu\nu)} + 3 F_{\nu\mu\mu} &= 0, &
t_{\mu(\mu\nu)} + f_{\mu\mu\nu} + \tfrac{1}{2} F_{\mu\nu\mu} &= 0, \\
T_{\mu(\nu\rho)} &= 0, & 
t_{\mu(\nu\rho)} &= 0,  \\
3t_{\mu\nu\nu} - T_{\mu\nu\nu} &= 0. & &
\end{aligned}
\end{equation}
Explicit substitution of the conditions \eqref{eq:ConstraintsBianchi} into the original system \eqref{eq:ProjectionsB2-1}--\eqref{eq:ProjectionsB2-6} and subsequent simplification shows that it is identically satisfied.%
	\footnote{It is worth noting that the conditions~\eqref{eq:ConstraintsBianchi} are complicated by the fact that they apply only when $\mu \neq \nu \neq \rho$. The way to deal with that is to decompose each of the auxiliary tensors~\eqref{eq:AuxiliaryTensors} into diagonals as in~\eqref{eq:Decomposition3Tensor}, which decouples~\eqref{eq:ConstraintsBianchi} so that they can be used without the index restrictions.}

\subsubsection*{Investigation of the constraints}

So far we haven't used any special property of the coefficients $\alpha_{\mu\nu}, \Omega_{\mu\nu}$, only definition of the tensors \eqref{eq:AuxiliaryTensors}. To investigate the constraints \eqref{eq:ConstraintsBianchi} we use the special form of these coefficients following from the integrability condition \eqref{eq:ExplicitCoefficientsIntCon}. 

The first two equations of \eqref{eq:ConstraintsBianchi} for $F_{\mu\nu\rho}$ and $f_{\mu\nu\rho}$ after we substitute for $U, V$ from \eqref{eq:FunctionsUV} and for $\Omega_{\mu\nu}, \alpha_{\mu\nu}, \alpha^{(\epsilon)}_{\mu}$ from \eqref{eq:ExplicitCoefficientsIntCon} yield trivial identities ($\mu \neq \nu \neq \rho$)
\begin{align*}
F_{\mu\nu\rho} &= \Omega_{\mu\nu} U_{\mu\rho} + \Omega_{\nu \rho} V_{\rho\mu} + \big(\alpha_{\mu\nu} + \alpha_{\mu}^{(\epsilon)} + \alpha_{\nu}^{(\epsilon)} - (\alpha_{\mu\rho} + \alpha_{\mu}^{(\epsilon)} + \alpha_{\rho}^{(\epsilon)})\big) U_{\nu\rho} \equiv 0, \\
f_{\rho\mu\nu} &= \Omega_{\mu\rho} U_{\nu\rho} - \Omega_{\nu \rho} U_{\mu\rho} + \Omega_{\mu\nu} (V_{\rho\nu} - V_{\rho\mu}) \equiv 0.
\end{align*}
Let us now concentrate on the third equation. Substitution for the tensors $T$ and $F$ leads to
\begin{align*}
2 T_{\mu(\mu\nu)} + 3 F_{\nu\mu\mu} &= M_\mu \partial_\mu(\alpha_{\mu\nu} + \alpha^{(\epsilon)}_\mu + \alpha^{(\epsilon)}_\nu) + (\alpha_{\mu\nu} + \alpha^{(\epsilon)}_\nu + \alpha^{(\epsilon)}_\mu) V_{\mu\nu}  \\
& \qquad + 3\Omega_{\mu\mu} V_{\mu\nu} + 3\Omega_{\mu\nu} U_{\nu\mu}.
\end{align*}
The derivative in the first term explicitly reads
\begin{align*}
\partial_\mu(\alpha_{\mu\nu} + \alpha^{(\epsilon)}_\mu + \alpha^{(\epsilon)}_\nu) &= \frac{d_\mu + x_\mu\partial_\mu d_\mu }{x_\mu^2 - x_\nu^2} - 2\frac{x_\mu^2 d_\mu + x_\nu^2 d_\mu - 2x_\mu x_\nu d_\nu}{(x_\mu^2 - x_\nu^2)^2}.
\end{align*}	
After substitution of all necessary quantities and algebraic manipulations we get 
\begin{align*}
2 T_{\mu(\mu\nu)} + 3 F_{\nu\mu\mu} &=\frac{M_\mu x_\mu}{x^2_\mu - x^2_\nu} \left( 3\Omega_{\mu\mu} + \partial_\mu d_\mu - 2\varepsilon\frac{d_\mu}{x_\mu}\right). 
\end{align*}
This combination vanishes if and only if the expression in the brackets is zero. We get a relation between \emph{diagonal} $\Omega_{\mu\mu}$ and the derivative of $d$-parameters
\begin{align} \label{eq:DiagonalOmega}
\Omega_{\mu\mu} &= -\frac{1}{3}\partial_\mu d_\mu + \frac{2}{3}\varepsilon\frac{d_\mu}{x_\mu}.
\end{align}

Further, we consider the fourth equation, which can be recast into
\begin{align*}
2t_{\mu(\mu\nu)} + 2f_{\mu\mu\nu} + F_{\mu\nu\mu} &= M_\mu \partial_\mu \Omega_{\mu\nu} + 3\Omega_{\mu \mu} U_{\nu\mu} + 3\Omega_{\mu\nu} V_{\mu\nu} + (\alpha_{\mu \nu} + \alpha_{\nu}^{(\epsilon)} - \alpha_{\mu}^{(\epsilon)})\big) U_{\nu\mu} = 0.
\end{align*}
The derivative of $\Omega_{\mu\nu}$ can be written as follows
\begin{align} \label{eq:DmuOmegaMuNU}
\partial_\mu \Omega_{\mu\nu} &= -2\frac{x_\mu}{x_\mu^2 - x_\nu^2}\Omega_{\mu\nu} + \frac{d_\nu + x_\mu\partial_\mu d_\nu  - x_\nu\partial_\mu d_\mu }{x_\mu^2 - x_\nu^2}.
\end{align}
With the help of this expression, the equation is simplified to
\begin{align*}
2t_{\mu(\mu\nu)} + 2f_{\mu\mu\nu} + F_{\mu\nu\mu} &= \frac{M_\mu x_\mu}{x_\mu^2 - x_\nu^2}\big( \partial_\mu d_\nu + 2\Omega_{\mu\nu} \big).
\end{align*}
We obtained a non-trivial relation between \emph{off-diagonal partial derivatives} of $d$s and parameters $\Omega$
\begin{align} \label{eq:dEquationOrig}
 \partial_\mu d_\nu = -  2\Omega_{\mu\nu}, \qquad \forall \mu \neq \nu.
\end{align}
When we further substitute for $\Omega_{\mu\nu}$ we get a \emph{differential equation} for $d_\mu$
\begin{align} \label{eq:dEquationFullForm}
\partial_\mu d_\nu &= -  2\frac{x_\mu d_\nu - x_\nu d_\mu}{x^2_\mu - x^2_\nu}, \qquad \forall \mu \neq \nu.
\end{align}
The rest of the equations \eqref{eq:ConstraintsBianchi} does not produce any new relations.

\subsection{The part with \texorpdfstring{$\hat{e}^0$}{e0}} \label{sec:CalculationsWithe0}

We continue to explore the part included only in odd dimensions, which contains $\hat{e}^0$. We empirically found that the only terms in the second Bianchi identity that contain $\hat{e}^0$ are of the following types, up to permutation (but as explained in the previous section, not symmetric under $e \leftrightarrow \hat{e}$ interchange):
\[
e \otimes e \otimes e \otimes \hat{e} \otimes \hat{e}^0, \quad
e \otimes \hat{e} \otimes \hat{e} \otimes \hat{e} \otimes \hat{e}^0, \quad
e \otimes \hat{e} \otimes \hat{e} \otimes \hat{e}^0 \otimes \hat{e}^0, \quad
e \otimes e \otimes e \otimes \hat{e}^0 \otimes \hat{e}^0.
\]
Now, the relevant subspaces from the decomposition \eqref{eq:YoungDecOf221} of $\mathrm{Res}_0[e, \hat{e}, \hat{e}^0](221)$ are:
{\small
\begin{align*}
2\times (21,1,1) \oplus (111,1,1) \oplus 2\times (1,21,1) \oplus (1,111,1) \oplus (1,11,2) \oplus (1,2,2) \oplus (21,0,2).
\end{align*}}
The projectors are then
\begin{align*}
(P^{(21,1,1)})^{\mu\nu\rho\sigma}_{iabcd} &= P_{\,\young(ac,bd,i)} P_{\,\young(\mu\nu,\rho)}\, E_{iabcd}^{\hat{\sigma}\mu\rho\nu 0} , &
(P^{(1,21,1)})^{\mu\nu\rho\sigma}_{iabcd} &= P_{\,\young(ac,bd,i)} P_{\,\young(\mu\nu,\rho)} \, E_{iabcd}^{0\hat{\mu}\hat{\rho}\hat{\nu}\sigma}, \\
(P^{(21,1,1)}_x)^{\mu\nu\rho\sigma}_{iabcd} &= P_{\,\young(ac,bd,i)} P_{\,\young(\mu\nu,\rho)} \, E_{iabcd}^{0\mu\rho\nu\hat{\sigma}}, & 
(P^{(1,21,1)}_x)^{\mu\nu\rho\sigma}_{iabcd} &= P_{\,\young(ac,bd,i)} P_{\,\young(\mu\nu,\rho)} \, E_{iabcd}^{\sigma\hat{\mu}\hat{\rho}\hat{\nu} 0}, \\
(P^{(111,1,1)})^{\mu\nu\rho\sigma}_{iabcd} &= P_{\,\young(ac,bd,i)} P_{\,\young(\mu,\nu,\rho)} \, E_{iabcd}^{\rho\mu\nu\hat{\sigma} 0}, &
(P^{(1,111,1)})^{\mu\nu\rho\sigma}_{iabcd} &= P_{\,\young(ac,bd,i)} P_{\,\young(\mu,\nu,\rho)} \, E_{iabcd}^{\hat{\rho}\hat{\mu}\hat{\nu}\sigma 0}, \\
(P^{(1,11,2)})^{\mu\nu\rho}_{iabcd} &= P_{\,\young(ac,bd,i)} P_{\,\young(\nu,\rho)} \, E_{iabcd}^{\hat{\rho}\mu 0\hat{\nu} 0} , &
(P^{1,2,2})^{\mu\nu\rho}_{iabcd} &= P_{\,\young(ac,bd,i)} P_{\,\young(\nu\rho)} \, E_{iabcd}^{\mu\hat{\nu} 0\hat{\rho} 0} , \\
(P^{(21,0,2)})^{\mu\nu\rho}_{iabcd} &= P_{\,\young(ac,bd,i)} P_{\,\young(\mu\nu,\rho)}\, E_{iabcd}^{\mu 0\rho 0\nu} . & &
\end{align*}

The equations that emerge from projections of the second Bianchi identity are much shorter and simpler than in the previous case without $\hat{e}^0$. In our specific case, projections $P^{(1,111,1)}, P^{(111,1,1)}$ yield zero and not all of the remaining seven projections are independent. The four independent equations can be written in the manner outlined below
\begin{align*}
P^{(21,1,1)}: 0 &= (\Omega_{\mu\rho} s_\mu + \alpha_{\mu\rho} s_\rho + \alpha^{(\epsilon)}_\rho s_\rho )\delta_{\mu\nu}\delta_{\rho\sigma} - (\mu \leftrightarrow \rho), \\
P^{(1,11,2)}: 0 &= \left( \frac{M_\nu}{x_\nu}\Omega_{\mu\nu} + (\alpha^{(\epsilon)}_\mu - \alpha^{(\epsilon)}_\nu) U_{\mu\nu} \right)\delta_{\mu\rho} - (\nu \leftrightarrow \rho), \\
P^{(1,2,2)}: 0 &= \frac{1}{3}\left(-M_\mu \partial_\mu \alpha^{(\epsilon)}_\nu + \alpha^{(\epsilon)}_\mu \frac{M_\mu}{x_\mu} + (\alpha^{(\epsilon)}_\mu - \alpha^{(\epsilon)}_\nu) V_{\mu\nu} + \alpha_{\mu\nu}\frac{M_\mu}{x_\mu} \right)\delta_{\nu\rho} \\
& \quad + \frac{1}{2}\left( \frac{M_\nu}{x_\nu}\Omega_{\mu\nu} + (\alpha^{(\epsilon)}_\mu - \alpha^{(\epsilon)}_\nu) U_{\mu\nu} \right)\delta_{\mu\rho} 
- \frac{1}{2}\left( \frac{M_\rho}{x_\rho}\Omega_{\mu\rho} + (\alpha^{(\epsilon)}_\mu - \alpha^{(\epsilon)}_\rho) U_{\mu\rho} \right)\delta_{\mu\nu}, \\
P^{(21,0,2)}: 0&= \left(-M_\mu \partial_\mu \alpha^{(\epsilon)}_\nu + \alpha^{(\epsilon)}_\mu \frac{M_\mu}{x_\mu} + (\alpha^{(\epsilon)}_\mu - \alpha^{(\epsilon)}_\nu) V_{\mu\nu} + \alpha_{\mu\nu}\frac{M_\mu}{x_\mu}\right)\delta_{\nu\rho} -  (\mu \leftrightarrow \rho), 
\end{align*}
where $(\mu \leftrightarrow \rho)$ denotes the same expression with swapped indices. Since these equations are much simpler than the previously investigated ones, there is no need for introduction of auxiliary tensors and we can approach their solution directly. However, when we use the obtained constraints \eqref{eq:DiagonalOmega} and \eqref{eq:dEquationOrig} the projections become trivial. Thus, they impose \emph{no other} restrictions. In odd dimensions the second Bianchi identity does not entail additional constraints.

\subsection{Conclusion of the calculation}

Let us summarize the results of sections \ref{sec:CalculationsNoe0} and \ref{sec:CalculationsWithe0} in a theorem. 

\begin{theorem}[II. Bianchi identity] \label{the:IIBianchiIdentity}
Consider a Kerr-NUT-(A)dS
spacetime~\eqref{eq:HigherDimensionalKNAdSmetric}, so that its Levi-Civita
connection is given by proposition~\eqref{lem:RotOneForms}.
Let $\mathcal{R}_{abcd}$ be a Riemann-symmetric tensor of totally zero weight with respect to the stabilizer group $S_n\rtimes U(1)^n$, parametrized as in~\eqref{eq:RiemannInBasis}, with the coefficients $\Omega_{\mu\nu}, \alpha_{\mu\nu}, \alpha^{(\epsilon)}_\mu$ parametrized by a vector $d_\mu$ as in \eqref{eq:ExplicitCoefficientsIntCon}.
Then the second Bianchi identity $\nabla_{[e} \mathcal{R}_{ab]cd} = 0$ implies the following additional constraints ($\forall \mu \neq \nu$):
\begin{align*}
\Omega_{\mu\mu} &= -\frac{1}{3}\partial_\mu d_\mu + \frac{2}{3}\varepsilon\frac{d_\mu}{x_\mu}, \tag{\ref*{eq:DiagonalOmega}} \\
\partial_\mu d_\nu &= -  2\frac{x_\mu d_\nu - x_\nu d_\mu}{x^2_\mu - x^2_\nu}. \tag{\ref*{eq:dEquationFullForm}}
\end{align*}
\end{theorem}

After we derived solution for off-diagonal coefficients $\Omega_{\mu\nu}, \alpha_{\mu\nu}$ and $\alpha^{(\epsilon)}_\mu$ from the integrability condition (theorem \ref{th:IntegrabilityCondition}), the question remained what further constraints which might restrict the Riemann tensor \eqref{eq:RiemannInBasis} to that of only Kerr-NUT-(A)dS form. The new differential constraint on the the vector of parameters $d_\mu$, namely~\eqref{eq:dEquationFullForm} does that job. Since the Kerr-NUT-(A)dS true Riemann tensor obviously satisfies the II. Bianchi identity, we know that the values~\eqref{eq:dAndPDQT} will satisfy~\eqref{eq:dEquationFullForm}. We will now verify that explicitly and moreover show that~\eqref{eq:dEquationFullForm} has no other solutions (theorem~\ref{th:SolutionOfdEquation}).

To summarize, the integrability condition together with the second Bianchi identity restricts the general Riemann tensor \eqref{eq:DecompositionOfRtoWeightSectors} directly to the Kerr-NUT-(A)dS form.

Equation~\eqref{eq:dEquationFullForm} does not fall into any class of PDEs with an established existence and uniqueness theory that we know of, with the exception of that considered by Darboux in a generalization of the integrability theorem of Frobenius (see~\cite{bjk2019} for details, where also Darboux's proof was extended to the $n$-dimensional case that we require). The earlier ad-hoc analysis in~\cite[section 2.8.3]{MatejovThesis} is consistent with Darboux's theory. We are spared applying this theory directly, because the equation can also be integrated exactly.

\newcommand{\del}{\partial}
\newcommand{\hdel}{\hat{\partial}}
\newcommand{\hpd}{\hat{\partial}}
\newcommand{\ka}{\kappa}
\newcommand{\la}{\lambda}

\begin{theorem} \label{th:SolutionOfdEquation}
Consider the equation~\eqref{eq:dEquationFullForm} for the vector of parameters $d_\mu$. Any function of the following form is a solution of~\eqref{eq:dEquationFullForm}:
\begin{equation} \label{eq:SolutionOfdEquation}
	d_\mu = \del_\mu D, \quad
	D = \sum_{\mu=1}^n \frac{X_\mu}{U_\mu} ,
\end{equation}
where $U_\mu$ is defined in~\eqref{eq:MetricFunctions} and $X_\mu = X_\mu(x_\mu)$ are arbitrary smooth functions of the individual indicated coordinates. Moreover, on any open, connected and simply-connected domain that avoids the partial diagonals (where $x_\mu=x_\nu$ for $\mu\ne\nu$) and the hyperplanes $x_\mu = 0$, every solution of~\eqref{eq:dEquationFullForm} must have the form~\eqref{eq:SolutionOfdEquation}.
\end{theorem}
First, it is not immediately obvious that the Kerr-NUT-(A)dS values~\eqref{eq:dAndPDQT} of $d_\mu$ are in the form~\eqref{eq:SolutionOfdEquation} in odd dimensions, because of the term proportional to $\varepsilon$ in~\eqref{eq:MainMetricFunction}. However, this discrepancy is easily remedied via the identity
\[
	\frac{1}{A^{(n)}} = \frac{1}{\prod_{\nu=1}^n x_\nu^2}
	= \sum_{\mu=1}^n \frac{(-1)^{n-1}}{x_\mu^2} \frac{1}{\prod_{\nu=1,\nu\ne\mu}^n (x_\nu^2-x_\mu^2)} ,
\]
where the last equality follows from evaluating the contour integral around the origin of the function $R(z) = \frac{1}{z\prod_{\mu=1}^n (z-x_\mu^2)}$ in two different ways: the single residue inside the contour (left-hand side), and all the residues outside the contour (right-hand side).

Next, we make some simplifications by introducing the coordinates $y_\mu = x_\mu^2$ and the derivatives $\hdel_\mu = \del/\del y_\mu = \frac{1}{2x_\mu}\del_\mu$. Equation~\eqref{eq:dEquationFullForm} then takes the form
\begin{equation} \label{eq:dequation2}
\hdel_\mu (d_\nu/x_\nu) = \frac{(d_\mu/x_\mu) - (d_\nu/x_\nu)}{y_\mu-y_\nu}, \qquad \forall \, \mu\ne\nu.
\end{equation}
An immediate consequence is that $\del_\mu (d_\nu/x_\nu) - \del_\nu (d_\mu/x_\mu) = 0$, or equivalently (by the Poincar\'e lemma) $d_\mu/x_\mu = \hdel_\mu D$ for some smooth function $D = D(y)$. Plugging this form into~\eqref{eq:dequation2} simplifies to
\begin{equation} \label{eq:dequation3}
	\hdel_\mu \hdel_\nu [(y_\mu - y_\nu) D] = 0 .
\end{equation}
Theorem~\ref{th:SolutionOfdEquation} is then a corollary of the following
\begin{lemma}
The formula
\begin{equation} \label{eq:SolutionOfdEquation3}
	D = \sum_{\mu=1}^n \frac{Y_\mu(y_\mu)}{\prod_{\nu=1,\nu\ne\mu}^n (y_\nu - y_\mu)} ,
\end{equation}
where $Y_\mu(y_\mu)$ are arbitrary smooth functions,
solves~\eqref{eq:dequation3}. Moreover, on any open, connected and
simply-connected domain that avoids the partial diagonals (where $y_\mu=y_\nu$
for $\mu\ne\nu$), every solution of~\eqref{eq:dequation3} must have the
form~\eqref{eq:SolutionOfdEquation3}.
\end{lemma}
The restrictions on the domain are there only to avoid problems with dividing by $(y_\mu-y_\nu)$ and with applying the Poincar\'e lemma.

\begin{proof}
To start, we check that~\eqref{eq:SolutionOfdEquation3} is indeed a solution. Observe that
\begin{multline*}
  (y_\mu - y_\nu) D
  = \sum_{\ka=1,\ka\ne\mu,\nu}^n \frac{Y_\ka}{\prod_{\la=1,\la\ne\ka,\mu,\nu}^n (y_\ka-y_\la)} \left(\frac{1}{y_\ka-y_\mu} - \frac{1}{y_\ka-y_\nu}\right)
  \\ {}
    + \frac{Y_\mu}{\prod_{\la\ne\mu,\nu} (y_\mu-y_\la)}
    - \frac{Y_\nu}{\prod_{\la\ne\mu,\nu} (y_\nu-y_\la)}
  = 0 ,
\end{multline*}
which is annihilated by $\hdel_\mu \hdel_\nu$ because it expands as a 
sum where each term is independent of either $y_\mu$ or $y_\nu$.

The rest of proof, that all solutions must be of the desired form, is by induction. The initial non-trivial case is $n=2$, where there is a single independent equation: from the fact that the smooth functions annihilated by $\hdel_1 \hdel_2$ are exactly of the form $Y_1(y_1) - Y_2(y_2)$ for some smooth functions $Y_1, Y_2$, it follows that
\[
	D = \frac{Y_1}{y_1-y_2} + \frac{Y_2}{y_2-y_1} .
\]

Now, assume that the result holds for all dimensions up to and including $n-1$, fix $\mu = n$ and for each $\nu\ne n$ integrate~\eqref{eq:dequation3} once:
\[
	\hdel_\nu [(y_n-y_\nu) D] = E_\nu ,
\]
where $E_\nu$ is smooth and may depend on all variables except $y_n$.
Next, multiply each resulting equation by factors that commute
with its derivative to get
\begin{equation} \label{eq:D-E}
	\hdel_\nu \left[D \prod_{\la=1,\la\ne n}^n (y_n - y_\la)\right]
	=  E_{\nu} \prod_{\la=1,\la\ne n,\nu}^n (y_n - y_\la) .
\end{equation}
Note that the right-hand side is polynomial in $y_n$ (of degree $n-2$) and, by the equality, exact with respect to the variables $y_\la$, $\la\ne n$. Hence, there exists a smooth potential $-E$ for the right-hand side, whose dependence on $y_n$ is also only polynomial (up to degree $n-2$). Hence~\eqref{eq:D-E} implies that
\begin{equation} \label{eq:D-E-sol}
	\hdel_\nu \left[D \prod_{\la=1,\la\ne n}^n (y_n - y_\la) + E\right] = 0
	\quad \implies \quad
	D = \frac{Y_n(y_n) - E}{\prod_{\la=1,\la\ne n}^n (y_n - y_\la)} .
\end{equation}
As a polynomial in $y_n$, the right-hand side of~\eqref{eq:D-E} each
coefficient must be individuall exact. At order $y_n^{n-2}$ this yields $E_\nu
= -\hdel_\nu \tilde{D}$ and plugging that back into~\eqref{eq:D-E}, after applying $\hdel_\ka$ and antisymmetrizing over $\ka, \nu$, gives the integrability condition
\[
	\hdel_\ka \hdel_\nu [(y_\ka - y_\nu) \tilde{D}] = 0 ,
	\quad \forall \mu,\nu \ne n .
\]
But by our inductive assumption in dimension $n-1$, the most general solution must be of the form~\eqref{eq:SolutionOfdEquation3} and direct calculation shows that the potential $E$ can be taken of the form
\[
	E = \sum_{\ka=1,\ka\ne n}^n
		\frac{Y_\ka \prod_{\la=1,\la\ne n,\ka}^n (y_n-y_\la)}
			{\prod_{\la=1,\la\ne n,\ka}^n (y_\ka-y_\la)} .
\]
Finally, plugging the above formula into~\eqref{eq:D-E-sol} and simplifying concludes the proof.
\end{proof}

\section{Discussion} \label{sec:discuss}

The results of this article may be regarded as concrete intermediate
steps towards an IDEAL characterization of the higher-dimensional
Kerr-NUT-(A)dS family. We have shown that the curvature of these
spacetimes lies in the ``totally zero weight'' subspace singled out by
the stabilizer group $U(1)^n\rtimes S_n$ of the principal tensor, where
it takes the compact form~\eqref{eq:RiemannInBasis} built from the
blocks $h_{\mu\,ab}$ of the principal tensor $h_{ab}$ and their squares
$Q_{\mu\,ab}$ (Theorem~\ref{the:RiemannInBasis}). We then proved a no-go
result (Theorem~\ref{th:NoTwoFormFromRiemann}): no non-trivial $2$-form,
including $h_{ab}$ itself, can be built covariantly from monomials of the
undifferentiated Riemann tensor. Finally, we characterized the
Riemann-symmetric tensors that could be the curvature of a
Kerr-NUT-(A)dS metric by two conditions: an \emph{algebraic} one, the
integrability condition~\eqref{eq:IntegrabilityConditionCKYE} of the
conformal Killing-Yano equation, which confines the curvature to a
finite-parameter family $r(d)_{abcd}$
(Theorems~\ref{the:GeneralizedIntegrability}
and~\ref{th:IntegrabilityCondition}), and a \emph{differential} one, the
second Bianchi identity, which fixes the remaining parameters $d_\mu$
(Theorems~\ref{the:IIBianchiIdentity} and~\ref{th:SolutionOfdEquation}).
We close by discussing what stands between these results and a full
IDEAL characterization, and by outlining the several directions in which
the gaps might be bridged.

\paragraph{Recovering $h_{ab}$ from derivatives of the curvature.}
The most direct route to a characterization would be to recover the
principal tensor covariantly from the curvature: by the strong
uniqueness properties of $h_{ab}$ (Proposition~\ref{th:Uniqueness}), such
a formula, substituted into the equations defining a principal tensor,
would already constitute an IDEAL characterization. Our no-go
Theorem~\ref{th:NoTwoFormFromRiemann} rules this out using the
undifferentiated Riemann tensor alone, so one is forced to admit its
covariant derivatives. For instance, the simplest formula one can think of is
\begin{align}\label{eq:TwoForm2CovDR}
\mathcal{H}_{ab} \equiv \nabla_j R_{ab}{}^{cd} \nabla_{i} R_{cdj}{}^{i}.
\end{align}
It is worth investigating whether $h_{ab}$ could be recovered algebraically from $\mathcal{H}_{ab}$.

We stress, however, that Theorem~\ref{th:NoTwoFormFromRiemann} only excludes an
\emph{explicit} covariant formula expressing $h_{ab}$ as a concomitant of the
undifferentiated curvature; it does not exclude that the algebraic structure of
$R_{abcd}$ fixes $h_{ab}$ through an \emph{implicit} IDEAL identity, so that
the principal tensor would be a conditional invariant of the curvature,
mirroring a step in the IDEAL characterization of Kerr~\cite{FerrandoSaez2009}.
The quadratic expression~\eqref{eq:RiemannInBasis} suggests $h_{ab}$ could be
considered a \emph{square root} of $R_{abcd}$. However, the dependence on
$h_{ab}$ is through its blocks $h_{\mu\, ab}$, with non-trivial coefficients, so
it is not obvious what IDEAL implicit equation would reproduce the same
dependence.

It is an open question how our Theorem~\ref{th:NoTwoFormFromRiemann} relates to
the Cartan-Karlhede method, which in $4$ dimensions generically requires
covariant derivatives of the curvature up to second order to fix the frame
invariantly, even for the Kerr geometry~\cite[Ch.9]{Stephani}. Whether it is a
manifestation of the same underlying necessity of curvature derivatives, or
whether the presence of the principal tensor in higher dimensions leaves room
for a characterization using fewer of them, we must leave for future
investigation.

\paragraph{Eliminating the reference to $h_{ab}$ via minors.}
An alternative to constructing $h_{ab}$ explicitly is to detect its mere
existence intrinsically. As noted in Remark~\ref{rem:TowardsIdealIC}, the
integrability condition~\eqref{eq:IntegrabilityConditionCKYE} can be read
as a linear homogeneous system $\mathcal{J}[h]=0$ for the unknown $2$-form
$h_{ab}$, with coefficients constructed covariantly from $R_{abcd}$. The
solvability of such a rectangular system---and hence the existence of a
principal tensor---can be expressed through the vanishing of appropriate
matrix minors, in analogy with the vanishing of the determinant of a
square system. These minors are covariant, albeit complicated,
expressions in $R_{abcd}$, and their vanishing would furnish genuine
IDEAL conditions that no longer refer to $h_{ab}$ as an external field.
Working out these minors, and comparing the resulting conditions with the
compact curvature form~\eqref{eq:RiemannInBasis}, would be a natural next step.

\paragraph{Removing the background-connection assumption.}
The principal remaining shortcoming was emphasized in
Remark~\ref{rem:IIBianchiInterpretation}. Our analysis of the second
Bianchi identity in Section~\ref{sec:IIBianchiIdentity} treats
$r(d)_{abcd}$ as an arbitrary Riemann-symmetric tensor solving the integrability condition,
but takes $\nabla_e$ to be the Levi-Civita connection of a
\emph{fixed} background Kerr-NUT-(A)dS geometry, expressed through its
Darboux frame (Proposition~\ref{lem:RotOneForms}). The conclusion that
$r(d)_{abcd}$ must be the curvature of a Kerr-NUT-(A)dS metric therefore
still presupposes features of the background, and so is not yet a
sufficient condition for a characterization.

A natural further step would be to replace $\nabla_e$ by a general
torsion-free affine connection and ask whether the sectors of the
Bianchi identity that are automatically satisfied in our calculation
would then \emph{force} that connection to coincide with the background
Levi-Civita connection. Should this be the case, the combination of
Theorems~\ref{the:GeneralizedIntegrability},
\ref{th:IntegrabilityCondition}, \ref{the:IIBianchiIdentity}
and~\ref{th:SolutionOfdEquation} would yield sufficient conditions for
the curvature of a Kerr-NUT-(A)dS geometry in which $h_{ab}$ is the only
external geometric field; combined with the elimination of $h_{ab}$
sketched above, this would complete the IDEAL characterization.

\paragraph{Computational aspects and reusability of the method.}
Independently of the characterization problem that motivated them, the
technical tools developed here may be of use in efficient symbolic
computations of invariants on higher-dimensional Kerr-NUT-(A)dS
geometries. For instance, the introduction of the
$h$-basis~\eqref{eq:hBasis} leads to concise formulas for powers of the
Riemann tensor, as noted in~\cite[Thm.3]{MatejovThesis}. The sparse
structure of the $h$-basis also allowed to clarify the role of certain
simple commutation relations between the powers of Riemann and principal
tensors~\cite[Eq.(5.81)]{FKK} for these
geometries~\cite[Thm.4]{MatejovThesis}. It would be interesting to
explore further applications.

\section*{Acknowledgments}

The authors thank Alfonso Garc{\'i}a-Parrado, Pavel Krtou{\v{s}},
D{\'a}vid Kubiz{\v{n}}{\'a}k, and Joan J. Ferrando for discussions and
helpful comments. 
IK was partially supported by the Czech Science Foundation
(project GA25-15544S) and the Czech Academy of Sciences (Research Plan
RVO: 67985840). 
DM was partially supported by the Charles University Grant Agency (project GAUK No. 332622).

\printbibliography

\end{document}